\documentclass[a4paper,reqno,10pt]{amsart}
\usepackage{eurosym}
\usepackage{amsfonts}
\usepackage{amsmath}
\usepackage{amssymb}
\usepackage{bbm}
\usepackage{amscd}
\usepackage[left=2cm,right=2cm,bottom=3cm,top=3cm]{geometry}
\usepackage[hidelinks]{hyperref}

\newtheorem{theorem}{Theorem}[section]
\newtheorem*{theorem*}{Theorem}

\newtheorem{corollary}[theorem]{Corollary}

\newtheorem{definition}[theorem]{Definition}

\newtheorem{lemma}[theorem]{Lemma}

\newtheorem{proposition}[theorem]{Proposition}
\newtheorem{remark}[theorem]{Remark}

\theoremstyle{definition}
\newtheorem{fact}[theorem]{Fact}

\begin{document}
\def\cprime{$'$}
\def\cprime{$'$}

\title[Fra\"{\i}ss\'{e} limits in functional analysis]{Fra\"{\i}ss\'{e}
limits in functional analysis}
\author{Martino Lupini}
\address{Mathematics Department\\
California Institute of Technology\\
1200 E. California Blvd\\
MC 253-37\\
Pasadena, CA 91125}
\email{lupini@caltech.edu}
\urladdr{http://www.lupini.org/}
\thanks{The author was supported by the York University Susan Mann
Dissertation Scholarship, and by the ERC Starting grant no.\ 259527 of
Goulnara Arzhantseva. Part of this work was done while the author was
visiting the Instituto de Ciencias Matem\'{a}ticas in Madrid. The
hospitality of the Institute is gratefully acknowledged.}
\dedicatory{}

\begin{abstract}
We provide a unified approach to Fra\"{\i}ss\'{e} limits in functional
analysis, including the Gurarij space, the Poulsen simplex, and their
noncommutative analogs. We obtain in this general framework many known and
new results about the Gurarij space and the Poulsen simplex, and at the same
time establish their noncommutative analogs. Particularly, we construct
noncommutative analogs of universal operators in the sense of Rota.
\end{abstract}

\maketitle

\section{ Introduction}

Classical Fra\"{\i}ss\'{e} theory studies \emph{countable homogeneous
structures}. A countable structure is homogeneous if any partial isomorphism
between two finitely generated substructures extends to an automorphism of
the whole structure. The foundational result of \emph{Fra\"{\i}ss\'{e} theory%
}, obtained by Fra\"{\i}ss\'{e} in \cite{fraisse_lextension_1954}, implies
that a countable homogeneous structure is completely determined by its \emph{%
age}. (The age\emph{\ }of a countable structure if the collection of all its
finitely generated substructures.) The classes of finitely generated
structures that arise as ages of countable homogeneous structures are now
called Fra\"{\i}ss\'{e} classes \cite{fraisse_lextension_1954}. The last
fifteen years have seen a renewed interest in countable homogeneous
structures and Fra\"{\i}ss\'{e} theory in view of the relations with\emph{\
Ramsey theory} and\emph{\ topological dynamics}. Indeed it was established
in \cite{kechris_fraisse_2005} that age of a countable homogeneous structure 
$S$ satisfies the Ramsey property if and only if the automorphism group of $%
S $ is \emph{extremely amenable}. (A topological group is extremely amenable
if any continuous action on a compact Hausdorff space has a fixed point.)
This fact, known as Kechris-Pestov-Todorcevic (KPT) correspondence,
initiated a new direction of research, a survey of which can be found in 
\cite{van_the_survey_2015}. One of the goals of this line of research is to
prove by combinatorial methods extreme amenability of interesting Polish
groups and, more generally, to compute their universal minimal compact
spaces.

The main ingredient in Fra\"{\i}ss\'{e}'s analysis is the back-and-forth
method. This technique consists in building an isomorphisms between two
limit structures by recursively defining approximations of it on finitely
generated substructures. The same basic idea is used in many arguments in
functional analysis and operator algebras, where it is more often called 
\emph{approximate intertwining}. This is not a coincidence. Many structures
in functional analysis have been recently recognized to be of Fra\"{\i}ss%
\'{e}-theoretic nature, due to works of Ben Yaacov \cite%
{ben_yaacov_fraisse_2015,ben_yaacov_linear_2014}, Ben Yaacov and Henson \cite%
{ben_yaacov_generic_2016}, Garbuli\'{n}ska-W\c{e}grzyn and Kubi\'{s} \cite%
{garbulinska_universal_2015}, Kubi\'{s} \cite%
{kubis_fraisse_2014,kubis_injective_2015,kubis_metric-enriched_2012}, Kubi%
\'{s} and Kwiatkowska \cite{kubis_lelek_2015}, Kubi\'{s} and Solecki \cite%
{kubis_proof_2013}, and unpublished work of Conley and T\"{o}rnquist. This
motivated Ben Yaacov \cite{ben_yaacov_fraisse_2015} to generalize Fra\"{\i}ss%
\'{e} theory from the discrete to the metric setting. In this framework he
established a correspondence between \emph{metric }Fra\"{\i}ss\'{e} classes
and separable metric structures (their \emph{limits}) that are \emph{%
approximately homogeneous}, in the sense that a partial isomorphism between
finitely generated substructures can be arbitrarily well approximated by an
automorphism. Other approaches to Fra\"{\i}ss\'{e} theory in the metric
setting have been suggested in \cite%
{schoretsanitis_fraisse_2007,kubis_fraisse_2014}. Fra\"{\i}ss\'{e} classes
arising in the theory of operator algebras have been studied in \cite%
{eagle_fraisse_2014}.

The aim of this paper is to provide a unified approach to the proof of the
fundamental properties of Fra\"{\i}ss\'{e} limits in functional analysis,
including the Gurarij space, the Poulsen simplex, and their noncommutative
analogs. The Gurarij space $\mathbb{G}$, first constructed by Gurarij in 
\cite{gurarij_spaces_1966}, is the unique separable approximately
ultrahomogeneous Banach space that is universal for separable Banach spaces 
\cite{lusky_separable_1977}. The Poulsen simplex $\mathbb{P}$, first
constructed by Poulsen in \cite{poulsen_simplex_1961} is the unique
nontrivial metrizable Choquet simplex with dense extreme boundary \cite%
{olsen_simplices_1980}.

It is known since the work of Lusky \cite%
{lusky_gurarij_1976,lusky_note_1978,lusky_primariness_1980,lusky_construction_1979,lusky_note_1978,lusky_separable_1977-1}
and Lindenstrauss-Olsen-Sternfeld \cite%
{lindenstrauss_poulsen_1978,olsen_simplices_1980} in the 1970s that the
Poulsen simplex and the Gurarij space can be studied by very similar
methods. This has been made precise in an unpublished work of Conley and T%
\"{o}rnquist, who studied the Poulsen simplex by looking a the associated 
\emph{function system}. A function system $V$ is a closed subspace of a real
Banach space of the form $C(K)$ containing the function constantly equal to $%
1$ (the \emph{unit}). The inclusion $V\subset C(K)$ defines on $V$ an order
structure which only depends on the norm and the unit of $V$. If $K$ is any
compact convex set, then the space $A(K)$ of continuous affine functions on $%
K$ is a function system. In the statement of Theorem \ref{Theorem:A(P)} (5)
we consider a compact convex set endowed with the norm coming from the
inclusion $K\subset A(K)^{\ast }$.

Kadison's representation theorem \cite[Theorem\ II.1.8]{alfsen_compact_1971}
asserts that any function system $V$ is of the form $A(K)$, where $K$ is the
space of unital positive linear functionals of $V$. Furthermore the
assignment $K\mapsto A(K)$ is a contravariant equivalence of categories from
the category of compact convex sets and continuous affine maps to the
category of function systems and unital positive linear maps. (Metrizable)
Choquet simplices correspond to (separable) function systems that are
moreover Lindenstrauss spaces. Thus for most purposes one can work with
separable Lindenstrauss function systems rather than metrizable Choquet
simplices. Conley and T\"{o}rnquist showed that $A(\mathbb{P})$ is the
unique separable function system that is approximately homogeneous and
universal for separable function systems. Thus $A(\mathbb{P})$ has the same
properties as $\mathbb{G}$, but in the category of function systems rather
than Banach spaces. We will call $A(\mathbb{P})$ the \emph{Poulsen system}.

We list here some known and new facts about the Poulsen simplex $\mathbb{P}$
that follow from our general framework and will be proved in \S \ref%
{Subsection:function-system}:

\begin{theorem}
\label{Theorem:A(P)}Let $\mathbb{P}$ be a nontrivial metrizable Choquet
simplex with dense extreme boundary $\partial _{e}\mathbb{P}$.

\begin{enumerate}
\item The space $A(\mathbb{P})$ is the unique approximately homogeneous
separable function system that contains unital isometric copies of any
separable function system.

\item $\mathbb{P}$ is the unique nontrivial metrizable Choquet simplex with
dense extreme boundary \cite[Theorem 2.3]{lindenstrauss_poulsen_1978}.

\item A metrizable compact convex set is a Choquet simplex if and only if it
is affinely homeomorphic to a closed proper face of $\mathbb{P}$ \cite[%
Theorem 2.5]{lindenstrauss_poulsen_1978}.

\item Any affine homeomorphism between two proper faces of $\mathbb{P}$
extends to an affine homeomorphism of $\mathbb{P}$ \cite[Theorem 2.3]%
{lindenstrauss_poulsen_1978}.

\item the set of norm-preserving continuous affine maps from a fixed Choquet
simplex $K$ to $\mathbb{P}$ with the property that the range is a closed
proper face of $\mathbb{P}$ is a dense $G_{\delta }$ subspace of the space
of continuous affine maps from $K$ to $\mathbb{P}$.

\item If $F$ is any closed proper face of $\mathbb{P}$ affinely homeomorphic
to $\mathbb{P}$ and $\phi :K_{0}\rightarrow K_{1}$ is a continuous affine
map between compact convex sets, then there exist continuous affine
surjections $\eta _{0}:F\rightarrow K_{0}$ and $\eta _{1}:\mathbb{P}%
\rightarrow K_{1}$ such that $\phi \circ \eta _{0}=\eta _{1}|_{F}$;

\item A homeomorphism between compact subsets of $\partial _{e}\mathbb{P}$
extends to an affine homeomorphisms of $\mathbb{P}$ \cite[Theorem 2.5]%
{lindenstrauss_poulsen_1978}.

\item Suppose that $F_{0},F_{1}$ are closed proper faces of $\mathbb{P}$.
Consider the complementary faces $F_{0}^{\prime }$ endowed with the compact
topology induced by the functions $a\in A(\mathbb{P})$ such that $a|_{F_{0}}$
is constant, and similarly for $F_{1}^{\prime }$. Then $F_{0}^{\prime }$ and 
$F_{1}^{\prime }$ are affinely homeomorphic \cite[Theorem 2.6]%
{lindenstrauss_poulsen_1978}.

\item The canonical action of $\mathrm{Aut}(\mathbb{P})$ on $\mathbb{P}$ is
minimal \cite[Theorem 5.2]{glasner_proximal_1976}.
\end{enumerate}
\end{theorem}

One can equivalently phrase (6) by asserting that if $F$ is any closed
proper face of $\mathbb{P}$, then any unital positive linear map between
separable function systems is a restriction-truncation to some subsystems of 
$A(\mathbb{P})$ of the unital quotient mapping $\Omega _{A(\mathbb{P})}:A(%
\mathbb{P})\rightarrow A(F)$, $f\mapsto f|_{F}$. This can be seen as a
function system version of universal operator in the sense of Rota \cite%
{rota_models_1960}.

We will prove below that one can associate to the Gurarij space $\mathbb{G}$
a geometric object with entirely analogous properties as $\mathbb{P}$. By a 
\emph{compact absolutely convex set }we mean a compact subset $K$ of a
locally convex topological vector space $V$ with that is closed under
absolutely convex combinations $\left( x,y\right) \mapsto \lambda x+\mu y$
when $\left\vert \lambda \right\vert +\left\vert \mu \right\vert \leq 1$. A
w*-continuous function between compact absolutely convex sets is \emph{%
symmetric} if it preserves the involution. One can associated to a compact
absolutely convex set $K$ the Banach space $A_{\sigma }(K)$ of real-valued
symmetric affine functions on $K$. Conversely any Banach space $X$ arises in
this way from the compact absolutely convex set $\mathrm{Ball}(X^{\ast })$ 
\cite[Lemma 1]{lazar_unit_1972}. Here $\mathrm{\mathrm{\mathrm{Ball}}}%
(X^{\ast })$ denotes the unit ball of the dual space $X^{\ast }$ of the
space $X$. Furthermore the assignment $K\mapsto A_{\sigma }(K)$ is a
contravariant equivalence of categories from the category of compact
absolutely convex sets and continuous symmetric affine maps to the category
of Banach spaces and linear maps of norm at most $1$. The compact absolutely
convex sets of the form $\mathrm{Ball}(X^{\ast })$ for some Lindenstrauss
space $X$ have been characterized by Lazar in \cite{lazar_unit_1972}; see
also \cite[Theorem 3.2]{effros_class_1971}. We will call these compact
absolutely convex sets \emph{Lazar simplices}. The Lazar simplex $\mathrm{%
Ball}(\mathbb{G})$ corresponding to the Gurarij space will be denoted by $%
\mathbb{L}$ and called the \emph{Lusky simplex}. We will prove that $\mathbb{%
L}$ plays the same role among Lazar simplices as $\mathbb{P}$ plays among
Choquet simplices. The analog of a face in this setting is the absolutely
convex hull of a face, called a \emph{biface }in \cite{effros_class_1971}
and a \emph{facial section }in \cite{lazar_banach_1971}. In the statement of
Theorem \ref{Theorem:G} we regard a compact absolutely convex set $K$ as
endowed with the norm coming from the inclusion $K\subset A_{\sigma }(K)$.

It will follow from our general results---see \S \ref{Subsection:Banach}%
---that the analogous statements hold for the Lusky simplex when one
replaces Choquet simplices with Lazar simplices and faces with bifaces:

\begin{theorem}
\label{Theorem:G}Let $\mathbb{L}$ be a Lazar simplex with dense extreme
boundary, and set $\mathbb{G}:=A_{0}\left( \mathbb{L}\right) $.

\begin{enumerate}
\item $\mathbb{G}$ is the unique approximately homogeneous separable real
Banach space that contains an isometric copy of any separable Banach space 
\cite{kubis_proof_2013,ben_yaacov_fraisse_2015}.

\item $\mathbb{L}$ is the unique Lazar simplex with dense extreme boundary 
\cite[Theorem 6.4]{lindenstrauss_poulsen_1978}.

\item A metrizable compact absolutely convex set is a Lazar simplex if and
only if it is symmetrically affinely homeomorphic to a closed proper biface
of $\mathbb{L}$ \cite[Corollary 4]{lusky_construction_1979}.

\item Any symmetric affine homeomorphism between closed proper bifaces of $%
\mathbb{L}$ extends to a symmetric affine homeomorphism of $\mathbb{L}$.

\item The set of norm-preserving continuous symmetric affine maps from a
fixed Lazar simplex $K$ to $\mathbb{L}$ with the property that the range is
a closed proper biface of $\mathbb{L}$ is a dense $G_{\delta }$ subspace of
the space of continuous symmetric affine maps from $K$ to $\mathbb{L}$.

\item If $H$ is any closed proper biface of $\mathbb{L}$ symmetrically
affinely homeomorphic to $\mathbb{L}$ and $\phi :K_{0}\rightarrow K_{1}$ is
a continuous symmetric affine map between compact absolutely convex sets,
then there exist symmetric continuous affine surjections $\eta
_{0}:F\rightarrow K_{0}$ and $\eta _{1}:\mathbb{L}\rightarrow K_{1}$ such
that $\phi \circ \eta _{0}=\eta _{1}|_{F}$.

\item A symmetric homeomorphism between proper compact subsets of $\partial
_{e}\mathbb{L}$ extends to a symmetric affine homeomorphism of $\mathbb{L}$.

\item Suppose that $H$ is a closed biface of $\mathbb{L}$. Consider the
complementary biface $H^{\prime }$ endowed with the w*-topology induced by
the functions $a\in A_{\sigma }\left( \mathbb{L}\right) $ such that $%
a|_{H}\equiv 0$. Then $H^{\prime }$ is affinely homeomorphic to $\mathbb{L}$.
\end{enumerate}
\end{theorem}

One can make (6) more precise, and assert that for any closed proper biface $%
H$ of $\mathbb{L}$ symmetrically affinely homeomorphic to $\mathbb{L}$ the
map $\Omega _{\mathbb{G}}:A_{\sigma }\left( \mathbb{L}\right) \rightarrow
A_{\sigma }(H)$, $f\mapsto f|_{H}$ is (conjugate to) the universal
nonexpansive operator on the Gurarij space constructed by Garbuli\'{n}%
ska-Wegryn and Kubi\'{s} in \cite{garbulinska_universal_2015}.

In \S \ref{Subsection:Banach-complex} we also obtain the natural analogs of
(1)--(6) above for \emph{complex }Banach spaces. If $X$ is a complex Banach
space, then we regard the unit ball $\mathrm{Ball}(X^{\ast })$ of the dual
space of $X$ as a compact circled convex set with a distinguished action of $%
\mathbb{T}$ given by $\left( \lambda ,x\right) \mapsto \lambda x$. A
w*-continuous affine map between compact circled convex sets is \emph{%
homogeneous} if it commutes with such an action.

A complex Banach space $X$ can be identified with the space $A_{\mathbb{T}%
}\left( \mathrm{Ball}(X^{\ast })\right) $ of complex-valued w*-continuous
homogeneous affine functions on $\mathrm{\mathrm{\mathrm{Ball}}}(X^{\ast })$%
. Furthermore, the map $K\mapsto A_{\mathbb{T}}(K)$ is a contravariant
equivalence of categories from the category of compact circled convex sets
and homogeneous continuous maps to the category of complex Banach spaces and
continuous linear maps of norm at most $1$. The class of compact circled
convex sets corresponding to complex Lindenstrauss spaces has been
characterized by Effros in \cite[Theorem 4.3]{effros_class_1974}. We will
refer to them as \emph{Effros simplices}. The natural analog of a biface in
this setting is the circled convex hull of a face (\emph{circled face}); see
Definition \ref{Definition:circled-face}. We will note in \S \ref%
{Subsection:Banach-complex} that statements (1)--(9) as in Theorem \ref%
{Theorem:G} hold for complex Banach spaces and Effros simplices, as long as
bifaces are replaced with circled faces, Lazar simplices are replaced with
Effros simplices, and symmetric functions are replaced with homogeneous
functions.

We will develop in Section \ref{Section:general} an even more general
framework to Fra\"{\i}ss\'{e} limits in functional analysis. The goal of
this further generalization is to obtain the natural noncommutative analogs
of the results above concerning the Gurarij space and the Poulsen simplex.
It is well known since the groundbreaking work of Arveson \cite%
{arveson_subalgebras_1969,arveson_subalgebras_1972} that operator systems
provide the natural noncommutative analog of compact convex sets. Indeed
compact convex sets as discussed above correspond via the map $K\mapsto A(K)$
to function systems, which are unital self-adjoint subspaces of unital
abelian C*-algebras. By replacing unital abelian C*-algebras with arbitrary
unital C*-algebras one obtains the notion of an operator system.

Let $B(H)$ be the algebra of bounded linear operators on a Hilbert space $H$
endowed with the operator norm. Concretely, an operator system is a closed
subspace $X$ of $B(H)$ that contains the identity operator $1$ and is closed
under taking adjoints. Abstractly, an operator system can be regarded as a
complex vector space containing a distinguished element $1$ (the \emph{unit}%
) endowed with the following further structure: a function $x\mapsto x^{\ast
}$ (corresponding to taking adjoints) and a norm on the space $M_{n}(X)$ on
the space of $n\times n$ matrices of elements of $X$ inherited from the
inclusion $M_{n}(X)\subset M_{n}(B(H))$. Here $M_{n}(B(H))$ is endowed with
the operator norm coming from the identification of $M_{n}(B(H))$ with the
space of operators on the Hilbertian sum of $n$ copies of $H$. The
corresponding notion of morphism $\phi :X\rightarrow Y$ is a \emph{unital
completely contractive map}. This means that $\phi $ maps the unit to the
unit (\emph{unital}), and $\left\Vert \phi ^{\left( n\right) }(x)\right\Vert
\leq \left\Vert x\right\Vert $ for any $n\in \mathbb{N}$ and $x\in M_{n}(X)$
(\emph{completely contractive}), where $\phi ^{\left( n\right) }(x)$ is the
element of $M_{n}(Y)$ obtained from $x$ by applying $\phi $ entrywise. For a
unital map, being completely contractive is equivalent to being completely
positive, which amounts at requiring that $\phi ^{\left( n\right) }(x)$ is
positive whenever $x\in M_{n}(X)$ is positive. Any function system $A(K)$
has a canonical (minimal) operator system structure, with matrix norms
defined by $\left\Vert x\right\Vert =\sup_{\phi }\left\Vert \phi
^{(n)}(x)\right\Vert $ for $x\in M_{n}\left( A(K)\right) $, where $\phi $
ranges among all the unital positive linear functionals on $A(K)$.

Works of Effros \cite{effros_aspects_1978}, Wittstock \cite%
{wittstock_matrix_1984}, Effros and Winkler \cite{effros_matrix_1997},
Webster and Winkler \cite{webster_krein-milman_1999}, and Winkler \cite%
{winkler_non-commutative_1999}, have made it clear that there exists a
natural geometric object that corresponds to an operator system and
completely encodes its structure: the matrix state space. If $X$ is an
operator system, let $S_{n}(X)$ be the compact convex set of unital
completely positive linear maps from $X$ to $M_{n}(\mathbb{C})$. The \emph{%
matrix state space }$\boldsymbol{S}(X)$ of $X$ is the sequence $\left(
S_{n}(X)\right) _{n\in \mathbb{N}}$. In $\boldsymbol{S}(X)$ one can define
the notion \emph{matrix convex combination}, which is an expression of the
form $\gamma _{1}^{\ast }v_{1}\gamma _{1}+\cdots +\gamma _{\ell }^{\ast
}v_{\ell }\gamma _{\ell }$ for $v_{i}\in S_{n_{i}}(X)$ and invertible $%
\gamma _{i}\in M_{n,n_{i}}(\mathbb{C})$. Such a matrix convex combination is 
\emph{proper }if $\gamma _{i}$ is right invertible for $i=1,2,\ldots ,\ell $
and $\gamma _{1}^{\ast }\gamma _{1}+\cdots +\gamma _{\ell }^{\ast }\gamma
_{\ell }=1$, and \emph{trivial} if for $i=1,2,\ldots ,\ell $ there exist $%
t_{i}\in \left[ 0,1\right] $ such that $\gamma _{i}^{\ast }\gamma
_{i}=t_{i}1 $ and $\gamma _{i}^{\ast }v_{i}\gamma _{i}=t_{i}v$ for $%
i=1,2,\ldots ,n$. An element of $\boldsymbol{S}(X)$ is a \emph{matrix
extreme point }if it can not be written in a nontrivial way as a proper
matrix convex combination (such a definition of matrix extreme point is
equivalent to \cite[Definition 2.1]{webster_krein-milman_1999} in view of 
\cite[Theorem A]{farenick_extremal_2000}). The original operator system $X$
can be canonically identified with the space $A\left( \boldsymbol{S}%
(X)\right) $ of \emph{matrix affine }w*-continuous mappings from $%
\boldsymbol{S}(X)$ to $\boldsymbol{S}(\mathbb{C})$ \cite[Definition 3.4]%
{webster_krein-milman_1999}.

Generally, a \emph{compact matrix convex set} $\boldsymbol{K}$ is a sequence 
$\left( K_{n}\right) $ of compact convex sets $K_{n}\subset M_{n}\left(
V\right) $ for some locally convex topological vector space $V$, that is
closed under matrix convex combinations \cite[Definition 1.1]%
{webster_krein-milman_1999}. Any compact matrix convex set arises from an
operator system as described above \cite[Proposition 3.5]%
{webster_krein-milman_1999}. Furthermore the map $\boldsymbol{K}\rightarrow
A(K)$ is a contravariant equivalence of categories from the category of
compact matrix convex sets and matrix affine continuous maps to the category
of operator systems and unital completely positive maps.

An operator system $A(\boldsymbol{K})$ is \emph{nuclear }if the identity map
of $A(\boldsymbol{K})$ is the pointwise limit of unital completely positive
maps that factor through finite-dimensional injective operator systems. In
the commutative case, a function system $A(K)$ is nuclear if and only if $%
A(K)$ is a Lindenstrauss space, which is in turn equivalent to the assertion
that $K$ is a Choquet simplex; see \cite[\S 8.6.4]{blecher_operator_2004}
and Subsection \ref{Subsection:function-system} below. Consistently, we say
that a compact matrix convex set $\boldsymbol{K}$ is a noncommutative
Choquet simplex if the associated operator system $A(\boldsymbol{K})$ is
nuclear. Several characterization of noncommutative Choquet simplices are
established in \cite{davidson_noncommutative_2016}, generalizing the
Choquet-Meyer, Bishop-de Leeuw, and Namioka-Phelps characterization of
Choquet simplices \cite%
{bishop_representations_1959,namioka_tensor_1969,olsen_simplices_1980}.

Suppose that $F$ is a compact convex subset of a metrizable Choquet simplex $%
K$. It follows from works of Lazar \cite{lazar_spaces_1968} and Alfsen and
Effros \cite%
{alfsen_structure_1972-1,alfsen_structure_1972-2,effros_class_1971} that $F$
is a face if and only the map $f\mapsto f|_{F}$ is a unital quotient mapping
whose kernel is an $M$-ideal of $A(K)$; see Proposition \ref%
{Proposition:characterize-face} below. We consider the noncommutative analog
of such a notion, and call a compact matrix convex subset $\boldsymbol{F}$
of a metrizable compact matrix convex set $\boldsymbol{K}$ a closed matrix
face if the canonical unital completely positive map $A(\boldsymbol{K}%
)\rightarrow A(\boldsymbol{F})$ is a complete quotient mapping and its
kernel is a complete $M$-ideal in the sense of Effros and Ruan \cite%
{effros_mapping_1994}.

The natural noncommutative analog $\mathbb{NP}$ of the Poulsen simplex $%
\mathbb{P}$ is the matrix state space of the Fra\"{\i}ss\'{e} limit $A(%
\mathbb{\mathbb{NP}})$ of the class of exact finite-dimensional operator
systems. We will call $\mathbb{NP}$ the\emph{\ noncommutative Poulsen simplex%
} and $A(\mathbb{\mathbb{NP}})$ the \emph{noncommutative Poulsen system}. A
direct proof of existence and uniqueness of $\mathbb{NP}$ can be found in 
\cite{davidson_noncommutative_2016}. It is also proved in \cite%
{davidson_noncommutative_2016} that $\mathbb{\mathbb{NP}}$ is the unique
nontrivial metrizable noncommutative Choquet simplex with dense matrix
extreme boundary, and $A(\mathbb{NP})$ is the unique separable nuclear
operator systems that is universal in the sense of Kirchberg and Wassermann 
\cite{kirchberg_c*-algebras_1998}. The model-theoretic properties of the
noncommutative Poulsen system have been investigated in \cite%
{goldbring_model-theoretic_2015}. The following noncommutative analog of
Theorem \ref{Theorem:A(NG)} follows from our general results; see Subsection %
\ref{Subsection:osystems}.

\begin{theorem}
\label{Theorem:A(NG)}Let $A(\mathbb{\mathbb{NP}})$ be the Fra\"{\i}ss\'{e}
limit of the class of finite-dimensional exact operator systems, and let $%
\mathbb{NP}$ be its matrix state space.

\begin{enumerate}
\item $A(\mathbb{\mathbb{NP}})$ is a nuclear operator system, and it is the
unique separable exact approximately homogeneous operator system that
contains unital completely isometric copies of any separable exact operator
system.

\item The set of matrix extreme points of $\mathbb{NP}$ is dense.

\item A metrizable compact matrix convex set is a noncommutative Choquet
simplex if and only if it is matrix affinely homeomorphic to a closed proper
matrix face of $\mathbb{NP}$.

\item The canonical action of $\mathrm{Aut}(\mathbb{\mathbb{NP}})$ on the
state space of $A(\mathbb{\mathbb{NP}})$ is minimal.
\end{enumerate}
\end{theorem}

The noncommutative Poulsen system is, in particular, the first example of a
nuclear operator system that contains a completely isometric copy of any
separable exact operator system.

We also consider the noncommutative analogs of the Gurarij Banach space and
of the Lusky simplex. Operator spaces \cite%
{pisier_introduction_2003,effros_operator_2000} are the noncommutative
analog of (complex) Banach spaces. Indeed a Banach space can be seen as a
subspace of an abelian C*-algebra. By considering arbitrary, not necessarily
abelian C*-algebras, one obtains the notion of an operator space.
Concretely, an operator space is a closed subspace of the algebra $B(H)$ of
bounded linear operators on a Hilbert space. Abstractly, an operator space $%
X\subset B(H)$ can be seen as a structure consisting of the vector space
operations together with the \emph{matrix norms }arising from the inclusion $%
M_{n}(X)\subset M_{n}(B(H))$. The corresponding notion of morphism is a
completely contractive linear map. An operator space is \emph{nuclear }if
the identity map of $X$ is the pointwise limit of completely contractive
maps that factor through finite-dimensional injective operator spaces.

Any Banach space can be regarded as an operator space with its canonical
minimal operator space structure (minimal quantization); see \cite[Section
3.3]{effros_operator_2000}. A Banach space $X$ is Lindenstrauss if and only
if it is a nuclear operator space with its minimal operator space structure 
\cite[Proposition 8.6.5]{blecher_operator_2004}. Thus, nuclear operator
spaces can be seen as the noncommutative analog of Lindenstrauss spaces.

As in the case of Banach spaces, one can associate with an operator space a
geometric object that completely encodes its structure. Suppose that $X$ is
an operator space. We let the complete dual ball $\mathrm{C\mathrm{Ball}}%
(X^{\ast })$ to be the sequence $\left( K_{n,m}\right) _{n,m\in \mathbb{N}}$
where $K_{n,m}$ is the unit ball of $M_{n,m}(X^{\ast })$. It is easy to see
that $\mathrm{C\mathrm{Ball}}(X^{\ast })$ is closed under \emph{rectangular
matrix convex combinations}. These are expressions of the form $\alpha
_{1}^{\ast }v_{1}\beta _{1}+\cdots +\alpha _{\ell }^{\ast }v_{\ell }\beta
_{\ell }$ where $\alpha _{i}\in M_{n_{i},n}(\mathbb{C})$ and $\beta _{i}\in
M_{m_{i},m}$ and $v_{i}\in K_{n_{i}m_{i}}$ for $1\leq i\leq \ell $.

If $V$ is a locally convex vector space, and $\boldsymbol{K}$ is a
collection of compact subsets $K_{n,m}\subset M_{n,m}\left( V\right) $, then
we say that $\boldsymbol{K}$ is a compact \emph{rectangular matrix convex set%
} if it is closed under rectangular matrix convex combinations. The notions
of matrix affine map, matrix affine combination, and matrix affine extreme
point have natural rectangular analogs. The bipolar theorem and the
Krein-Milman theorem for compact rectangular matrix convex sets have been
established in \cite[Section 3]{fuller_boundary_2016}.

If $\boldsymbol{K}$ is a compact rectangular convex set, then we let $%
A_{\sigma }(\boldsymbol{K})$ be the space of continuous rectangular affine
maps from $\boldsymbol{K}$ to $\mathbb{C}$ endowed with its canonical
operator space structure; see \cite[Section 3]{fuller_boundary_2016}. It is
proved in \cite[Section 3]{fuller_boundary_2016} using the bipolar theorem
for compact rectangular matrix convex sets that if $\boldsymbol{K}$ is a
compact rectangular matrix convex set, then $\boldsymbol{K}$ can be
identified with $\mathrm{C\mathrm{Ball}}(X^{\ast })$, where $X$ is the
operator space $A_{\sigma }(\boldsymbol{K})$, via the map sending $x\in X$
to the continuous rectangular affine map $\left[ \phi _{ij}\right] \mapsto %
\left[ \phi _{ij}(x)\right] $ . Furthermore the assignment $\boldsymbol{K}%
\mapsto A_{\sigma }(\boldsymbol{K})$ is a contravariant equivalence of
categories from the category of compact rectangular matrix convex and
continuous rectangular matrix convex maps to the category of operator spaces
and completely contractive linear maps. Consistently with the commutative
setting, say that $\boldsymbol{K}$ is a (metrizable) noncommutative Lazar
simplex if $A_{\sigma }(K)$ is nuclear.

It has been proved by Lazar and Lindenstrauss that a compact absolutely
convex subset $F$ of a Lazar simplex $K$ is a closed biface if and only if
the kernal of the map $A_{\sigma }(K)\rightarrow A_{\sigma }\left( F\right) $%
, $f\mapsto f|_{F}$ is an $M$-ideal; see Proposition \ref%
{Proposition:characterize-biface}. A similar characterization holds for
complex Lindenstrauss spaces by results of Ellis-Rao-Roy-Utterud \cite%
{ellis_facial_1981} and Olsen \cite{olsen_edwards_1976}; see Proposition \ref%
{Proposition:characterize-circled-face}. Consistently, we define a closed
rectangular matrix face of a compact rectangular matrix convex set $%
\boldsymbol{K}$ to be a compact rectangular matrix convex subset $%
\boldsymbol{F}$ of $\boldsymbol{K}$ such that the map $A_{\sigma }(%
\boldsymbol{K})\rightarrow A_{\sigma }(\boldsymbol{F})$, $f\mapsto f|_{%
\boldsymbol{F}}$ is a complete quotient mapping whose kernel is a complete $%
M $-ideal.

The natural noncommutative analog of the Gurarij space is a nuclear operator
space that is approximately ultrahomogeneous and contains a completely
isometric copy of any separable exact operator space. It follows from the
general results of this paper that such a space exists, it is unique, and it
coincides with the noncommutative Gurarij space $\mathbb{\mathbb{NG}}$
defined in \cite{oikhberg_non-commutative_2006} and proved to be unique in 
\cite{lupini_uniqueness_2016}. The space $\mathrm{C\mathrm{\mathrm{Ball}}}(%
\mathbb{NG}^{\ast })$ can be seen as the noncommutative analog of the Lusky
simplex $\mathbb{L}=\mathrm{Ball}(\mathbb{G}^{\ast })$. We will call $%
\mathrm{C\mathrm{\mathrm{Ball}}}(\mathbb{NG}^{\ast })$ the \emph{%
noncommutative Lusky simplex }and denote it by $\mathbb{NL}$. The following
result is the natural noncommutative analog of (the complex version of)
Theorem \ref{Theorem:G}.

\begin{theorem}
\label{Theorem:NG}Let $\mathbb{NG}$ be the Gurarij operator space, and $%
\mathbb{NL}$ be the noncommutative Lazar simplex.

\begin{enumerate}
\item $\mathbb{NG}$ is a nuclear operator space, and it is the unique
separable exact operator space that contains an isometric copy of any
separable exact operator space \cite{lupini_uniqueness_2016}.

\item The set of rectangular matrix extreme points of $\mathbb{NL}$ is dense
in $\mathbb{NL}$.

\item A metrizable compact rectangular convex set is a noncommutative Lazar
simplex if and only if it is rectangular affinely homeomorphic to a proper
closed rectangular matrix face of $\mathbb{NL}$.
\end{enumerate}
\end{theorem}

Our framework also covers other new examples of Fra\"{\i}ss\'{e} classes,
such as the class of finite-dimensional operator sequence spaces (\S \ref%
{Subsection:sequence}), and the class of finite-dimensional $p$-multinormed
spaces for every $p\in \left( 1,+\infty \right) $ (\S \ref%
{Subsection:p-multinormed}). The corresponding limits $\mathbb{CG}$ (the 
\emph{column Gurarij space}) and $\mathbb{GM}^{p}$ (the $p$-\emph{%
multinormed Gurarij space}) give new examples---in addition to $\mathbb{G}$
and $A(\mathbb{P})$---of separable metric structures whose first order
theory is separably categorical and admits elimination of quantifiers \cite[%
\S 13]{ben_yaacov_model_2008}. As a consequence the corresponding
automorphism groups $\mathrm{Aut}(\mathbb{G})$, $\mathrm{Aut}(\mathbb{P})$, $%
\mathrm{Aut}(\mathbb{CG)}$, and $\mathrm{Aut}(\mathbb{GM}^{p}\mathbb{)}$ for 
$p\in \left( 1,+\infty \right) $ are new examples of Roelcke precompact
Polish groups \cite[Definition 1.1,Theorem 2.4]{ben_yaacov_weakly_2013}.
Similar results as the ones mentioned above hold for $\mathbb{CG}$ and $%
\mathbb{GM}^{p}$.

In addition to the results above, our general framework will apply to
produce commutative and noncommutative analogs of universal operators in the
sense of Rota, generalizing work of Garbuli\'{n}ska-W\c{e}grzyn and Kubi\'{s}
\cite{garbulinska_universal_2015}; see Theorem \ref{Theorem:operator-G},
Theorem \ref{Theorem:operator-A(P)}, Theorem \ref{Theorem:universal-cc},
Theorem \ref{Theorem:universal-ucp}, Theorem \ref%
{Theorem:universal-projection-G}, Theorem \ref%
{Theorem:universal-projection-A(P)}, Theorem \ref%
{Theorem:universal-projection-NG}, and Theorem \ref%
{Theorem:universal-projection-NP}.

The rest of the paper is organized as follows. In Section \ref%
{Section:injective} we present the general framework of Fra\"{\i}ss\'{e}
classes generated by injective objects, and provide a characterization of
the corresponding limits. A characterization of retracts of the limit $M$ is
provided in Section \ref{Section:retracts}. The existence of generic
(universal) morphisms $M\rightarrow M$ in this general setting is proved in
Section \ref{Section:operators}, while in Section \ref{Section:states} we
prove the existence of generic morphisms $M\rightarrow R$ for any separable
approximately injective structure $R$. Section \ref{Section:examples}
provides several examples, explaining how real and complex Banach spaces,
function systems, $M_{q}$-spaces, $M_{q}$-systems, operator sequence spaces,
and $p$-multinormed spaces fit into the general framework. Finally Section %
\ref{Section:general} considers an even more general approach, suitable to
deal with the cases of exact operator spaces and exact operator systems.
These examples are presented in Section \ref{Section:more-examples}.

\subsection*{Acknowledgments}

We would like to thank Ita\"{\i} Ben Yaacov, Ken Davidson, Isaac Goldbring,
Ilijas Farah, Adam Fuller, Alexander Kechris, Matthew Kennedy, Michael
Hartz, Ward Henson, Fernando Lled\'{o}, Jordi L\'{o}pez-Abad, Wieslaw Kubi%
\'{s}, Timur Oikhberg, Slawomir Solecki, Pedro Tradacete, and Todor Tsankov
for their comments and many helpful conversations.

\section{Fra\"{\i}ss\'{e} classes generated by injective objects\label%
{Section:injective}}

\subsection{Morphisms and embeddings\label{Subsection:morphism}}

Throughout this section we suppose that $\mathcal{L}$ is a countable \emph{%
language} in the logic for metric structures. For simplicity we will assume
that $\mathcal{L}$ is single-sorted. A complete introduction to the logic
for metric structures can be found in \cite{ben_yaacov_model_2008}. We
recall here the key concepts. The language $\mathcal{L}$ is a countable
collection of \emph{function symbols} and \emph{relation symbols}. Every
symbol $B$ in $\mathcal{L}$ has assigned an \emph{arity} $n_{B}\in \mathbb{N}
$ and a\emph{\ modulus of continuity} $\varpi _{B}$. An $\mathcal{L}$%
-structure $X$ is a complete metric space with metric bounded by $1$ endowed
with the \emph{interpretation} $X^{B}$ for any relation symbol $B$ in $%
\mathcal{L}$. Here $X^{B}$ is a function from $X^{n_{B}}$ to either $X$ or a
compact interval $\left[ \lambda _{B},\mu _{B}\right] \subset \mathbb{R}$
(depending whether $B$ is a function or a relation symbol) that is uniformly
continuous with modulus $\varpi _{B}$ with respect to the supremum metric on 
$X^{n_{B}}$. We will assume that $\mathcal{L}$ contains a distinguished
binary relation symbol whose interpretation in an $\mathcal{L}$-structure is
the distance function.

Suppose that $\left( x_{n}\right) $ is a fixed collection of \emph{variables}%
. We denote by $\bar{x}$ a tuple of such variables. Terms in the language $%
\mathcal{L}$ are defined recursively by declaring that any variable $x$ is a
term $t(x)$, and if $t_{1}(\bar{x}_{1}),\ldots ,t_{n}\left( \bar{x}%
_{n}\right) $ are terms, and $f$ is an $n$-ary function symbol in $\mathcal{L%
}$, then $f(t_{1},\ldots ,t_{k})$ is a term $t\left( \bar{x}_{1},\ldots ,%
\bar{x}_{n}\right) $. An atomic formula $\varphi (\bar{x})$ in the language $%
\mathcal{L}$ is an expression of the form $B\left( t_{1}(x),\ldots
,t_{n}(x)\right) $ where $t_{1},\ldots ,t_{n}$ are terms and $B$ is an $n$%
-ary relation symbol in $\mathcal{L}$. The interpretation of an atomic
formula $\varphi (x)$ in an $\mathcal{L}$-structure $M$ is defined in the
obvious way in terms of the interpretation of $B$ and of the function
symbols that appear in the terms $t_{1},\ldots ,t_{n}$. A quantifier-free
formula is an expression $q\left( \varphi _{1}(\bar{x}),\ldots ,\varphi _{n}(%
\bar{x})\right) $ where $\varphi _{1}(\bar{x}),\ldots ,\varphi _{n}(\bar{x})$
are quantifier-free formulas and $q:\mathbb{R}^{n}\rightarrow \mathbb{R}$ is
a continuous function.

\begin{definition}
\label{Definition:morphism}If $E,F$ are $\mathcal{L}$-structures and $%
T:E\rightarrow F$ is a function, then we say that $T$ is:

\begin{itemize}
\item a \emph{morphism }if $T\left( \varphi (\bar{a})\right) \leq \varphi (%
\bar{a})$ for any atomic formula $\varphi (x)$ and tuple $\bar{a}$ in $E$;

\item an \emph{embedding }if $T\left( \varphi (\bar{a})\right) =\varphi (%
\bar{a})$ for any atomic formula $\varphi (x)$ and tuple $\bar{a}$ in $E$.
\end{itemize}

A \emph{retraction} of an $\mathcal{L}$-structure $A$ is a morphism $%
r:A\rightarrow A$ such that $r\circ r=r$. A \emph{retract} is the range of a
retraction.
\end{definition}

We regard $\mathcal{L}$-structures as objects of a category where morphisms
are defined as in Definition \ref{Definition:morphism}. Observe that the
isomorphisms in such a category are precisely the surjective embeddings. We
note here that when these notion are applied to Banach spaces as metric
structures (by identifying them with their unit ball), morphisms as in
Definition \ref{Definition:morphism} correspond to linear maps of norm at
most $1$, embeddings as in Definition \ref{Definition:morphism} correspond
to isometric linear maps, and isomorphisms as in Definition \ref%
{Definition:morphism} correspond to linear isometric isomorphism.

\begin{definition}
\label{Definition:substructure}If $E$ is an $\mathcal{L}$-structure and $%
\bar{a}$ is a finite tuple in $E$, then we denote by $\left\langle \bar{a}%
\right\rangle $ the \emph{substructure of }$E$\emph{\ generated by }$\bar{a}$%
. This is by definition the set of $b\in E$ such that, whenever $%
f,g:E\rightarrow F$ are morphisms such that $f(\bar{a})=g(\bar{a})$, one has
that $f(b)=g(b)$. We say that $X$ is \emph{finitely generated }if $%
X=\left\langle \bar{a}\right\rangle $ for some finite tuple $\bar{a}$ in $X$%
. A subset $Y$ of $X$ is a \emph{substructure} if it contains $\left\langle 
\bar{a}\right\rangle $ for any finite tuple $\bar{a}$ in $Y$.
\end{definition}

The phrasing of the notion of substructure is chosen in such a way that,
when a Banach space is seen as a structure by looking at its unit ball, then
the substructure generated by a tuple $\bar{a}$ coincides with the unit ball
of the linear span of $\bar{a}$; see Subsection \ref{Subsection:Banach}.

Observe that if $\phi :E\rightarrow F$ is a morphism, then the image $\phi %
\left[ E\right] $ of $E$ under $\phi $ is a substructure of $F$. If $\bar{a},%
\bar{a}^{\prime }$ are two tuples in $E$ of the same length, then we set $%
d\left( \bar{a},\bar{a}^{\prime }\right) =\max_{i}d\left(
a_{i},a_{i}^{\prime }\right) $. We convene that $d\left( \bar{a},\bar{a}%
^{\prime }\right) =+\infty $ if $\bar{a}$ and $\bar{a}^{\prime }$ have
different lengths.

\begin{definition}
\label{Definition:I}If $T,S:X\rightarrow Y$ are morphisms, then we let $I(T)$
be the supremum of%
\begin{equation}
\left\vert \varphi (\bar{a})-\varphi \left( T(\bar{a})\right) \right\vert 
\text{\label{Equation:I}}
\end{equation}%
where $\varphi (x)$ is an atomic formula and $\bar{a}$ is a tuple in $X$.
Similarly we let $d\left( T,S\right) $ be the supremum of $d\left(
T(x),S(x)\right) $ where $x$ ranges in $X$.
\end{definition}

Observe that $I\left( \phi \circ \psi \right) \leq I(\phi )+I\left( \psi
\right) $.

\begin{definition}
\label{Definition:GH}We define the \emph{Gromov-Hausdorff (GH) }distance $%
d\left( X,Y\right) $ of two structures $X,Y$ in $\mathcal{A}$ as follows: $%
d\left( X,Y\right) $ is the infimum of $\varepsilon >0$ such that there
exists morphisms $f:X\rightarrow Y$ and $g:Y\rightarrow X$ such that $%
d\left( g\circ f,id_{X}\right) <\varepsilon $, $d\left( f\circ
g,id_{Y}\right) <\varepsilon $, $I\left( f\right) <\varepsilon $, and $%
I\left( g\right) <\varepsilon $.
\end{definition}

It is not difficult to verify that the GH distance is indeed a metric.
Dropping the requirement that $I\left( f\right) <\varepsilon $ and $I\left(
g\right) <\varepsilon $ in Definition \ref{Definition:GH} yields an
equivalent metric.

\subsection{Basic sequences\label{Subsection:basic}}

Suppose that $\mathcal{A}$ is a class of $\mathcal{L}$-structures such that

\begin{enumerate}
\item a structure belongs to $\mathcal{A}$ if and only if each of its
finitely generated substructures belong to $\mathcal{A}$,

\item $\mathcal{A}$ is closed under inductive limits with embeddings as
connective maps,

\item $\mathcal{A}$ has arbitrary products,

\item $\mathcal{A}$ has a universal initial object, which is a finitely
generated structure,

\item if $f_{i}:X\rightarrow Y_{i}$ is a collection of morphism between
structures in $\mathcal{A}$, $Y$ is the product of $Y_{i}$, $f:X\rightarrow
Y $ is the morphism obtained from the universal property of the product, $%
\varphi $ is an atomic formula, and $\bar{a}$ is a tuple in $X$, then $%
\varphi \left( f(\bar{a})\right) =\sup_{i}\varphi \left( f_{i}(\bar{a}%
)\right) $,

\item for structures $A,X,Y$ in $\mathcal{A}$, morphisms $%
f_{X}^{(i)}:A\rightarrow X$ and $f_{Y}^{(i)}:A\rightarrow Y$ for $%
i=1,2,\ldots ,n$, atomic formula $\varphi (x)$, and tuple $\bar{a}$ in $A$,
if $Z$ is the product of $X$ and $Y$, $f^{(i)}:A\rightarrow Z$ are the
morphisms obtained from $f_{X}^{(i)}$ and $f_{Y}^{(i)}$, respectively, and
the universal property of the product, then $\varphi (f^{(1)}(\bar{a}%
),\ldots ,f^{(n)}(\bar{a}))\leq \max \{\varphi (f_{X}^{(1)}(\bar{a}),\ldots
,f_{X}^{(n)}(\bar{a})),\varphi (f_{Y}^{(1)}(\bar{a}),\ldots ,f_{Y}^{(n)}(%
\bar{a}))\}$;

\item if $A,B$ are finitely-generated structures in $\mathcal{A}$, then the
space of morphisms from $A$ to $B$ is totally bounded with respect to the
metric from Definition \ref{Definition:I}.
\end{enumerate}

Observe that in particular these assumptions guarantee that the canonical
morphism from $X$ to the product of $X$ and $Y$ is an embedding. We also
suppose that any structure $X$ in $\mathcal{A}$ is endowed with a collection
of finite tuples of pairwise distinct elements of $X$ that we call \emph{%
basic tuples}. We assume that any finitely generated structure in $\mathcal{A%
}$ has a generating basic tuple, and any finite tuple contains a basic
subtuple.

\begin{definition}
\label{Definition:fundamental}We say that a subset $D$ of a structure $X$ in 
$\mathcal{A}$ is \emph{fundamental} if it generates a dense substructure of $%
X$, and the set of basic tuples from $D$ is dense in the set of basic tuples
from $X$.
\end{definition}

We assume that any separable structure in $\mathcal{A}$ hs a countable
fundamental subset. In the following we fix for every separable structure $X$
in $\mathcal{A}$ a countable fundamental subset $D_{X}$ of $X$. We also
assume that for any structure $X$ in $\mathcal{A}$ and basic tuple $\bar{a}$
in $X$ there exists a strictly increasing function $\rho _{\bar{a}%
}:[0,\delta _{\bar{a}})\rightarrow \lbrack 0,+\infty )$ that is vanishing at 
$0$ and continuous at $0$ such that if $f,g:X\rightarrow Y$ are morphisms
such that $d\left( f(\bar{a}),g(\bar{a})\right) \leq \delta \leq \delta _{%
\bar{a}}$, then there exists a morphism $h:\left\langle f(\bar{a}%
)\right\rangle \rightarrow Y$ such that $d\left( h\circ f,g\right) \leq \rho
_{\bar{a}}(\delta )$. The latter requirement can be seen as the assertion
that basic tuples satisfy the natural analogue of the small perturbation
lemma from Banach space and operator space theory \cite[Lemma 2.13.2]%
{pisier_introduction_2003}.

A \emph{marked structure }$\left( E,\bar{a}\right) $ in $\mathcal{A}$ is a
structure $E$ in $\mathcal{A}$ endowed with a distinguished generating basic
tuple $\bar{a}$. We call a marked structure $\left( E,\bar{a}\right) $ where 
$\bar{a}$ has length $n$ an $n$-marked structure. In the following we denote
the marked structure $\left( E,\bar{a}\right) $ simply by $\bar{a}$ and
refer to $E$ as $\left\langle \bar{a}\right\rangle $. If $\bar{a},\bar{b}$
are $n$-marked structures, we let $\partial \left( \bar{a},\bar{b}\right) $
be the infimum of $\max \left\{ I(f),d\left( f(\bar{a}),\bar{b}\right)
\right\} $ where $f$ ranges among all the morphisms $f:\left\langle \bar{a}%
\right\rangle \rightarrow \left\langle \bar{b}\right\rangle $. Observe that $%
\partial \left( \bar{a},\overline{c}\right) \leq \partial \left( \bar{a},%
\bar{b}\right) +\partial \left( \bar{b},\overline{c}\right) $. However $%
\partial $ might not be symmetric, and hence it is not a metric in general.

\subsection{Fra\"{\i}ss\'{e} classes generated by injective objects\label%
{Subsection:injective}}

We say that a structure $A$ in $\mathcal{A}$ is \emph{injective }if it is an
injective object of $\mathcal{A}$ when regarded as a category with the
notion of morphisms from Definition \ref{Definition:morphism}. This means
that if $X\subset Y$ are structures in $\mathcal{A}$ and $f:X\rightarrow A$
is a morphism, then there exists a morphism $g:Y\rightarrow A$ that extends $%
f$. We suppose in the following that $\mathcal{I}$ is a countable collection
of finitely generated injective elements of $\mathcal{A}$ closed under
finite products.

For now and the rest of the section we fix a function $\varpi :[0,+\infty
)\rightarrow \lbrack 0,+\infty )$ that is a strictly increasing, continuous
at $0$, and vanishing at $0$.

\begin{definition}
\label{Definition:approximate-inverses}The class $\mathcal{A}$ has \emph{%
enough injectives} from $\mathcal{I}$ with modulus $\varpi $ if every
finitely-generated structure in $\mathcal{A}$ is the limit with respect to
the Gromov-Hausdorff distance of finitely-generated substructures of
structures in $\mathcal{I}$, and for any separable structures $X,\widehat{X}%
,A$ in $\mathcal{A}$ with $X$ finitely generated and $A\in \mathcal{I}$, and
morphisms $\phi :X\rightarrow \widehat{X}$ and $f:X\rightarrow A$ such that $%
I(\phi )\leq \delta $ there exists a morphism $h:\widehat{X}\rightarrow A$
such that $d\left( h\circ \phi ,f\right) \leq \varpi (\delta )$.
\end{definition}

\begin{definition}
\label{Definition:stable-homogeneity}A structure $M$ in $\mathcal{A}$ is 
\emph{stably homogeneous }with modulus $\varpi $ if whenever $E$ is a
finitely generated structure in $\mathcal{A}$, and $\phi :E\rightarrow M$
and $f:E\rightarrow M$ are morphisms such that $I(f)<\delta $ and $I(\phi
)<\delta $, then there exists an automorphism $\alpha $ of $M$ such that $%
d\left( \alpha \circ \phi ,f\right) <\varpi (\delta )$.
\end{definition}

The following is the main general theorem characterizing Fra\"{\i}ss\'{e}
classes generated by injective objects. We will recall the notion of
(metric) Fra\"{\i}ss\'{e} class as defined in \cite[Definition 3.12]%
{ben_yaacov_fraisse_2015} in Subsection \ref{Subsection:class}.

\begin{theorem}
\label{Theorem:abstract-nonsense}Assume that $\mathcal{A}$ is a category of $%
\mathcal{L}$-structures satisfying the assumptions of Subsection \ref%
{Subsection:basic}. Let $\mathcal{I}$ be a collection of finitely generated
injective structures of $\mathcal{A}$ closed under finite products. Denote
by $\mathcal{C}$ the class of finitely-generated structures in $\mathcal{A}$%
. The following statements are equivalent:

\begin{enumerate}
\item $\mathcal{C}$ is a Fra\"{\i}ss\'{e} class, the limit $M$ of $\mathcal{C%
}$ can be realized as an inductive limit of structures from $\mathcal{I}$
with embeddings as morphisms, any structure in $\mathcal{I}$ is isomorphic
to a retract of $M$, and $M$ is stably homogeneous with modulus $\varpi $;

\item $\mathcal{A}$ has enough injectives from $\mathcal{I}$ with modulus $%
\varpi $.
\end{enumerate}
\end{theorem}

The implication (1)$\Rightarrow $(2) is a consequence of the universality
property of the Fra\"{\i}ss\'{e} limit together with our assumption on basic
sequences. The rest of this section and the next section are devoted to
prove the implication (2)$\Rightarrow $(1). We will assume throughout that $%
\mathcal{A}$ and $\mathcal{I}$ are classes of $\mathcal{L}$-structures
satisfying the assumptions of Theorem \ref{Theorem:abstract-nonsense} and
such that $\mathcal{A}$ has enough injectives from $\mathcal{I}$ with
modulus $\varpi $. A characterization of the Fra\"{\i}ss\'{e} limit of $%
\mathcal{C}$ will be given in Proposition \ref%
{Proposition:characterize-limit}.

\subsection{Approximate pushouts\label{Subsection:pushout}}

In this subsection we prove that the assumptions above on $\mathcal{A}$
allow one to \emph{amalgamate }the structures in $\mathcal{A}$ over a common
substructure.

\begin{lemma}
\label{Lemma:NAP}Suppose that $E,X,Y$ are separable structures in $\mathcal{A%
}$ such that $X,Y$ belong to $\mathcal{I}$ and $E$ is finitely generated,
and $f_{X}:E\rightarrow X$ and $f_{Y}:E\rightarrow Y$ are morphisms. If $%
I\left( f_{X}\right) \leq \delta $ and $I\left( f_{Y}\right) \leq \delta $,
then there exists a structure $Z$ in $\mathcal{I}$ and embeddings $%
i:X\rightarrow Z$ and $j:Y\rightarrow Z$ such that $d\left( i\circ
f_{X},j\circ f_{Y}\right) \leq \varpi (\delta )$.
\end{lemma}

\begin{proof}
Since $Y$ is injective, and $\mathcal{A}$ enough injectives from $\mathcal{I}
$ with modulus $\varpi $, there exists a morphism $h_{X}:X\rightarrow Y$
such that $d\left( h_{X}\circ f_{X},f_{Y}\right) \leq \varpi (\delta )$.
Similarly there exists a morphism $h_{Y}:Y\rightarrow X$ such that $d\left(
h_{Y}\circ f_{Y},h_{X}\right) \leq \varpi (\delta )$. Let now $Z$ be the
product of $X$ and $Y$, and $i:X\rightarrow Z$ be the morphism obtained from
the morphisms $id_{X}:X\rightarrow X$ and $h_{X}:X\rightarrow Y$ using the
universal property of the product. Similarly let $j:Y\rightarrow Z$ be the
morphism obtained from the morphisms $h_{Y}:Y\rightarrow X$ and $%
id_{Y}:Y\rightarrow Y$ using the universal property of the product. Observe
that $i$ and $j$ are embeddings. Furthermore $d\left( i\circ f_{X},j\circ
f_{Y}\right) \leq \varpi (\delta )$ by Condition (6) of Subsection \ref%
{Subsection:basic}.
\end{proof}

\begin{lemma}
\label{Lemma:pushout}Suppose that $X,\widehat{X},Y$ are structures in $%
\mathcal{A}$, and $\phi :X\rightarrow \widehat{X}$ and $f:X\rightarrow Y$
are morphisms such that $I(\phi )\leq \delta $. Then there exist a structure 
$\widehat{Y}$ in $\mathcal{A}$, a morphism $\widehat{f}:\widehat{X}%
\rightarrow \widehat{Y}$, and an embedding $j:Y\rightarrow \widehat{Y}$ such
that $d(\widehat{f}\circ \phi ,j\circ f)\leq \varpi (\delta )$ and
furthermore for any structure $Z$ in $\mathcal{A}$ and morphisms $g:\widehat{%
X}\rightarrow Z$ and $h:Y\rightarrow Z$ such that $d(g\circ \phi ,h\circ
f)\leq \varpi (\delta )$ there exists a morphism $\tau :\widehat{Y}%
\rightarrow Z$ such that $g=\tau \circ \widehat{f}$ and $h=\tau \circ j$. If
moreover $I(f)\leq \delta $, then $\widehat{f}$ is an embedding. If $%
\widehat{X},Y$ are finitely generated, then $\widehat{Y}$ is finitely
generated.
\end{lemma}

\begin{proof}
Consider the collection $\left( g_{i},h_{i}\right) $ of all the morphisms $%
g_{i}:\widehat{X}\rightarrow A_{i}$ and $h_{i}:Y\rightarrow A_{i}$ for $%
A_{i}\in \mathcal{I}$ such that $d\left( g_{i}\circ \phi ,h_{i}\circ
f\right) \leq \varpi (\delta )$. Let $W$ be the product of $A_{i}$ in $%
\mathcal{A}$ and $\widehat{f}:\widehat{X}\rightarrow W$ and $j:Y\rightarrow
W $ the morphisms obtained from the morphisms $g_{i}$ and $h_{i}$ and the
universal property of the product. We claim that $j$ is an embedding. In
fact, suppose that $\bar{b}$ is a tuple in $Y$ and $\psi (x)$ is an atomic
formula such that $\psi (\bar{b})>r$. Then there exists $A\in \mathcal{I}$
and a morphism $h:Y\rightarrow A$ such that $\psi \left( h(\bar{b})\right)
>r $. Since $\mathcal{A}$ has enough injectives from $\mathcal{I}$, there
exists a morphism $g:\widehat{X}\rightarrow A$ such that $d\left( g\circ
\phi ,h\circ f\right) \leq \varpi (\delta )$. Therefore $g=g_{i}$ and $%
h=h_{i}$ for some $i$ as above and hence%
\begin{equation*}
\psi \left( j(\bar{b})\right) \geq \psi \left( h(\bar{b})\right) >r\text{.}
\end{equation*}%
This shows that $j:Y\rightarrow W$ is an embedding. Let $\widehat{Y}$ be the
substructure of $W$ generated by the union of the ranges of $\widehat{f}$
and $j$. Suppose now that $Z$ is a structure in $\mathcal{A}$ and $g:%
\widehat{X}\rightarrow Z$ and $\eta :\widehat{Y}\rightarrow Z$ are morphisms
such that $d\left( g\circ \phi ,\eta \circ f\right) \leq \varpi (\delta )$.
Since $\mathcal{A}$ has enough injectives from $\mathcal{I}$, $Z$ embeds
into a product $\widehat{Z}$ of structures in $\mathcal{I}$. The definition
of $W$ above guarantees the existence of a unique morphism $\tau
:W\rightarrow \widehat{Z}$ such that $\tau \circ \widehat{f}=g$ and $\tau
\circ j=\eta $. Since $\widehat{Y}$ is the substructure of $W$ generated by
the ranges of $\widehat{f}$ and $j$, we have that $\tau $ maps $\widehat{Y}$
into $Z$. Finally under the assumption that $I(f)\leq \delta $ one can prove
that $\widehat{f}$ is an embedding reasoning as above.
\end{proof}

The structure $\widehat{Y}$ in $\mathcal{A}$ constructed in Lemma \ref%
{Lemma:pushout} will be called the \emph{approximate pushout} of the
morphisms $f$ and $\phi $ with tolerance $\varpi (\delta )$.

\subsection{The Fra\"{\i}ss\'{e} class\label{Subsection:class}}

Let $\mathcal{C}$ be the class of finitely generated elements of $\mathcal{A}
$. We aim at showing that $\mathcal{C}$ is a (complete) Fra\"{\i}ss\'{e}
class in the sense of \cite[Definition 3.12]{ben_yaacov_fraisse_2015}. Fix $%
n\in \mathbb{N}$ and let $\mathcal{C}_{n}$ be the class of $n$-marked
structures in $\mathcal{A}$. (It should be remarked that arbitrary tuples of
generators are considered in \cite{ben_yaacov_fraisse_2015}, rather than
only basic tuples as we do here. However this does not pose any problem, and
all the results in \cite{ben_yaacov_fraisse_2015} go through only
considering basic tuples.) Recall that the Fra\"{\i}ss\'{e} metric $d_{%
\mathcal{C}}$ on $\mathcal{C}_{n}$ is defined by%
\begin{equation*}
d_{\mathcal{C}}\left( \bar{a},\bar{b}\right) =\inf_{\phi ,\psi }d\left( \phi
(\bar{a}),\psi (\bar{b})\right)
\end{equation*}%
where $\phi :\left\langle \bar{a}\right\rangle \rightarrow Z$ and $\psi
:\left\langle \bar{b}\right\rangle \rightarrow Z$ range among all the joint
embeddings into a third structure $Z$ in $\mathcal{C}$; see also \cite[%
Definition 3.11]{ben_yaacov_fraisse_2015}.

In order to prove that $\mathcal{C}$ is a Fra\"{\i}ss\'{e} class as in \cite[%
Definition 3.12]{ben_yaacov_fraisse_2015}, we need to show that

\begin{itemize}
\item $\mathcal{C}$ satisfies the \emph{hereditary property }(HP), that is, $%
\mathcal{C}$ is closed under taking finitely generated substructures;

\item $\mathcal{C}$ satisfies the \emph{joint embedding property }(JEP),
that is, any two structures in $\mathcal{C}$ simultaneously embed into a
third structure in $\mathcal{C}$;

\item $\mathcal{C}$ satisfies the \emph{near amalgamation property }(NAP),
that is, if $\bar{a}$ is a marked structure in $\mathcal{C}$, $\varepsilon
>0 $, $B_{i}$ are structures in $\mathcal{C}$ and $\phi _{i}:\left\langle 
\bar{a}\right\rangle \rightarrow B_{i}$ are embeddings for $i\in \left\{
0,1\right\} $, then there exists a structure in $\mathcal{C}$ and embeddings 
$\psi _{i}:B_{i}\rightarrow C$ such that $d\left( \left( \psi _{0}\circ \phi
_{0}\right) (\bar{a}),\left( \psi \circ \phi _{1}\right) (\bar{a})\right)
<\varepsilon $;

\item $\left( \mathcal{C}_{n},d_{\mathcal{C}}\right) $ is a separable and
complete metric space for every $n\in \mathbb{N}$.
\end{itemize}

Since $\mathcal{A}$ is by assumption closed under substructures, $\mathcal{C}
$ satisfies the hereditary property. The joint embedding property is proved
by taking binary products. Lemma \ref{Lemma:pushout} shows that $\mathcal{C}$
satisfies the near amalgamation property. To conclude the proof it remains
to show that $\left( \mathcal{C}_{n},d_{\mathcal{C}}\right) $ is a separable
and complete metric space.

Suppose that $\bar{a},\bar{b}$ are $n$-marked structures in $\mathcal{A}$.
Recall that $\partial \left( \bar{a},\bar{b}\right) $ is by definition 
\begin{equation*}
\inf_{f}\max \left\{ I(f),d\left( f(\bar{a}),\bar{b}\right) \right\}
\end{equation*}
where $f$ ranges among all the morphisms $f:\left\langle \bar{a}%
\right\rangle \rightarrow \left\langle \bar{b}\right\rangle $. It follows
from Lemma \ref{Lemma:pushout} that%
\begin{equation}
d_{\mathcal{C}}\left( \bar{a},\bar{b}\right) \leq \varpi \left( \partial
\left( \bar{a},\bar{b}\right) \right) +\partial \left( \bar{a},\bar{b}%
\right) \text{.\label{Equation:deinde}}
\end{equation}%
Furthermore it follows from the assumptions on basic tuples from Subsection %
\ref{Subsection:basic} that%
\begin{equation}
\partial \left( \bar{a},\bar{b}\right) \leq \rho _{\bar{a}}\left( d_{%
\mathcal{C}}\left( \bar{a},\bar{b}\right) \right) \text{\label{Equation:dc}.}
\end{equation}%
Let $\left( A_{i}\right) $ be an enumeration of the structures in $\mathcal{I%
}$.\ For any $i\in \mathbb{N}$ let $D_{i}\subset A_{i}$ be a countable
fundamental subset; see Definition \ref{Definition:fundamental}. Let $\left( 
\bar{a}_{i,k}\right) $ be an enumeration of all the basic $n$-tuples in $%
D_{i}$. It follows from the fact that $\mathcal{A}$ has enough injectives
from $\mathcal{I}$, Lemma \ref{Lemma:pushout}, and our assumptions on basic
tuples that if $\bar{b}$ is an $n$-marked structure in $\mathcal{A}$ and $%
\varepsilon >0$ then there exist $i,k\in \mathbb{N}$ such that $\partial
\left( \bar{b},\bar{a}_{ik}\right) <\varepsilon $. Together with Equation %
\eqref{Equation:deinde} this shows that $\left\{ \bar{a}_{i,k}:i,k\in 
\mathbb{N}\right\} $ is dense in $\left( \mathcal{C}_{n},d_{\mathcal{C}%
}\right) $.

Suppose now that $\left( \bar{a}_{j}\right) $ is a Cauchy sequence in $%
\left( \mathcal{C}_{n},d_{\mathcal{C}}\right) $. Using Lemma \ref%
{Lemma:pushout} and the fact that $\mathcal{A}$ is closed under limits of
direct sequences with embeddings as connective maps one can show that there
exists a structure $X$ in $\mathcal{A}$ and embeddings $\phi
_{j}:\left\langle \bar{a}_{j}\right\rangle \rightarrow X$ such that $\left(
\phi _{j}\left( \bar{a}_{j}\right) \right) $ is a Cauchy sequence in $X^{n}$
with max distance. If $\bar{a}$ is a limit of such a sequence in $X^{n}$,
then it is clear that $\bar{a}$ is a limit of $\left( \bar{a}_{j}\right) $
in $\left( \mathcal{C}_{n},d_{\mathcal{C}}\right) $. This concludes the
proof that $\left( \mathcal{C}_{n},d_{\mathcal{C}}\right) $ is complete, and 
$\mathcal{C}$ is a Fra\"{\i}ss\'{e} class. In the following subsections we
will give an independent proof of existence and uniqueness of the Fra\"{\i}ss%
\'{e} limit of $\mathcal{C}$ in the sense of \cite[Definition 3.15]%
{ben_yaacov_fraisse_2015}; see also \cite[Corollary 3.20]%
{ben_yaacov_fraisse_2015}.

It follows from the fact that $\left( \mathcal{C}_{n},d_{\mathcal{C}}\right) 
$ is separable and our assumptions on basic sequences that the class of
finitely generated structures in $\mathcal{A}$ is separable with respect to
the Gromov-Hausdorff distance introduced in Definition \ref{Definition:GH}.

\subsection{Fra\"{\i}ss\'{e} limit: existence\label{Subsection:existence}}

Here we want to give a direct proof---not relying on the general results
from \cite{ben_yaacov_fraisse_2015}---of existence of the limit of the class 
$\mathcal{C}$ of finitely generated structures in $\mathcal{A}$. Precisely
we will prove that there exists a separable structure $M$ in $\mathcal{A}$
that satisfies the following \emph{approximate extension property }with
modulus $\varpi $: if $E=\left\langle \bar{a}\right\rangle $ and $F$ are
finitely generated structures in $\mathcal{A}$, $\varepsilon >0$, $\phi
:E\rightarrow F$ and $f:E\rightarrow M$ are morphisms such that $\max
\left\{ I(\phi ),I(f)\right\} <\delta $, then there exists a morphism $%
g:F\rightarrow M$ such that $I\left( g\right) <\varepsilon $ and $d\left(
g\circ \phi ,f\right) <\varpi (\delta )$. It is easy to see using \cite[%
Corollary 3.20]{ben_yaacov_fraisse_2015} that a structure $M$ satisfying the
approximate extension property is a limit of $\mathcal{C}$ in the sense of 
\cite[Definition 3.15]{ben_yaacov_fraisse_2015}. Furthermore the proof will
show that $M$ can be realized as the limit of an inductive sequence of
elements of $\mathcal{I}$ with embeddings as connective maps.

Let us say that a subset $D$ of a metric space $A$ is $\varepsilon $-dense
for some $\varepsilon >0$ if every element of $A$ is at distance at most $%
\varepsilon $ from some element of $D$. Let $\left( X_{m}\right) $ be a
sequence of finitely generated structures in $\mathcal{A}$ that is dense
with respect to the Gromov-Hausdorff distance. Let $\left( A_{d}\right) $ be
an enumeration of the structures in $\mathcal{I}$. For every $m,d,k\in 
\mathbb{N}$ let $\mathcal{E}_{m,d,k}$ be a finite $2^{-k}$-dense set of
morphisms from $X_{m}$ to $A_{d}$. Using Lemma \ref{Lemma:NAP} one can
define by recursion on $k\in \mathbb{N}$ sequences $\left( d_{k}\right)
,\left( j_{k}\right) ,\left( \mathcal{F}_{m,k}\right) $ such that

\begin{enumerate}
\item $d_{k}\in \mathbb{N}$,

\item $j_{k}:A_{d_{k}}\rightarrow A_{d_{k+1}}$ is an embedding, and

\item $\mathcal{F}_{m,k}$ is a finite $2^{-k}$-dense subset of the space of
morphisms from $X_{m}$ to $A_{d_{k}}$,
\end{enumerate}

such that for every $m,d\leq k$, $f\in \mathcal{F}_{m,k}$, and $\phi \in 
\mathcal{E}_{m,d,k}$ there exists $\widehat{f}:A_{d}\rightarrow A_{d_{k+1}}$
such that $d(\widehat{f}\circ \phi ,j_{k}\circ f)\leq \varpi \left( \max
\left\{ I(f),I(\phi )\right\} \right) $.

One can define now $M$ to be the limit of the inductive sequence $\left(
A_{d_{k}}\right) $ with connective maps $j_{k}:A_{d_{k}}\rightarrow
A_{d_{k+1}}$. It is not difficult to verify that $M$ satisfies the
approximate extension property using the assumption that $\mathcal{A}$ has
enough injectives from $\mathcal{I}$ together with our hypotheses on basic
sequences.

\subsection{Fra\"{\i}ss\'{e} limit: uniqueness and stable homogeneity\label%
{Subsection:stable}}

In this section we want to prove that the Fra\"{\i}ss\'{e} limit $M$ of the
class of finitely generated structures in $\mathcal{A}$ is stably
homogeneous with modulus $\varpi $ in the sense of Definition \ref%
{Definition:stable-homogeneity}. The argument is analogous to the one of 
\cite[Theorem 1.1]{kubis_proof_2013}.

\begin{proposition}
\label{Proposition:stable-homogeneity}Let $M$ be the limit of the class of
finitely generated structures in $\mathcal{A}$ as constructed in Subsection %
\ref{Subsection:existence}. Suppose that $E$ is a finitely generated
structure in $\mathcal{A}$, $\phi :E\rightarrow M$ and $f:E\rightarrow M$
are morphisms such that $I(f)<\delta $ and $I(\phi )<\delta $. Then there
exists an automorphism $\alpha $ of $M$ such that $d\left( \alpha \circ \phi
,f\right) <\varpi (\delta )$.
\end{proposition}

Fix $\eta ,\delta _{0}>0$ such that $\varpi \left( \delta _{0}\right) +\eta
<\varpi (\delta )$, $I(f)<\delta _{0}$, and $I(\phi )<\delta _{0}$. Using
the property of $M$ established in Subsection \ref{Subsection:existence} one
can easily define by recursion on $n$ increasing sequences $\left(
X_{n}\right) $ and $\left( Y_{n}\right) $ of substructures of $M$ with dense
union, $\delta _{n}>0$, morphisms $\alpha _{n}:X_{n}\rightarrow Y_{n}$ and $%
\beta _{n}:Y_{n}\rightarrow X_{n+1}$, such that

\begin{enumerate}
\item $X_{1}\supset \phi \left[ E\right] $, $Y_{1}\supset f\left[ E\right] $%
, and $d\left( \alpha _{1}\circ \phi ,f\right) <\delta _{0}$,

\item $\varpi \left( \delta _{n}\right) <2^{-\left( n+1\right) }\eta $,

\item $I\left( \alpha _{n}\right) <\delta _{n}$ and $I\left( \beta
_{n}\right) <\delta _{n}$,

\item $d\left( \alpha _{n+1}\circ \beta _{n},i_{Y_{n}}\right) <\varpi \left(
\delta _{n}\right) $ and $d\left( \beta _{n}\circ \alpha
_{n},i_{X_{n}}\right) <\varpi \left( \delta _{n}\right) $.
\end{enumerate}

In (4) $i_{X_{n}}$ denotes the inclusion map of $X_{n}$ into $X$, and $%
i_{Y_{n}}$ denotes the inclusion map from $Y_{n}$ into $Y$. It then follows
from (3) and (4) that 
\begin{equation*}
d(\alpha _{n},\left( \alpha _{n+1}\right) |_{X_{n}})<2\varpi \left( \delta
_{n}\right) <2^{-n}\eta
\end{equation*}%
and similarly for $\beta _{n}$ and $\left( \beta _{n+1}\right) |_{Y_{n}}$.
Therefore the sequences $\left( \alpha _{n}\right) $ and $\left( \beta
_{n}\right) $ induce morphisms $\alpha :M\rightarrow M$ and $\beta
:M\rightarrow M$. By (3) $\alpha $ and $\beta $ are embeddings, and by (4)
they are inverse of each other. Finally by (1) and (2)%
\begin{equation*}
d(\alpha \circ \phi ,f)<\varpi \left( \delta _{0}\right) +\eta
\sum_{n=1}^{+\infty }2^{-n}<\varpi (\delta )\text{.}
\end{equation*}%
This concludes the proof. The same argument show that there exists a unique
separable structure in $\mathcal{A}$ that satisfies the approximate
extension property from Subsection \ref{Subsection:existence}. A one-sided
version of the proof above can be used to prove that any separable structure
in $\mathcal{A}$ embeds into $M$.

\subsection{Fra\"{\i}ss\'{e} limits: characterization}

It turns out that there are several seemingly different properties that
characterize the Fra\"{\i}ss\'{e} limit $M$ up to isomorphism.

\begin{proposition}
\label{Proposition:characterize-limit}Suppose that $M$ is a separable
structure in $\mathcal{A}$. The following statements are equivalent.

\begin{enumerate}
\item $M$ is the limit of $\mathcal{C}$;

\item For every finitely generated structure $F$ in $\mathcal{A}$, there
exists an embedding from $F$ to $M$, and for any $\delta >0$, and morphisms $%
\phi :F\rightarrow M$ and $\psi :F\rightarrow M$ such that $I(\phi )<\delta $
and $I\left( \psi \right) <\delta $ there exists an automorphism $\alpha $
of $M$ such that $d\left( \alpha \circ \phi ,\psi \right) <\varpi (\delta )$;

\item For any finitely generated structures $E,F$ in $\mathcal{A}$, $\delta
,\varepsilon >0$, morphisms $\phi :E\rightarrow F$ and $f:E\rightarrow M$
such that $I(\phi )<\delta $ and $I(f)<\delta $ there exists a morphism $%
g:F\rightarrow M$ such that $d\left( g\circ \phi ,f\right) <\varpi (\delta )$
and $I\left( g\right) <\varepsilon $;

\item For any finitely generated structure $F$ in $\mathcal{A}$, tuple $\bar{%
a}$ in $F$, embedding $\phi :\left\langle \bar{a}\right\rangle \rightarrow M$%
, and $\varepsilon >0$, there exists an embedding $\psi :F\rightarrow M$
such that $d\left( \psi (\bar{a}),\phi (\bar{a})\right) <\varepsilon $;

\item For any finitely generated structures $E$ in $\mathcal{A}$ and $F$ in $%
\mathcal{I}$, $\varepsilon >0$, embeddings $f:E\rightarrow M$ and $\phi
:E\rightarrow F$ there exists an embedding $g:F\rightarrow M$ such that $%
d\left( g\circ \phi ,f\right) <\varepsilon $;

\item Suppose that $A\in \mathcal{I}$, $\bar{a}$ is a finite tuple in a
fixed countable fundamental subset $D_{A}$ of $A$, and $f:\left\langle \bar{a%
}\right\rangle \rightarrow M$ is a morphism belonging to a fixed countable
uniformly dense collection of morphisms from $\left\langle \bar{a}%
\right\rangle $ to $M$. If $I(f)\leq \delta $, then there exists a morphism $%
g:A\rightarrow M$ such that $d\left( g(\bar{a}),f(\bar{a})\right) <\varpi
(\delta )$ and $I\left( g\right) <\varepsilon $.
\end{enumerate}
\end{proposition}

\begin{proof}
The equivalence of (1) and (4) follows from \cite[Corollary 3.20]%
{ben_yaacov_fraisse_2015}. The argument of Subsection \ref{Subsection:stable}
gives a proof of (3)$\Rightarrow $(2), while the implication (2)$\Rightarrow 
$(3) is obvious. Since clearly (2) implies (4), Proposition \ref%
{Proposition:stable-homogeneity} together with uniqueness of the limit shows
that (4) and (2) are in fact equivalent. A similar proof as the one in
Subsection \ref{Subsection:stable} shows that any two separable structures
satisfying (5) are isomorphic. This gives the implication (5)$\Rightarrow $%
(3), while the converse implication is obvious. The fact that $\mathcal{A}$
has enough injectives from $\mathcal{I}$ and our hypotheses on basic
sequences show that (6) implies (4), while the converse implication is
obvious.
\end{proof}

\section{Retracts of the limit\label{Section:retracts}}

\subsection{Approximate injectivity and retracts\label{Subsection:retracts}}

Suppose that $A$ is an $\mathcal{L}$-structure. A \emph{retraction }$\pi $
of $A$ is a morphisms $\pi :A\rightarrow A$ that is \emph{idempotent}, that
is\ $\pi \circ \pi =\pi $. A \emph{retract }of $A$ is the image of $A$ under
a retraction. Suppose that $\mathcal{A}$ is a class of $\mathcal{L}$%
-structures satisfying all the assumptions from Section \ref%
{Section:injective}. In the following we will characterize (up to
isomorphism) the retracts of the Fra\"{\i}ss\'{e} limit $M$ of the class of
finitely generated structures from $\mathcal{A}$. The same proof as Lemma %
\ref{Lemma:pushout} gives the following lemma.

\begin{lemma}
\label{Lemma:pushout2}Suppose that $X,\widehat{X},Y$ are structures in $%
\mathcal{A}$, $\bar{a}$ is a tuple in $X$, $\phi :X\rightarrow \widehat{X}$
and $f:X\rightarrow Y$ are morphisms such that $I(\phi )<\delta $. Then
there exists a structure $\widehat{Y}$, a morphism $\widehat{f}:\widehat{X}%
\rightarrow \widehat{Y}$, and an embedding $j:Y\rightarrow \widehat{Y}$ such
that $d((\widehat{f}\circ \phi )(\bar{a}),(j\circ f)(\bar{a}))\leq \varpi
(\delta )$ and furthermore for any $Z\in \mathcal{A}$ and morphisms $g:%
\widehat{X}\rightarrow Z$ and $h:\widehat{Y}\rightarrow Z$ such that $%
d((g\circ \phi )(\bar{a}),(h\circ f)(\bar{a}))\leq \varpi (\delta )$ there
exists a morphism $\tau :\widehat{Y}\rightarrow Z$ such that $g=\tau \circ 
\widehat{f}$ and $h=\tau \circ j$. If moreover $I(f)<\delta $ then $\widehat{%
f}$ is an embedding. If $X,\widehat{X},Y$ are finitely generated, then $%
\widehat{Y}$ is finitely generated.
\end{lemma}

The structure $\widehat{Y}$ in Lemma \ref{Lemma:pushout2} together with the
canonical morphisms $\widehat{f}:\widehat{X}\rightarrow \widehat{Y}$ and $%
j:Y\rightarrow \widehat{Y}$ will be called the\emph{\ approximate pushout}
of $f$ and $\phi $ over $\bar{a}$ with tolerance $\varpi (\delta )$. One can
similarly define the approximate pushout of a finite sequences of maps $%
f_{i}:X\rightarrow Y_{i}$ and $\phi _{i}:X\rightarrow \widehat{X}_{i}$ over $%
\bar{a}\subset X$ with tolerance $\varpi \left( \delta _{i}\right) $ for $%
i=1,2,\ldots ,k$.

\begin{definition}
\label{Definition:approx-inj}Suppose that $X$ is an $\mathcal{L}$-structure
in $\mathcal{A}$. We say that $X$ is \emph{approximately injective} if
whenever $A$ is a structure in $\mathcal{I}$, $\bar{a}$ is a tuple in $A$, $%
f:\left\langle \bar{a}\right\rangle \rightarrow X$ is a morphism, and $%
\varepsilon >0$, there exists a morphism $g:A\rightarrow X$ such that $d(g(%
\bar{a}),f(\bar{a}))\leq \varepsilon $.
\end{definition}

As observed in Subsection \ref{Subsection:existence}, the Fra\"{\i}ss\'{e}
limit $M$ of the class of finitely generated structures in $\mathcal{A}$ can
be realized as the limit of an inductive sequence of elements of $\mathcal{I}
$ with embeddings as connective maps. It follows from this fact and
injectivity of elements of $\mathcal{I}$ that $M$ is approximately
injective. Therefore any retract of $M$ is approximately injective as well.
The next theorem shows that, conversely, any approximately injective
separable structure in $\mathcal{A}$ is isomorphic as $\mathcal{L}$%
-structure to a retract of $M$.

\begin{theorem}
\label{Theorem:retracts}Let $M$ denote the Fra\"{\i}ss\'{e} limit of the
class of finitely generated structures in $\mathcal{A}$. A separable
structure $X$ in $\mathcal{A}$ is approximately injective if and only if
there exist an embedding $\phi :X\rightarrow M$ and an idempotent morphism $%
\pi :M\rightarrow M$ such that the range of $\phi $ coincides with the range
of $\pi $.
\end{theorem}

Theorem \ref{Theorem:retracts} can be proved using the construction of
approximate pushouts as in Lemma \ref{Lemma:pushout2}. We omit the proof,
since we will prove a more general result in Section \ref{Section:general}.
An alternative proof of Theorem \ref{Theorem:retracts} follows from the
results of Subsection \ref{Subsection:universal-state}. Similar
characterizations of retracts of Fra\"{\i}ss\'{e} limits have been obtained
by Dolinka in the countable case \cite{dolinka_characterization_2012} and by
Kubi\'{s} in \cite{kubis_injective_2015}.

\subsection{Approximate injectivity and nuclearity\label%
{Subsection:nuclearity}}

We consider now a notion of $\mathcal{I}$-nuclearity for structures in $%
\mathcal{A}$; see Definition \ref{Definition:nuclear}. The term $\mathcal{I}$%
-nuclear is inspired by the characterization of nuclearity for unital
C*-algebras and operator systems in terms of the completely positive
approximation property; see \cite{han_approximation_2011} and \cite[Section
2.3]{brown_c*-algebras_2008}.

\begin{definition}
\label{Definition:nuclear}A structure $X$ in $\mathcal{A}$ is\emph{\ }$%
\mathcal{I}$\emph{-nuclear} if there exist nets $\left( \gamma _{i}\right) $
and $\left( \rho _{i}\right) $ of morphisms $\gamma _{i}:X\rightarrow A_{i}$
and $\rho _{i}:A_{i}\rightarrow X$ such that $A_{i}\in \mathcal{I}$ and $%
\rho _{i}\circ \gamma _{i}$ converges pointwise to the identity map of $X$.
\end{definition}

We now prove that $\mathcal{I}$-nuclearity is equivalent to approximate
injectivity. If $f,g:E\rightarrow F$ are functions between $\mathcal{L}$%
-structures, and $\bar{a}$ is an $n$-tuple in $E$, we write $f\approx _{\bar{%
a},\varepsilon }g$ to express the fact that $d\left( f(\bar{a}),g(\bar{a}%
)\right) \leq \varepsilon $.

\begin{proposition}
\label{Proposition:nuclear}Suppose that $X$ is a structure in $\mathcal{A}$.
The following assertions are equivalent:

\begin{enumerate}
\item $X$ is approximately injective;

\item $X$ is $\mathcal{I}$-nuclear;

\item Whenever $E,F$ are finitely generated structures in $\mathcal{A}$, $%
\bar{a}$ is a finite tuple in $E$, $\phi :E\rightarrow F$ and $%
f:E\rightarrow X$ are morphisms such that $I(\phi )<\delta $, there exists a
morphism $g:F\rightarrow X$ such that $g\circ \phi \approx _{\bar{a},\varpi
(\delta )}f$.
\end{enumerate}
\end{proposition}

\begin{proof}
We present the proofs of the nontrivial implications below.

(1)$\Rightarrow $(2): If $X$ is approximately injective, then by Theorem \ref%
{Theorem:retracts} $X$ is isomorphic to a retract of the Fra\"{\i}ss\'{e}
limit $M$ of the class of finite-dimensional structures in $\mathcal{A}$.
Therefore it is enough to prove that $M$ is $\mathcal{I}$-nuclear. Recall
that $M$ contains an increasing sequence $\left( B_{n}\right) $ of
structures from $\mathcal{I}$ with dense union. Therefore it is enough to
prove that if $\bar{a}\subset B_{n}\subset M$ is a finite tuple and $%
\varepsilon >0$, then there exist morphisms $\gamma :M\rightarrow B_{n}$ and 
$\rho :B_{n}\rightarrow M$ such that $\left( \rho \circ \gamma \right) (\bar{%
a})=\bar{a}$. Consider the identity map of $B_{n}$ and observe that by
injectivity of $B_{n}$ it extends to a morphism $\gamma :M\rightarrow B_{n}$%
. Let now $\rho :B_{n}\rightarrow M$ be the inclusion map and observe that $%
\left( \gamma \circ \rho \right) (\bar{a})=\bar{a}$.

(2)$\Rightarrow $(3): Let $E,F,\bar{a},\phi ,f$ be as in (3). Let $\delta
_{0}>0$ be such that $I(\phi )<\delta _{0}<\delta $. Fix also $\varepsilon
>0 $ such that $\varpi \left( \delta _{0}\right) +\varepsilon <\varpi
(\delta )$. By assumption there exist $A\in \mathcal{I}$ and morphisms $%
\gamma :X\rightarrow A$ and $\rho :A\rightarrow X$ such that $\rho \circ
\gamma \circ f\approx _{\bar{a},\varepsilon }f$. Since $\mathcal{A}$ has
enough injectives from $\mathcal{I}$ with modulus $\varpi $ there exists a
morphism $h:F\rightarrow A$ such that $d\left( h\circ \phi ,f\right) \leq
\varpi \left( \delta _{0}\right) $. Set $g=\rho \circ h$ and observe that $%
d((g\circ \phi )(\bar{a}),f(\bar{a}))\leq \varpi \left( \delta _{0}\right)
+\varepsilon <\varpi (\delta )$.
\end{proof}

\subsection{$\mathcal{I}$-structures\label{Subsection:I-structure}}

In the following if $A,B$ are subsets of a structure $X$, we write $A\subset
_{\varepsilon }B$ if every element of $A$ is at distance at most $%
\varepsilon $ from some element of $B$.

\begin{definition}
\label{Definition:I-structure}We say that structure $X$ in $\mathcal{A}$ is
an $\mathcal{I}$\emph{-structure} if for every finitely generated
substructure $E$ of $X$ and $\varepsilon >0$ there exist a finitely
generated substructure $B$ of $X$ containing $E$, a structure $\widehat{B}$
in $\mathcal{I}$ such that $d(B,\widehat{B})<\varepsilon $. We say that $X$
is a \emph{rigid} $\mathcal{I}$\emph{-structure} if for every finite subset $%
x_{1},\ldots ,x_{n}$ of $X$ there exists a substructure $A$ of $X$ that
belongs to $\mathcal{I}$ such that $\left\{ x_{1},\ldots ,x_{n}\right\}
\subset _{\varepsilon }A$ .
\end{definition}

Every structure in $\mathcal{A}$ that can be represented as the direct limit
of elements of $\mathcal{I}$ with embeddings as connective maps is clearly a
rigid $\mathcal{I}$-structure. Particularly, the Fra\"{\i}ss\'{e} limit $M$
of finite-dimensional structures in $\mathcal{A}$ is a rigid $\mathcal{I}$%
-structure. In turn, it follows from injectivity of elements of $\mathcal{I}$
together with the fact that $\mathcal{A}$ has enough injectives from $%
\mathcal{I}$ and our assumptions on basic sequences that any rigid $\mathcal{%
I}$-structure is an $\mathcal{I}$-structure, and that an $\mathcal{I}$%
-structure is approximately injective. The following proposition provides a
characterization among the (rigid) $\mathcal{I}$-structures of the Fra\"{\i}%
ss\'{e} limit of the class of finitely generated structures in $\mathcal{A}$.

\begin{proposition}
\label{Proposition:characterize-limit-I-structure}Let $X$ be a separable
structure in $\mathcal{A}$. The following statements are equivalent:

\begin{enumerate}
\item $X$ is the Fra\"{\i}ss\'{e} limit $M$ of the class of finitely
generated structures in $\mathcal{A}$;

\item $M$ is an $\mathcal{I}$-structure, and for any $\delta ,\varepsilon >0$%
, structures $A,\widehat{A}\in \mathcal{I}$, embedding $\phi :A\rightarrow 
\widehat{A}$, and morphism $f:A\rightarrow X$ such that $I(f)<\delta $,
there exists a morphism $\widehat{f}:\widehat{A}\rightarrow X$ such that $I(%
\widehat{f})<\varepsilon $ and $d(\widehat{f}\circ \phi ,f)<\varpi (\delta )$%
;

\item $M$ is a rigid $\mathcal{I}$-structure, and for any structures $A,%
\widehat{A}\in \mathcal{I}$, embeddings $\phi :A\rightarrow \widehat{A}$ and 
$f:A\rightarrow X$, and $\varepsilon >0$, there exists a morphism $\widehat{f%
}:\widehat{A}\rightarrow X$ such that $I(\widehat{f})<\varepsilon $ and $d(%
\widehat{f}\circ \phi ,f)<\varepsilon $.
\end{enumerate}
\end{proposition}

\begin{proof}
The implications (1)$\Rightarrow $(2) and (1)$\Rightarrow $(3) follow from
Proposition \ref{Proposition:characterize-limit} and the already observed
fact that the Fra\"{\i}ss\'{e} limit of the class of finitely generated
structures in $\mathcal{A}$ is a rigid $\mathcal{I}$-structure.

We now prove that (2) implies (1). Fix a countable fundamental subset $D_{X}$
of $X$ as in Definition \ref{Definition:fundamental} and a sequence $\left(
\delta _{n}\right) $ of strictly positive real numbers such that $%
\sum_{n}\varpi \left( 2\delta _{n}\right) <+\infty $. Using the hypothesis,
and proceeding as in Subsection \ref{Subsection:existence}, one can define
by recursion on $n$:

\begin{itemize}
\item structures $B_{n},\widetilde{C}_{n}\in \mathcal{I}$, and substructures 
$C_{n}$ of $X$,

\item morphisms $\alpha _{n}:B_{n}\rightarrow C_{n}$, $f_{n}:C_{n}%
\rightarrow \widetilde{C}_{n}$, $g_{n}:\widetilde{C}_{n}\rightarrow C_{n}$,
and embeddings $\beta _{n}:C_{n}\rightarrow B_{n+1}$ and $\phi
_{n}:B_{n}\rightarrow B_{n+1}$,
\end{itemize}

such that

\renewcommand{\labelenumi}{(\alph{enumi})}

\begin{enumerate}
\item $\left\{ x_{1},\ldots ,x_{n}\right\} \subset _{\delta _{n}}C_{n}$,

\item $I\left( \alpha _{n}\right) <\delta _{n}$, $I\left( f_{n}\right)
<\delta _{n}$, $I\left( g_{n}\right) <\delta _{n}$,

\item $d(f_{n}\circ g_{n},id_{\widetilde{C}_{n}})<\delta _{n}$, $d\left(
g_{n}\circ f_{n},id_{C_{n}}\right) <\delta _{n}$, $d\left( \beta _{n}\circ
\alpha _{n},\phi _{n}\right) <\varpi \left( \delta _{n}\right) $, and $%
d\left( \alpha _{n+1}\circ \beta _{n},g_{n}\right) <\varpi \left( 2\delta
_{n}\right) $, where $\iota _{n}:C_{n}\rightarrow X$ is the inclusion map,
and

\item the limit of the inductive sequence $\left( B_{n}\right) $ with
connective maps $\phi _{n}$ is the Fra\"{\i}ss\'{e} limit $M$ of the class
of finitely generated structures in $\mathcal{A}$.
\end{enumerate}

Suppose that we have defined $B_{k},\alpha _{k},C_{k},\widetilde{C}%
_{k},f_{k},g_{k},\phi _{k-1},\beta _{k-1}$ for $k\leq n$. Proceeding as in
Subsection \ref{Subsection:existence} one can define a structure $B_{n+1}\in 
\mathcal{I}$ and an embedding $\phi _{n}:B_{n}\rightarrow B_{n+1}$
satisfying all the requirements of the $n$-th step of Subsection \ref%
{Subsection:existence}. Using the recursion hypothesis, we can moreover
guarantee that there exists a morphism $\beta _{n}:\widetilde{C}%
_{n}\rightarrow B_{n+1}$ such that $d\left( \beta _{n}\circ f_{n}\circ
\alpha _{n},\phi _{n}\right) <\varpi \left( \delta _{n}\right) $. We apply
now the hypothesis to $g_{n}\circ \beta _{n}^{-1}$ to define a morphism $%
\alpha _{n+1}:B_{n+1}\rightarrow X$ such that $I\left( \alpha _{n+1}\right)
<\delta _{n}$ and $d\left( \alpha _{n+1}\circ \beta _{n},g_{n}\right)
<\varpi \left( 2\delta _{n}\right) $. Define $C_{n+1}$ to be the range of $%
\beta _{n+1}$. Finally one can obtain $\widetilde{C}_{n+1}$, $f_{n+1}$, and $%
g_{n+1}$ by applying the hypothesis that $X$ is an $\mathcal{I}$-structure.
This concludes the recursive construction. Granted the construction, the
sequences of morphisms $\left( \alpha _{k}\right) $ induces at the limit a
morphisms $\alpha :M\rightarrow X$. Such a morphism is well defined by (c),
it is an embedding by (b), and it is onto by (a) and (c).

We now prove that (3) implies (1). Fix a dense sequence $\left( x_{n}\right) 
$ of elements of $X$ and a sequence $\left( \delta _{n}\right) $ of strictly
positive real numbers that converges to $0$ fast enough. One can define by
recursion on $n$:

\begin{itemize}
\item structures $B_{n},C_{n}\in \mathcal{I}$ with $C_{n}\subset X$,

\item morphisms $\alpha _{n}:B_{n}\rightarrow C_{n}$ and embeddings $\beta
_{n}:C_{n}\rightarrow B_{n+1}$ and $j_{n}:B_{n}\rightarrow B_{n+1}$,
\end{itemize}

such that, if $\iota _{n}:C_{n}\rightarrow X$ is the inclusion map, then

\begin{enumerate}
\item $\left\{ x_{1},\ldots ,x_{n}\right\} \subset _{\delta _{n}}C_{n}$,

\item $I\left( \alpha _{n}\right) <\delta _{n}$,

\item $d\left( i_{n+1}\circ \alpha _{n},i_{n}\right) <\delta _{n}$, and $%
d\left( \beta _{n+1}\circ \alpha _{n},j_{n}\right) <\varpi \left( \delta
_{n}\right) $,

\item the limit of the inductive sequence $\left( B_{n}\right) $ with
connective maps $j_{n}:B_{n}\rightarrow B_{n+1}$ is isomorphic to the Fra%
\"{\i}ss\'{e} limit of the class of finitely generated structures in $%
\mathcal{A}$.
\end{enumerate}

\renewcommand{\labelenumi}{\arabic{enumerator}}

This can be seen proceeding as the proof of (2)$\Rightarrow $(1), using
furthermore the assumption that $X$ is a rigid $\mathcal{I}$-structure and
the construction of the approximate pushout from Lemma \ref{Lemma:NAP}.
\end{proof}

\section{Universal morphisms\label{Section:operators}}

Throughout this section and the next section we will use the same notation
and terminology as in Section \ref{Section:injective}. Particularly we will
suppose that $\mathcal{L}$ is a language in the logic for metric structures, 
$\mathcal{A}$ is a class of $\mathcal{L}$-structures, and $\mathcal{I}%
\subset \mathcal{A}$ is a countable collection of finitely generated
injective structures satisfying the assumptions of Theorem \ref%
{Theorem:abstract-nonsense} such that $\mathcal{A}$ has enough injectives
from $\mathcal{I}$ with modulus $\varpi $. Again we stick for simplicity to
the case when $\mathcal{L}$ is single-sorted.

\subsection{Rota universal operators\label{Subsection:Rota}}

In \cite{rota_note_1959,rota_models_1960} Rota constructed a surjective
contractive linear operator $\Omega $ on $\ell ^{2}$ which is a universal
model for bounded linear operators on separable Hilbert spaces. This means
that if $H_{0},H_{1}$ are separable Hilbert spaces and $T:H_{0}\rightarrow
H_{1}$ is a bounded linear operator, then there exist injective bounded
linear maps $\alpha _{0}:H_{0}\rightarrow \ell ^{2}$ and $\alpha
_{1}:H_{1}\rightarrow \ell ^{2}$ such that $\alpha _{1}\circ T=\Omega \circ
\alpha _{0}$. Clearly, it follows that when $H_{0}=H_{1}$ one can take $%
\alpha _{0}=\alpha _{1}$. An example of such an operator is the infinite
amplification on the unilateral shift on $\ell ^{2}$. Rota's original
motivation comes from the invariant subspace problem for operators on the
separable infinite-dimensional Hilbert space. Operators that are universal
in the sense of Rota have been characterized in \cite{caradus_universal_1969}%
, and are currently the subject of active research; see for example \cite%
{cowen_consequences_2015,cowen_introduction_2016}.

An analogue of Rota's universal operator for the class of operators on
arbitrary separable Banach spaces was constructed by Garbuli\'{n}ska-W\c{e}%
grzyn and Kubi\'{s} in \cite{garbulinska_universal_2015}. In this section we
will prove a general result concerning the existence of a \textquotedblleft
universal morphism\textquotedblright\ defined on the Fra\"{\i}ss\'{e} limit $%
M$ of a Fra\"{\i}ss\'{e} class $\mathcal{C}$ as in Theorem \ref%
{Theorem:abstract-nonsense}. As a consequence of our general results from
this an the next section, we can give an explicit characterization of the
universal operator constructed by Garbuli\'{n}ska-W\c{e}grzyn and Kubi\'{s};
see Theorem \ref{Theorem:operator-G} below and Subsection \ref%
{Subsection:Banach}.

A\emph{\ quotient mapping }$\phi :X\rightarrow Y$ between Banach spaces is a
linear function that sends the open unit ball of $X$ onto the open unit ball
of $Y$. This is equivalent to the assertion that the map $X/\mathrm{Ker}%
(\phi )\rightarrow Y$ induced by $\phi $ is a surjective linear isometry.
Recall that the Lusky simplex $\mathbb{L}$ is the unit ball of the dual
space of the Gurarij space $\mathbb{G}$. The definition of $M$-ideal in a
Banach space can be found in Subsection \ref{Subsection:Banach}; see also 
\cite{alfsen_structure_1972-1,alfsen_structure_1972-2}.

\begin{theorem}
\label{Theorem:operator-G}Suppose that $T:\mathbb{G}\rightarrow \mathbb{G}$
is a linear map of norm at most $1$, $N$ is the kernel of $T$, and $%
H=N^{\bot }\cap \mathbb{L}$. The following assertions are equivalent:

\begin{enumerate}
\item $T$ is a quotient mapping and $N$ is a nonzero $M$-ideal of $\mathbb{G}
$;

\item $T$ is a quotient mapping and $H$ is a closed proper biface of $%
\mathbb{L}$ symmetrically affinely homeomorphic to $\mathbb{L}$;

\item whenever $E_{0}\subset F_{0}$ and $E_{1}\subset F_{1}$ are
finite-dimensional Bananch spaces, $f_{0}:E_{0}\rightarrow \mathbb{G}$ and $%
f_{1}:E_{1}\rightarrow \mathbb{G}$ are linear isometries, $%
L:F_{0}\rightarrow F_{1}$ is a linear map of norm at most $1$ mapping $E_{0}$
to $E_{1}$ such that $T\circ f_{0}=f_{1}\circ L$, and $\varepsilon >0$, then
there exist linear isometries $\widehat{f}_{0}:F_{0}\rightarrow \mathbb{G}$
and $\widehat{f}_{1}:F_{1}\rightarrow \mathbb{G}$ such that $\left\Vert
T\circ \widehat{f}_{0}-\widehat{f}_{1}\circ L\right\Vert <\varepsilon $.
\end{enumerate}

The set of operators satisfying the equivalent conditions above is a dense $%
G_{\delta }$ subset of the space $\mathrm{\mathrm{Ball}}\left( B(\mathbb{G}%
)\right) $ of linear operators on $\mathbb{G}$ of norm at most $1$, and
forms a single orbit under the action $\mathrm{Aut}(\mathbb{G}%
)\curvearrowright \mathrm{\mathrm{Ball}}\left( B(\mathbb{G})\right) $, $%
\left( \alpha ,S\right) \mapsto S\circ \alpha ^{-1}$. If $\Omega _{\mathbb{G}%
}:\mathbb{G}\rightarrow \mathbb{G}$ is such an operator, then the kernel of $%
\Omega _{\mathbb{G}}$ is isometrically isomorphic to $\mathbb{G}$. In other
words the sequence 
\begin{equation*}
0\longrightarrow \mathbb{G\longrightarrow G}\overset{\Omega _{\mathbb{G}}}{%
\longrightarrow }\mathbb{G}\longrightarrow 0
\end{equation*}%
where the first arrow is a linear isometry, is exact. Furthermore $\Omega _{%
\mathbb{G}}$ is a universal operator between separable Banach spaces, in the
sense that any if $L:E_{0}\rightarrow E_{1}$ is a linear map of norm at most 
$1$ between separable Banach spaces, then there exist linear isometries $%
\eta _{0}:E_{0}\rightarrow \mathbb{G}$ and $\eta _{1}:E_{1}\rightarrow 
\mathbb{G}$ such that $\Omega _{\mathbb{G}}\circ \eta _{0}=\eta _{1}\circ L$.
\end{theorem}

A similar result holds for complex scalars; see \ref%
{Subsection:Banach-complex}. We will also prove in Subsection \ref%
{Subsection:function-system} the analogous statement for the space of affine
functions on the Poulsen simplex.

\begin{theorem}
\label{Theorem:operator-A(P)}Suppose that $T:A(\mathbb{P})\rightarrow A(%
\mathbb{P})$ is a unital positive linear map, $N$ is the kernel of $T$, and $%
H=N^{\bot }\cap \mathbb{P}$. The following assertions are equivalent:

\begin{enumerate}
\item $T$ is a quotient mapping and $N$ is a nonzero $M$-ideal of $A(\mathbb{%
P})$;

\item $T$ is a quotient mapping and $H$ is a closed proper face of $\mathbb{P%
}$;

\item whenever $E_{0}\subset F_{0}$ and $E_{1}\subset F_{1}$ are
finite-dimensional function systems, $f_{0}:E_{0}\rightarrow A(\mathbb{P})$
and $f_{1}:E_{1}\rightarrow A(\mathbb{P})$ are unital linear isometries, $%
L:F_{0}\rightarrow F_{1}$ is a unital positive linear function mapping $%
E_{0} $ to $E_{1}$ such that $T\circ f_{0}=f_{1}\circ L$, and $\varepsilon
>0 $, then there exist unital linear isometries $\widehat{f}%
_{0}:F_{0}\rightarrow A(\mathbb{P})$ and $\widehat{f}_{1}:F_{1}\rightarrow A(%
\mathbb{P})$ such that $\left\Vert T\circ \widehat{f}_{0}-\widehat{f}%
_{1}\circ L\right\Vert <\varepsilon $.
\end{enumerate}

The set of unital positive linear maps satisfying the equivalent conditions
above is a dense $G_{\delta }$ subset of the space $\mathrm{UP}\left( A(%
\mathbb{P})\right) $ of unital positive linear maps on $A(\mathbb{P})$, and
forms a single orbit under the action $\mathrm{Aut}\left( A(\mathbb{P}%
)\right) \curvearrowright \mathrm{UP}\left( A(\mathbb{P})\right) $, $\left(
\alpha ,S\right) \mapsto S\circ \alpha ^{-1}$. If $\Omega _{A(\mathbb{P})}:A(%
\mathbb{P})\rightarrow A(\mathbb{P})$ is such an operator, then the set 
\begin{equation*}
\left\{ x\in A(\mathbb{P}):\Omega _{A(\mathbb{P})}(x)\text{ is a scalar
multiple of the identity}\right\}
\end{equation*}%
is a function system unitally isometrically isomorphic to $A(\mathbb{P})$.
\end{theorem}

As a further application of the general results of this section we will
obtain the existence of a noncommutative analog of the Garbuli\'{n}ska-W\c{e}%
grzyn--Kubi\'{s} operator defined on the noncommutative Gurarij space \cite%
{oikhberg_non-commutative_2006,lupini_uniqueness_2016}; see Subsection \ref%
{Subsection:ospaces}.

\begin{theorem}
\label{Theorem:universal-cc}There exists a complete quotient mapping $\Omega
_{\mathbb{NG}}:\mathbb{NG}\rightarrow \mathbb{NG}$ such that, if $%
T:X\rightarrow Y$ is a completely contractive linear map between separable
exact operator spaces, then there exist completely isometric linear maps $%
\alpha _{0}:X\rightarrow \mathbb{NG}$ and $\alpha _{1}:Y\rightarrow \mathbb{%
NG}$ such that $\alpha _{1}\circ T=\Omega _{\mathbb{NG}}\circ \alpha _{0}$.
Furthermore $\Omega _{\mathbb{NG}}$ is generic in the sense that the orbit $%
\left\{ \Omega _{\mathbb{NG}}\circ \beta :\beta \in \mathrm{Aut}(\mathbb{NG}%
)\right\} $ with respect to the continuous action $\mathrm{Aut}(\mathbb{NG}%
)\curvearrowright \mathrm{Ball}\left( B(\mathbb{NG})\right) $, $\left(
\alpha ,T\right) \mapsto T\circ \alpha ^{-1}$ is a dense $G_{\delta }$
subspace of the space $\mathrm{Ball}\left( B(\mathbb{NG})\right) $ of linear
complete contractions on $\mathbb{NG}$ endowed with the topology of
pointwise convergence. The kernel of $\Omega _{\mathbb{NG}}$ is completely
isometric to $\mathbb{NG}$. In other words there exists an exact sequence%
\begin{equation*}
0\longrightarrow \mathbb{NG\longrightarrow NG}\overset{\Omega _{\mathbb{NG}}}%
{\longrightarrow }\mathbb{NG}\longrightarrow 0
\end{equation*}%
where the second arrow is a linear complete isometry. A completely
contractive linear map $T:\mathbb{NG}\rightarrow \mathbb{NG}$ belongs to the 
$\mathrm{Aut}(\mathbb{NG})$-orbit of $\Omega _{\mathbb{NG}}$ if and only if
it satisfies the following property: whenever $E_{0}\subset F_{0}$ and $%
E_{1}\subset F_{1}$ are finite-dimensional exact operator spaces, $%
f_{0}:E_{0}\rightarrow \mathbb{\mathbb{NG}}$ and $f_{1}:E_{1}\rightarrow 
\mathbb{\mathbb{NG}}$ are linear complete isometries, $L:F_{0}\rightarrow
F_{1}$ is a linear complete contraction mapping $E_{0}$ to $E_{1}$ such that 
$T\circ f_{0}=f_{1}\circ L$, and $\varepsilon >0$, then there exist linear
complete isometries $\widehat{f}_{0}:F_{0}\rightarrow \mathbb{\mathbb{NG}}$
and $\widehat{f}_{1}:F_{1}\rightarrow \mathbb{\mathbb{NG}}$ such that $%
\left\Vert T\circ \widehat{f}_{0}-\widehat{f}_{1}\circ L\right\Vert
_{cb}<\varepsilon $.
\end{theorem}

The same result holds in the operator systems category, yielding a universal
unital completely positive map $\Omega _{A(\mathbb{\mathbb{NP}})}$ defined
on the noncommutative Poulsen system $A(\mathbb{\mathbb{NP}})$; see
Subsection \ref{Subsection:osystems}.

\begin{theorem}
\label{Theorem:universal-ucp}There exists a unital completely positive
quotient mapping $\Omega _{A(\mathbb{\mathbb{NP}})}:A(\mathbb{\mathbb{NP}}%
)\rightarrow A(\mathbb{\mathbb{NP}})$ such that, if $T:X\rightarrow Y$ is a
unital completely positive linear map between separable exact operator
systems, then there exist unital completely isometric linear maps $\alpha
_{0}:X\rightarrow \mathbb{NG}$ and $\alpha _{1}:Y\rightarrow \mathbb{NG}$
such that $\alpha _{1}\circ T=\Omega _{\mathbb{NG}}\circ \alpha _{0}$.
Furthermore $\Omega _{A(\mathbb{\mathbb{NP}})}$ is generic in the sense that
the orbit $\left\{ \Omega _{A(\mathbb{\mathbb{NP}})}\circ \beta :\beta \in 
\mathrm{Aut}(A(\mathbb{NP}))\right\} $ with respect to the continuous action 
$\mathrm{Aut}(A(\mathbb{NP}))\curvearrowright \mathrm{UCP}(A(\mathbb{NP}))$
is a dense $G_{\delta }$ subspace of the space $\mathrm{UCP}(A(\mathbb{NP}))$
of unital completely positive maps from $A(\mathbb{\mathbb{NP}})$ to itself
endowed with the topology of pointwise convergence. The set 
\begin{equation*}
\left\{ x\in A(\mathbb{\mathbb{NP}}):\Omega _{A(\mathbb{\mathbb{NP}})}\text{
is a scalar multiple of the identity}\right\}
\end{equation*}%
is unitally completely isometrically isomorphic to $A(\mathbb{\mathbb{NP}})$%
. A unital completely positive map $T:A(\mathbb{\mathbb{NP}})\rightarrow A(%
\mathbb{\mathbb{NP}})$ belongs to the $\mathrm{Aut}(A(\mathbb{NP}))$-orbit
of $\Omega _{\mathbb{NG}}$ if and only if it satisfies the following
property: whenever $E_{0}\subset F_{0}$ and $E_{1}\subset F_{1}$ are
finite-dimensional exact operator spaces , $f_{0}:E_{0}\rightarrow A(\mathbb{%
\mathbb{NP}})$ and $f_{1}:E_{1}\rightarrow A(\mathbb{\mathbb{NP}})$ are
unital linear complete isometries, $L:F_{0}\rightarrow F_{1}$ is a unital
completely positive linear function mapping $E_{0}$ to $E_{1}$ such that $%
T\circ f_{0}=f_{1}\circ L$, and $\varepsilon >0$, then there exist unital
linear complete isometries $\widehat{f}_{0}:F_{0}\rightarrow A(\mathbb{%
\mathbb{NP}})$ and $\widehat{f}_{1}:F_{1}\rightarrow A(\mathbb{\mathbb{NP}})$
such that $\left\Vert T\circ \widehat{f}_{0}-\widehat{f}_{1}\circ
L\right\Vert _{cb}<\varepsilon $.
\end{theorem}

Results analogous to Theorem \ref{Theorem:operator-G}, Theorem \ref%
{Theorem:operator-A(P)}, Theorem \ref{Theorem:universal-cc}, and Theorem \ref%
{Theorem:universal-ucp} also hold for $M_{q}$-spaces, $M_{q}$-systems,
operator sequence spaces, and $p$-multinormed spaces; see Subsections \ref%
{Subsection:Mq}, \ref{Subsection:Mqu}, \ref{Subsection:sequence}, and \ref%
{Subsection:p-multinormed}.

\subsection{Morphisms between morphisms\label{Subsection:2morphism}}

We can regard morphisms between structures in $\mathcal{A}$ as objects of a
category $\mathcal{A}^{\rightarrow }$. Suppose that $T:X\rightarrow Y$ is a
morphism between structures in $\mathcal{A}$. We use the notation $D_{0}(T)$
and $D_{1}(T)$ to denote the domain and the codomain of $T$, respectively. A
morphism in $\mathcal{A}^{\rightarrow }$ from the morphism $%
T:D_{0}(T)\rightarrow D_{1}(T)$ to the morphism $S:D_{0}\left( S\right)
\rightarrow D_{1}\left( S\right) $ is given by a pair $\alpha =\left( \alpha
_{0},\alpha _{1}\right) $, where $\alpha _{0}:D_{0}(T)\rightarrow
D_{0}\left( S\right) $ and $\alpha _{1}:D_{1}(T)\rightarrow D_{1}\left(
S\right) $ are morphisms in $\mathcal{A}$. We do not require that $\alpha
_{1}\circ T=S\circ \alpha _{0}$. If $\alpha $ is a morphism from $T$ to $S$
as above, then we set $I\left( \alpha \right) $ to be maximum of $I\left(
\alpha _{0}\right) ,I\left( \alpha _{1}\right) $, and%
\begin{equation*}
\sup_{x}d\left( \left( \alpha _{1}\circ T\right) (x),\left( S\circ \alpha
_{0}\right) (x)\right)
\end{equation*}%
where $x$ ranges in $D_{0}(X)$, and $I\left( \alpha _{0}\right) ,I\left(
\alpha _{1}\right) $ are defined as in Subsection \ref{Subsection:morphism}.
Observe that $I\left( \alpha \right) $ measures how close $\alpha $ is to be
a pair of embeddings that commute with $T$ and $S$. If $\alpha ,\beta $ are
morphisms from $T$ to $S$ then we set $d\left( \alpha ,\beta \right) $ to be
the maximum of $d\left( \alpha _{0},\beta _{0}\right) $ and $d\left( \alpha
_{1},\beta _{1}\right) $. An \emph{embedding }from $T$ to $S$ is a morphism $%
\alpha $ as above such that moreover $\alpha _{0},\alpha _{1}$ are
isometries and $\alpha _{1}\circ T=S\circ \alpha _{0}$. An \emph{automorphism%
} of $T$ is an embedding $\left( \alpha _{0},\alpha _{1}\right) $ from $T$
to $T$ such that $\alpha _{0}$ and $\alpha _{1}$ are surjective.

Observe that the objects of $\mathcal{A}^{\rightarrow }$ can naturally be
regarded as structures in a language $\mathcal{L}^{\rightarrow }$.\ Here $%
\mathcal{L}^{\rightarrow }$ is the two-sorted language in sorts $D_{0}$ and $%
D_{1}$ that has

\begin{itemize}
\item an $n$-ary function symbols $f_{i}:D_{i}^{n}\rightarrow D_{i}$ for
every $i\in \left\{ 0,1\right\} $ and every $n$-ary function symbol $f$ in $%
\mathcal{L}$,

\item an $n$-relation symbol $R_{i}:D_{i}^{n}\rightarrow \left[ 0,1\right] $
for every $i\in \left\{ 0,1\right\} $ and every $n$-ary relation symbol $R$
in $\mathcal{L}$,

\item a unary function symbol $D_{0}\rightarrow D_{1}$.
\end{itemize}

Clearly a structure $T$ in $\mathcal{A}^{\rightarrow }$ is finitely
generated as $\mathcal{L}^{\rightarrow }$-structure if and only if both $%
D_{0}(T)$ and $D_{1}(T)$ are finitely generated as $\mathcal{L}$-structures.

\subsection{The generic morphism\label{Subsection:universal}}

Let $\mathcal{C}^{\rightarrow }\subset \mathcal{A}^{\rightarrow }$ be the
class of morphisms between finitely generated structures in $\mathcal{A}$.
We aim at showing that $\mathcal{C}^{\rightarrow }$ is a (complete) Fra\"{\i}%
ss\'{e} class in the sense of \cite[Definition 3.15]{ben_yaacov_fraisse_2015}%
. The fact that the class $\mathcal{C}_{n}^{\rightarrow }$ of $n$-marked
structures in $\mathcal{A}^{\rightarrow }$ is complete and separable can be
proved as in Subsection \ref{Subsection:class}. The same holds for the
hereditary property and the joint embedding property.\ It remains to prove
the near amalgamation property.

\begin{lemma}
\label{Lemma:pushout-morphism}Suppose that $T,\widehat{T},S$ are structures
in $\mathcal{A}^{\rightarrow }$, $\phi :T\rightarrow \widehat{T}$ and $%
f:T\rightarrow S$ are morphisms such that $I(\phi )\leq \delta $. Then there
exist a structure $\widehat{S}$ in $\mathcal{A}^{\rightarrow }$, a morphism $%
\widehat{f}:\widehat{T}\rightarrow \widehat{S}$, and an embedding $%
j:S\rightarrow \widehat{S}$ such that $\widehat{S}\circ \widehat{f}_{0}=%
\widehat{f}_{1}\circ \widehat{T}$ and $d(\widehat{f}\circ \phi ,j\circ
f)\leq \varpi (\delta )+2\delta $. If moreover $I(f)\leq \delta $ then $%
\widehat{f}$ is an embedding. If $T,\widehat{T},S$ are finitely generated,
then $\widehat{S}$ is finitely generated.
\end{lemma}

\begin{proof}
Let $D_{1}(\widehat{S})$ be the approximate pushout of $f_{1}$ and $\phi
_{1} $ with tolerance $\varpi (\delta )$ defined as in Lemma \ref%
{Lemma:pushout}. Consider also the canonical embedding $j_{1}:D_{1}(S)%
\rightarrow D_{1}(\widehat{S})$ and the canonical morphism $\widehat{f}%
_{1}:D_{1}(\widehat{T})\rightarrow D_{1}(\widehat{S})$. Define $D_{0}(%
\widehat{S})$ to be the approximate pushout of $f_{0}$ and $\phi _{0}$ with
tolerance $\varpi (\delta )+2\delta $. Again we have a canonical embedding $%
j_{0}:D_{0}(S)\rightarrow D_{0}(\widehat{S})$ and a canonical morphism $%
\widehat{f}_{0}:D_{0}(\widehat{T})\rightarrow D_{0}(\widehat{S})$. Observe
now that $j_{1}\circ S:D_{0}(S)\rightarrow D_{1}(\widehat{S})$ and $\widehat{%
f}_{1}\circ \widehat{T}:D_{0}(\widehat{T})\rightarrow D_{1}(\widehat{S})$
are morphisms such that 
\begin{eqnarray*}
d(j_{1}\circ S\circ f_{0},\widehat{f}_{1}\circ \widehat{T}\circ \phi _{0})
&\leq &d\left( S\circ f_{0},f_{1}\circ T\right) +d(\widehat{T}\circ \phi
_{0},\phi _{1}\circ T)+d(j_{1}\circ f_{1},\widehat{f}_{1}\circ \phi _{1}) \\
&\leq &\varpi (\delta )+2\delta \text{.}
\end{eqnarray*}%
Therefore by the universal property of the approximate pushout there exists
a unique morphism $\widehat{S}:D_{0}(\widehat{S})\rightarrow D_{1}(\widehat{S%
})$ such that $\widehat{S}\circ \widehat{f}_{0}=\widehat{f}_{1}\circ 
\widehat{T}$ and $\widehat{S}\circ j_{0}=j_{1}\circ S$. The same
construction also works to prove the other assertions.
\end{proof}

One can also consider in this context an analog of Lemma \ref{Lemma:pushout2}
involving approximate pushouts over a tuple. It is immediate to observe that
Lemma \ref{Lemma:pushout-morphism} shows that $\mathcal{C}^{\rightarrow }$
has the near amalgamation property. We can therefore conclude that $\mathcal{%
C}^{\rightarrow }$ is a Fra\"{\i}ss\'{e} class.

\begin{proposition}
\label{Proposition:existence-limit-arrow}The class $\mathcal{C}^{\rightarrow
}$ of finitely generated $\mathcal{L}^{\rightarrow }$-structures in $%
\mathcal{A}^{\rightarrow }$ is a Fra\"{\i}ss\'{e} class. The corresponding
Fra\"{\i}ss\'{e} limit is a morphism $\Omega _{M}:M\rightarrow M$, where $M$
is the Fra\"{\i}ss\'{e} limit of the class $\mathcal{C}$ of finitely
generated $\mathcal{L}$-structures in $\mathcal{A}$.
\end{proposition}

\begin{proof}
We have shown above that the collection $\mathcal{C}^{\rightarrow }$ of
finitely generated structures in $\mathcal{A}^{\rightarrow }$ is a Fra\"{\i}%
ss\'{e} class. The corresponding limit is a morphism $\Omega
_{M}:D_{0}(\Omega _{M})\rightarrow D_{1}(\Omega _{M})$. Using the
characterization of the Fra\"{\i}ss\'{e} limit from Proposition \ref%
{Proposition:characterize-limit}---see also \cite[Corollary 3.20]%
{ben_yaacov_fraisse_2015}---one can conclude that $D_{0}(\Omega _{M})$ and $%
D_{1}(\Omega _{M})$ satisfy the characterizing property of the Fra\"{\i}ss%
\'{e} limit $M$ of the class $\mathcal{C}$ of finitely generated structures
in $\mathcal{A}$. Therefore $D_{0}(\Omega _{M})$ and $D_{1}(\Omega _{M})$
are both isomorphic to $M$.
\end{proof}

It follows from universality of the Fra\"{\i}ss\'{e} limit that $\Omega _{M}$
is a \emph{universal morphism }between separable structures in $\mathcal{A}$%
. This means that if $T:D_{0}(T)\rightarrow D_{1}(T)$ is a morphism between
separable structures in $\mathcal{A}$, then there exist embeddings $\phi
_{0}:D_{0}(T)\rightarrow M$ and $\phi _{1}:D_{1}(T)\rightarrow M$ such that $%
\Omega _{M}\circ \phi _{0}=\phi _{1}\circ T$.

One can prove a characterization of $\Omega _{M}$ similar to the
characterization of $M$ given by Proposition \ref%
{Proposition:characterize-limit}. In particular if $S:M\rightarrow M$ is a
morphism, then the following statements are equivalent:

\begin{enumerate}
\item There exists a automorphisms $\alpha _{0},\alpha _{1}$ of $M$ such
that $\alpha _{0}\circ \Omega _{M}\circ \alpha _{1}=S$;

\item For every morphisms $T,\widehat{T}$ between finitely generated
structures in $\mathcal{A}$, $\delta >0$, morphisms $f:T\rightarrow S$ and $%
\phi :T\rightarrow \widehat{T}$ such that $I(f)<\delta $ and $I(\phi
)<\delta $, there exists an embedding $g:\widehat{T}\rightarrow S$ such that 
$d\left( g\circ \phi ,f\right) <\varpi (\delta )+2\delta $;

\item For every morphisms $T,\widehat{T}$ between finitely generated
structures in $\mathcal{A}$, embeddings $f:T\rightarrow S$ and $\phi
:T\rightarrow \widehat{T}$, and $\varepsilon >0$, there exists an embedding $%
g:\widehat{T}\rightarrow S$ such that $d\left( g\circ \phi ,f\right)
<\varepsilon $,

\item Whenever $T$ is a morphism between finitely-generated structures in $%
\mathcal{A}$, $f:T\rightarrow S$ and $\phi :T\rightarrow S$ are morphisms
such that $I\left( f\right) <\delta $ and $I\left( \phi \right) <\delta $,
there exists an automorphisms $\beta $ of $M$ such that $\beta \circ
S=S\circ \beta $ and $d\left( \beta \circ \phi ,f\right) <\varpi (\delta
)+2\delta $;

\item For any finite tuple $\bar{a}$ in $M$, morphisms $f:T|_{\left\langle 
\bar{a}\right\rangle }\rightarrow S$ and $\phi :T|_{\left\langle \bar{a}%
\right\rangle }\rightarrow S$ such that $I\left( f\right) <\delta $, $%
I\left( \phi \right) <\delta $, there exists an automorphism $\beta $ of $M$
such that $S\circ \beta =\beta \circ S$ and $d\left( \beta \circ \phi
,f\right) <\varpi (\delta )+2\delta $;

\item The same as (5) where the tuple $\bar{a}$ belongs to some fixed
countable fundamental subset of $M$, and $\phi (\bar{a}),f(\bar{a})$ belong
to some fixed countable fundamental subset of $M$.
\end{enumerate}

One can deduce from such a characterization that the orbit $\left\{ \alpha
\circ \Omega _{M}\circ \alpha _{1}:\alpha _{0},\alpha _{1}\in \mathrm{Aut}%
(M)\right\} $ of $\Omega _{M}$ is a dense $G_{\delta }$ subset of $\mathrm{%
End}(M)$. Here $\mathrm{End}(M)$ is the Polish space of morphisms $%
S:M\rightarrow M$ endowed with the topology of pointwise convergence, and $%
\mathrm{Aut}(M)\subset \mathrm{End}(M)$ is the $G_{\delta }$ subspace of
automorphisms of $M$.

Recall our assumption from Subsection \ref{Subsection:basic} that $\mathcal{A%
}$ has a universal initial object $A_{0}$ that is a finitely-generated
structure.

\begin{proposition}
\label{Proposition:kernel}Suppose that $A_{0}$ is also a universal initial
object in the category that has the same objects as $\mathcal{A}$ and
embeddings as morphisms. Identify canonically $A_{0}$ with a substructure of
any object of $\mathcal{A}$. Assume furthermore that for any structure $X$
in $\mathcal{A}$ there exists a morphism from $X$ to $A_{0}$. If $%
f:X\rightarrow Y$ is a morphism, we set $\mathrm{Ker}(f)=\left\{ x\in
X:f(x)\in A_{0}\right\} $. Then the morphism $\Omega _{M}$ is surjective and 
$\mathrm{\mathrm{Ker}}\left( \Omega _{M}\right) $ is isomorphic to $M$.
\end{proposition}

\begin{proof}
In order to prove that $\Omega _{M}$ is surjective, it is enough to show
that the range of $\Omega _{M}$ is dense. Fix $y\in M$ and $\varepsilon >0$.
Let $\left\langle y\right\rangle $ be the substructure of $M$ generated by $%
y $. Observe that $A_{0}\subset \left\langle y\right\rangle $. By the
characterization of $\Omega _{M}$, there exist embeddings $\psi _{0},\psi
_{1}:\left\langle y\right\rangle \rightarrow M$ such that $d\left( \psi
_{1}(y),y\right) <\varepsilon $ and $\Omega _{M}\circ \psi _{0}=\psi _{1}$.
Therefore $\psi _{1}(y)=\Omega _{M}(x)$ where $x=\psi _{0}(y)$ and $d\left(
y,\Omega _{M}(x)\right) <\varepsilon $. This concludes the proof that the
range of $\Omega _{M}$ is dense. We now show that $\mathrm{Ker}\left( \Omega
_{M}\right) $ is isomorphic to $M$. Suppose that $E$ is a finitely generated
structure in $\mathcal{A}$, $\bar{a}$ is a finite tuple in $E$, and $\phi
:\left\langle \bar{a}\right\rangle \rightarrow \mathrm{Ker}\left( \Omega
_{M}\right) $ is an embedding. Let $T:E\rightarrow A_{0}$ be a morphism. By
the properties of $\Omega _{M}$ there exists an embedding $\psi
:E\rightarrow M$ such that $d\left( \psi (\bar{a}),\phi (\bar{a})\right)
<\varepsilon $ and $\Omega _{M}\circ \psi =T$. This implies that the range
of $\psi $ is contained in $\mathrm{Ker}\left( \Omega _{M}\right) $. It
therefore follows from Proposition \ref{Proposition:characterize-limit} that 
$\mathrm{Ker}\left( \Omega _{M}\right) $ is isomorphic to $M$.
\end{proof}

\section{Universal states\label{Section:states}}

Throughout this section we still use the same notation and terminology as in
Section \ref{Section:injective}. Namely we assume that $\mathcal{A}$ and $%
\mathcal{I}$ are classes of $\mathcal{L}$-structures satisfying the
assumptions of Theorem \ref{Theorem:abstract-nonsense} such that $\mathcal{A}
$ has enough injectives from $\mathcal{I}$ with modulus $\varpi $.

\subsection{Kubi\'{s} universal projections\label%
{Subsection:universal-projection}}

In \cite[\S 4.1]{kubis_metric-enriched_2012} Kubi\'{s} constructs, for any
separable Lindenstrauss space $Y$, a projection $\Omega _{\mathbb{G}}^{Y}$
of norm $1$ on the Gurarij space $\mathbb{G}$ with the following properties:

\begin{itemize}
\item the range of $\Omega _{\mathbb{G}}^{Y}$ is isometrically isomorphic to 
$Y$, and the kernel of $\Omega _{\mathbb{G}}^{Y}$ is isometric to $\mathbb{G}
$;

\item for any separable Banach space $X$, and contractive linear mapping $%
\phi :X\rightarrow \mathbb{G}$ whose range is contained in the range of $%
\Omega _{\mathbb{G}}^{Y}$, there exists an embedding $\eta :X\rightarrow 
\mathbb{G}$ such that $\phi =\Omega _{\mathbb{G}}^{Y}\circ \eta $.
\end{itemize}

The existence of such a projection implies that $\mathbb{G}$ is
topologically isomorphic to $\mathbb{G}\oplus X$ for any separable
Lindenstrauss space $X$. (Recall that two Banach spaces $X,Y$ are
topologically isomorphic if there exists a bounded linear isomorphism from $%
X $ to $Y$.) It follows that if $Z$ is a separable Lindenstrauss space that
contains a complemented subspace isomorphic to $\mathbb{G}$, then $Z$ is
topologically isomorphic to $\mathbb{G}$. Hence $\mathbb{G}$ is also
topologically isomorphic to $\mathbb{G}\otimes X$ for any separable
Lindenstrauss space $X$. A similar result is obtained in \cite[Section 6]%
{cabello_sanchez_quasi-banach_2014} for $p\in (0,1]$ for the $p$-Gurarij
space $\mathbb{G}_{p}$, which is the Fra\"{\i}ss\'{e} limit of the class of
finite-dimensional $p$-Banach spaces.

In this section we will prove general results that imply the following
characterization of Kubi\'{s}' universal projection universal projection $%
\Omega _{\mathbb{G}}^{Y}$; see Subsection \ref{Subsection:Banach}.

\begin{theorem}
\label{Theorem:universal-projection-G}Fix a separable Lindenstraus space.
Suppose that $T:\mathbb{G}\rightarrow Y$ is a linear map of norm at most $1$%
, $N$ is the kernel of $T$, and $H=N^{\bot }\cap \mathbb{L}$. The following
assertions are equivalent:

\begin{enumerate}
\item $T$ is a quotient mapping and $N$ is a nonzero $M$-ideal of $\mathbb{G}
$;

\item $T$ is a quotient mapping and $H$ is a closed proper biface of $%
\mathbb{L}$ symmetrically affinely homeomorphic to $\mathrm{Ball}\left(
Y^{\ast }\right) $;

\item whenever $E\subset F$ are finite-dimensional Banach spaces, $%
f:E\rightarrow Y$ is a linear isometry, $s:F\rightarrow Y$ is a linear map
of norm at most $1$ such that $T\circ f=s$, and $\varepsilon >0$, there
exists a linear isometry $\widehat{f}:F\rightarrow Y$ such that $\left\Vert
T\circ \widehat{f}-s\right\Vert <\varepsilon $.
\end{enumerate}

The set of operators satisfying the equivalent conditions above is a dense $%
G_{\delta }$ subset of the space $\mathrm{\mathrm{\mathrm{Ball}}}\left( B(%
\mathbb{G})\right) $ of linear maps from $\mathbb{G}$ to $Y$ of norm at most 
$1$, and forms a single orbit under the action $\mathrm{Aut}(\mathbb{G}%
)\curvearrowright \mathrm{\mathrm{\mathrm{Ball}}}\left( B(\mathbb{G})\right) 
$, $\left( \alpha ,S\right) \mapsto S\circ \alpha ^{-1}$. If $\Omega _{%
\mathbb{G}}^{Y}:\mathbb{G}\rightarrow Y$ is such an operator, then the
kernel of $\Omega _{\mathbb{G}}^{Y}$ is isometrically isomorphic to $\mathbb{%
G}$. In other words the sequence 
\begin{equation*}
0\longrightarrow \mathbb{G\longrightarrow G}\overset{\Omega _{\mathbb{G}}^{Y}%
}{\longrightarrow }Y\longrightarrow 0
\end{equation*}%
where the first arrow is a linear isometry, is exact. Furthermore $\Omega _{%
\mathbb{G}}^{Y}$ is a universal liner map of norm at most $1$ from a
separable Banach space to $Y$, in the sense that any if $E$ is a separable
Banach space, and $L:E\rightarrow Y$ is a linear map of norm at most $1$,
then there exists a linear isometry $\eta :E\rightarrow \mathbb{G}$ such
that $\Omega _{\mathbb{G}}^{Y}\circ \eta =L$. In particular $\Omega _{%
\mathbb{G}}^{Y}$ can be regarded as a projection of norm $1$ onto an
isometric copy of $Y$ inside $\mathbb{G}$.
\end{theorem}

The analog of Theorem \ref{Theorem:universal-projection-G} in the case of
the Poulsen system holds as well.

\begin{theorem}
\label{Theorem:universal-projection-A(P)}Fix $K$ is a metrizable Choquet
simplex. Suppose that $T:A(\mathbb{P})\rightarrow A(K)$ is a unital positive
linear map, $N$ is the kernel of $T$, and $H=N^{\bot }\cap \mathbb{P}$. The
following assertions are equivalent:

\begin{enumerate}
\item $T$ is a quotient mapping and $N$ is a nonzero $M$-ideal of $A(\mathbb{%
P})$;

\item $T$ is a quotient mapping and $H$ is a closed proper face of $\mathbb{P%
}$ affinely homeomorphic to $K$;

\item whenever $E\subset F$ are finite-dimensional function systems, $%
f:E\rightarrow A(K)$ is a linear isometry, $s:F\rightarrow A(K)$ is a unital
linear function such that $T\circ f=s$, and $\varepsilon >0$, there exists a
unital linear isometry $\widehat{f}:F\rightarrow A(K)$ such that $\left\Vert
T\circ \widehat{f}-s\right\Vert <\varepsilon $.
\end{enumerate}

The set of operators satisfying the equivalent conditions above is a dense $%
G_{\delta }$ subset of the space $\mathrm{UP}\left( A(\mathbb{P}%
),A(K)\right) $ of unital positive linear maps from $A(\mathbb{P})$ to $A(K)$%
, and forms a single orbit under the action $\mathrm{Aut}\left( A(\mathbb{P}%
)\right) \curvearrowright \mathrm{UP}\left( A(\mathbb{P}),A(K)\right) $, $%
\left( \alpha ,S\right) \mapsto S\circ \alpha ^{-1}$. If $\Omega _{A(\mathbb{%
P})}^{A(K)}:A(\mathbb{P})\rightarrow A(K)$ is such an operator, then 
\begin{equation*}
\left\{ x\in A(\mathbb{P}):\Omega _{A(\mathbb{P})}^{A(K)}(x)\text{ is a
scalar multiple of the identity}\right\}
\end{equation*}%
is unitally isometrically isomorphic to $A(\mathbb{P})$. Furthermore $\Omega
_{A(\mathbb{P)}}^{A(K)}$ is a universal unital positive liner map from a
separable function system to $A(K)$, in the sense that any if $A(T)$ is a
separable function system, and $L:A(T)\rightarrow A(K)$ is a unital positive
linear map, then there exists a unital linear isometry $\eta
:A(T)\rightarrow A(\mathbb{P})$ such that $\Omega _{A(\mathbb{P}%
)}^{A(K)}\circ \eta =L$. In particular $\Omega _{A(\mathbb{P})}^{A(K)}$ can
be regarded as a projection onto a unital isometric copy of $A(K)$ inside $A(%
\mathbb{P})$.
\end{theorem}

The noncommutative analogs of Theorem \ref{Theorem:universal-projection-G}
and Theorem \ref{Theorem:universal-projection-A(P)} hold as well; see
Subsection \ref{Subsection:ospaces} and Subsection \ref{Subsection:osystems}.

\begin{theorem}
\label{Theorem:universal-projection-NG}Fix a separable nuclear operator
space $Y$ and let $\mathbb{NG}$ be the noncommutative Gurarij space. There
exists a linear complete contraction $\Omega _{\mathbb{\mathbb{NG}}}^{Y}:%
\mathbb{\mathbb{NG}}\rightarrow Y$ such that if $E$ is a separable Banach
space, and $L:E\rightarrow Y$ is a completely contractive linear map, then
there exists a linear complete isometry $\eta :E\rightarrow \mathbb{\mathbb{%
NG}}$ such that $\Omega _{\mathbb{\mathbb{NG}}}^{Y}\circ \eta =L$.
Furthermore $\Omega _{\mathbb{\mathbb{NG}}}^{Y}$ is generic, in the sense
that the orbit of $\Omega _{\mathbb{\mathbb{NG}}}^{Y}$ inside the space $%
\mathrm{\mathrm{\mathrm{Ball}}}\left( CB(\mathbb{NG})\right) $ of completely
contractive linear maps from $\mathbb{NG}$ to $Y$ under the action $\mathrm{%
Aut}(\mathbb{NG})\curvearrowright $ $\mathrm{\mathrm{\mathrm{Ball}}}\left(
CB(\mathbb{NG})\right) $, $\left( \alpha ,S\right) \mapsto S\circ \alpha
^{-1}$ is a dense $G_{\delta }$ set. The kernel of $\Omega _{\mathbb{\mathbb{%
NG}}}^{Y}$ is isometrically isomorphic to $\mathbb{\mathbb{NG}}$. In other
words the sequence 
\begin{equation*}
0\longrightarrow \mathbb{NG\longrightarrow \mathbb{NG}}\overset{\Omega _{%
\mathbb{\mathbb{NG}}}^{Y}}{\longrightarrow }Y\longrightarrow 0
\end{equation*}%
where the first arrow is a linear isometry, is exact. A completely
contractive linear map $T:\mathbb{NG}\rightarrow Y$ belongs to the $\mathrm{%
Aut}(\mathbb{NG})$-orbit of $\Omega _{\mathbb{NG}}^{Y}$ if and only if it
satisfies the following property: whenever $E\subset F$ are
finite-dimensional operator spaces, $f:E\rightarrow Y$ is a linear complete
isometry, $s:F\rightarrow Y$ is a completely contractive linear map such
that $T\circ f=s$, and $\varepsilon >0$, there exists a linear complete
isometry $\widehat{f}:F\rightarrow Y$ such that $\left\Vert T\circ \widehat{f%
}-s\right\Vert <\varepsilon $.
\end{theorem}

Two operator spaces $X,Y$ are completely isomorphic if there exists a
completely bounded linear isomorphism from $X$ to $Y$. A subspace of an
operator space is completely complemented if it is the range of a completely
bounded projection.

\begin{corollary}
\label{Corollary:isomorphism}The noncommutative Gurarij space $\mathbb{NG}$
is completely isomorphic to $\mathbb{NG}\oplus Y$ for any separable nuclear
operator space $Y$. If a nuclear operator space contains a completely
complemented subspace isomorphic to $\mathbb{NG}$, then it is isomorphic to $%
\mathbb{NG}$. In particular $\mathbb{NG}$ is isomorphic to $\mathbb{NG}%
\otimes Y$ for any separable nuclear operator space $X$.
\end{corollary}

\begin{proof}
By Theorem \ref{Theorem:universal-projection-NG} one has an exact sequence
of completely contractive maps%
\begin{equation*}
0\longrightarrow \mathbb{NG\longrightarrow NG}\longrightarrow
Y\longrightarrow 0
\end{equation*}%
where the second map is a complete isometry. It follows that $\mathbb{NG}$
is completely isomorphic to $\mathbb{NG}\oplus Y$. The other assertions
follow as in the proof of \cite[Corollary 6.6]%
{cabello_sanchez_quasi-banach_2014}.
\end{proof}

\begin{theorem}
\label{Theorem:universal-projection-NP}Fix a separable nuclear operator
system $Y$ and let $\mathbb{NP}$ be the noncommutative Poulsen simplex, with
associated operator system $A(\mathbb{\mathbb{NP}})$. There exists a unital
completely positive map $\Omega _{A(\mathbb{\mathbb{NP)}}}^{Y}:A(\mathbb{%
\mathbb{NP}})\rightarrow Y$ such that if $X$ is a separable operator system,
and $L:E\rightarrow Y$ is a unital completely positive linear map, then
there exists a unital linear complete isometry $\eta :X\rightarrow A(\mathbb{%
NP})$ such that $\Omega _{\mathbb{\mathbb{NP}}}^{Y}\circ \eta =L$.
Furthermore $\Omega _{A(\mathbb{\mathbb{NP}})}^{Y}$ is generic, in the sense
that the $\mathrm{Aut}(A(\mathbb{NP}))$-orbit of $\Omega _{\mathbb{\mathbb{NP%
}}}^{Y}$ inside the space $\mathrm{UCP}\left( A(\mathbb{\mathbb{NP}}%
),Y\right) $ of unital completely positive linear maps from $A(\mathbb{%
\mathbb{NP}})$ to $Y$ under the action $\mathrm{Aut}(A(\mathbb{NP}%
))\curvearrowright $ $\mathrm{UCP}\left( A(\mathbb{\mathbb{NP}}),Y\right) $, 
$\left( \alpha ,S\right) \mapsto S\circ \alpha ^{-1}$ is a dense $G_{\delta
} $ set. The set%
\begin{equation*}
\left\{ x\in A(\mathbb{\mathbb{NP}}):\Omega _{\mathbb{NP}}^{Y}(x)\text{ is a
scalar multiple of the identity}\right\}
\end{equation*}%
is unitally completely isometric isomorphic to $A(\mathbb{\mathbb{NP}})$. A
unital completely positive map $T:A(\mathbb{\mathbb{NP}})\rightarrow Y$
belongs to the $\mathrm{Aut}(A(\mathbb{NP}))$-orbit of $\Omega _{A(\mathbb{%
\mathbb{NP}})}^{Y}$ if and only if it satisfies the following property:
whenever $E\subset F$ are finite-dimensional operator systems, $%
f:E\rightarrow Y$ is a unital linear complete isometry, $s:F\rightarrow Y$
is a unital completely positive such that $T\circ f=s$, and $\varepsilon >0$%
, there exists a unital linear complete isometry $\widehat{f}:F\rightarrow Y$
such that $\left\Vert T\circ \widehat{f}-s\right\Vert <\varepsilon $.
\end{theorem}

The universal operators $\Omega _{\mathbb{G}}$, $\Omega _{\mathbb{P}}$, $%
\Omega _{\mathbb{NG}}$, and $\Omega _{\mathbb{NP}}$ from Theorem \ref%
{Theorem:operator-G}, Theorem \ref{Theorem:operator-A(P)}, Theorem \ref%
{Theorem:universal-cc}, and Theorem \ref{Theorem:universal-ucp} can be
obtained from Theorem \ref{Theorem:universal-projection-G}, Theorem \ref%
{Theorem:universal-projection-A(P)}, Theorem \ref%
{Theorem:universal-projection-NG}, and Theorem \ref%
{Theorem:universal-projection-NP} in the particular case when $Y=\mathbb{G}$%
, $Y=A(\mathbb{P})$, $Y=\mathbb{\mathbb{NG}}$, and $Y=A(\mathbb{\mathbb{NP}}%
) $, respectively.

\subsection{States as structures}

Fix an approximately injective separable structure $R$ in $\mathcal{A}$; see
Definition \ref{Definition:approx-inj}. An $R$\emph{-state} is a morphism $%
s:X_{s}\rightarrow R$ from a structure $X_{s}$ in $\mathcal{A}$ to $R$. The
terminology comes from the case of function systems, for which a state is a
unital positive linear functional; see \S \ref{Subsection:function-system}.
We regard $R$-states as structures in a category $\mathcal{A}_{R}$. A
morphism from $s$ to $t$ is a morphism $f:X_{s}\rightarrow X_{t}$ in $%
\mathcal{A}$. We do not require that $t\circ f=s$. We consider $R$-states as
structures in a language $\mathcal{L}_{R}$ containing two sorts $D_{X}$ and $%
D_{R}$ and a function symbol $D_{X}\rightarrow D_{R}$. Furthermore for any $%
n $-ary function symbol $f$ in $\mathcal{L}$ one has an $n$-ary function
symbols $f_{X}:D_{X}^{n}\rightarrow D_{X}$ and an $n$-ary function symbol $%
f_{R}:D_{R}^{n}\rightarrow D_{R}$. Similarly for any $n$-ary relation symbol 
$B$ in $\mathcal{L}$ one has an $n$-ary relation symbol $B_{X}:D_{X}^{n}%
\rightarrow \mathbb{R}$ and an $n$-ary relation symbol $B_{R}:D_{R}^{n}%
\rightarrow \mathbb{R}$. If $X$ is a separable structure in $\mathcal{A}$,
then the space $S\left( X,R\right) $ of $R$-states on $X$ endowed with the
topology of pointwise convergence is a Polish space. The Polish group $%
\mathrm{Aut}(X)$ acts continuously on $S\left( X,R\right) $ by $\left(
\alpha ,s\right) \mapsto s\circ \alpha ^{-1}$.

\subsection{The generic state\label{Subsection:universal-state}}

Suppose that $X,\widehat{X},Y$ are structures in $\mathcal{A}$, $\widehat{s}$
and $t$ are $R$-states on $\widehat{X}$ and $Y$ respectively, and $%
f:X\rightarrow Y$ and $\phi :X\rightarrow \widehat{X}$ are morphisms such
that $I(\phi )<\delta $ and $d\left( \widehat{s}\circ \phi ,t\circ f\right)
\leq \varpi (\delta )$. Let $\widehat{Y}$ be the approximate pushout of $f$
and $\phi $ defined as in Lemma \ref{Lemma:pushout}, with canonical morphism 
$\widehat{f}:\widehat{X}\rightarrow \widehat{Y}$ and embedding $%
j:Y\rightarrow \widehat{Y}$. It follows from the universal property of the
approximate pushout that there exists a (unique) $R$-state $\widehat{t}$ on $%
\widehat{Y}$ such that $\widehat{t}\circ \widehat{f}=\widehat{s}$ and $%
\widehat{t}\circ j=t$. Again a similar argument applies the approximate
pushouts over a tuple as in Lemma \ref{Lemma:pushout2}.

Using this observation one can show that the states $s$ in $\mathcal{A}_{R}$
such that $X_{s}$ is a finitely generated structure form a Fra\"{\i}ss\'{e}
class. The corresponding limit $\Omega _{M}^{R}$ is an $R$-state on the Fra%
\"{\i}ss\'{e} limit $M$ of the class of finitely generated structures in $%
\mathcal{A}$, as it can be verified using uniqueness of the limit and
approximate injectivity of $R$. Furthermore if $s$ is an $R$-state on $M$,
then the following assertions are equivalent:

\begin{enumerate}
\item There exists an automorphism $\alpha $ of $M$ such that $s\circ \alpha
=\Omega _{M}^{R}$;

\item Whenever $\phi :E\rightarrow F$ is a morphism between finitely
generated structures in $\mathcal{A}$ such that $I(\phi )<\delta $, $t$ is
an $R$-state on $F$, and $f:E\rightarrow M$ is a morphism such that $d\left(
t\circ \phi ,s\circ f\right) <\varpi (\delta )$, there exists an embedding $%
g:F\rightarrow M$ such that $s\circ g=t$ and $d\left( g\circ \phi ,f\right)
<\varpi (\delta )$;

\item Whenever $E,F$ are finitely generated structures in $\mathcal{A}$ such
that $F\in \mathcal{I}$, $t$ is an $R$-state of $F$, $\phi :E\rightarrow F$
and $f:E\rightarrow M$ are embeddings such that $t\circ \phi =s\circ f$, and 
$\varepsilon >0$, there exists an embedding $g:F\rightarrow M$ such that $%
s\circ g=t$ and $d\left( g\circ \phi ,f\right) <\varepsilon $;

\item For any finitely generated structure $E$ in $\mathcal{A}$, $R$-state $%
t $ on $E$, and morphisms $f:E\rightarrow M$ and $\phi :E\rightarrow M$ such
that $I\left( \phi \right) <\delta $, $I\left( f\right) <\delta $, and $%
d\left( s\circ \phi ,s\circ f\right) <\varpi (\delta )$, there exists an
automorphism $\beta $ of $M$ such that $d\left( \beta \circ \phi ,f\right)
<\varpi (\delta )$ and $s\circ \beta =s$;

\item For any finite tuple $\bar{b}$ in $M$, morphisms $f:\left\langle \bar{b%
}\right\rangle \rightarrow M$ and $\phi :\left\langle \bar{b}\right\rangle
\rightarrow M$ such that $I\left( f\right) <\delta $, $I\left( \phi \right)
<\delta $, $s\circ \phi \thickapprox _{\bar{a},\delta }f$, there exists an
automorphism $\beta $ of $M$ such that $\beta \circ \phi \thickapprox _{\bar{%
a},\varpi (\delta )}f$ and $s\circ \beta =s$;

\item the same as (5) where moreover $\bar{b}\in M_{0}$ and $f(\bar{b}),\phi
(\bar{b})\in B_{0}$ for some fixed countable fundamental subsets $M_{0}$ of $%
M$ and $B_{0}$ of $B$ as in Definition \ref{Definition:fundamental}.
\end{enumerate}

Such a characterization in particular shows that the set $\left\{ \Omega
_{M}^{R}\circ \alpha :\alpha \in \mathrm{Aut}(M)\right\} $ is a dense $%
G_{\delta }$ subset of the space of $R$-states of $M$. It is not difficult
to verify using the universal property characterizing the universal state $%
\Omega _{M}^{M}$ and the universal operator $\Omega _{M}$ as in Subsection %
\ref{Subsection:universal} that $\Omega _{M}$ and $\Omega _{M}^{R}$ for $R=M$
have the same $\mathrm{Aut}(M)$-orbit. In the case of rigid $\mathcal{I}$%
-structures as in Definition \ref{Definition:I-structure}, $\Omega _{M}^{R}$
admits the following further characterization, which can be proved similarly
as Proposition \ref{Proposition:characterize-limit-I-structure}.

\begin{proposition}
\label{Proposition:characterize-universal-state}Let $X$ be a rigid $\mathcal{%
I}$-structure, and $s$ be an $R$-state on $X$. If for any $\varepsilon >0$,
structures $A,\widehat{A}\in \mathcal{I}$, $R$-state $t$ on $\widehat{A}$,
and embeddings $\phi :A\rightarrow \widehat{A}$ and $f:A\rightarrow X$ such
that $t\circ \phi =s\circ f$, there exists an embedding $\widehat{f}:%
\widehat{A}\rightarrow X$ such that $d(s\circ \widehat{f},t)<\varepsilon $
and $d(\widehat{f}\circ \phi ,f)<\varepsilon $, then there exists an
isomorphism $\alpha :X\rightarrow M$ such that $\Omega _{M}^{R}\circ \alpha
=s$.
\end{proposition}

One can prove similarly as in Proposition \ref{Proposition:kernel} that
under the same assumptions of Proposition \ref{Proposition:kernel} the
morphism $\Omega _{R}$ is surjective, and $\mathrm{Ker}\left( \Omega
_{R}\right) $ is isomorphic to $M$. Universality of the Fra\"{\i}ss\'{e}
limit implies that if $X$ is a separable structure in $\mathcal{A}$ and $s$
is an $R$-state on $X$, then there exists an embedding $\phi :X\rightarrow M$
such that $\Omega _{M}^{R}\circ \phi =s$. In particular letting $X=R$ and $s$
be the identity map of $R$ one can conclude that there exists an embedding $%
\eta _{R}:R\rightarrow X$ such that $\Omega _{M}^{R}\circ \eta _{R}$ is the
identity map of $R$. This also implies that $\Omega _{M}^{R}$ is surjective.
Defining $\rho _{R}$ to be $\eta _{R}\circ \Omega _{M}^{R}$ gives a
retraction of $M$ onto a substructure of $M$ isomorphic to $R$. This shows
that $R$ is isomorphic to a retract of $M$, which is the content of Theorem %
\ref{Theorem:retracts}. Furthermore $\rho _{R}$ is a \emph{universal
retraction }in the following sense. If $X$ is a separable structure in $%
\mathcal{A}$ and $s$ is a state on $X$ whose range is contained in the range
of $\rho _{R}$, then there exists an embedding $\psi :X\rightarrow M$ such
that $\rho _{R}\circ \psi =s$.

\begin{remark}
\label{Remark:lift}A further back-and-forth argument together with Condition
(4) in the characterization of the universal state $\Omega _{M}^{R}$ shows
that any automorphism of $R$ \textquotedblleft lifts\textquotedblright\ to
an automorphism of $M$. This means that if $\sigma $ is an automorphism of $%
R $, then there exists an automorphism $\widehat{\sigma }$ of $M$ such that $%
\sigma \circ \Omega _{M}^{R}=\Omega _{M}^{R}\circ \widehat{\sigma }$.
\end{remark}

A similar construction to the one above is performed in \cite[\S 4.1]%
{kubis_metric-enriched_2012} in the case of Banach spaces and, more
generally, in \cite[Section 6]{cabello_sanchez_quasi-banach_2014} in the
case of $p$-Banach spaces for every $p\in (0,1]$. The case of Banach spaces
is subsumed by the above general results; see \S \ref{Subsection:Banach}.
The case of $p$-Banach spaces for $p\in \left( 0,1\right) $ does not fit in
the framework of this paper, since no nontrivial $p$-Banach space for $p\in
\left( 0,1\right) $ is injective \cite[Proposition 5.2]%
{cabello_sanchez_quasi-banach_2014}. However, one can consider a
generalization of the assumptions considered in this paper, where the
structures in the class $\mathcal{I}$ are not assumed to be injective, but
only approximately injective as in Definition \ref{Definition:approx-inj}.
In this more general framework one can recover the main results of \cite%
{cabello_sanchez_quasi-banach_2014} concerning $p$-Banach spaces for
arbitrary $p\in (0,1]$.

\subsection{The $\mathrm{Aut}(M)$-space $S(M,R)$\label%
{Subsection:state-space}}

The automorphism group $\mathrm{Aut}(M)$ of $M$ is a Polish group when
endowed with the topology of pointwise convergence. Also $S(M,R)$ is a
Polish space endowed with the topology of pointwise convergence.

We regard $S(M,R)$ as a uniform $\mathrm{Aut}(M)$-space endowed with the
uniformity generated by the sets of the form%
\begin{equation*}
\left\{ \left( s_{0},s_{1}\right) \in S(M,R)\times S(M,R):d\left(
s_{0}(x),s_{1}(x)\right) <\varepsilon \right\}
\end{equation*}%
for $x\in M$ and $\varepsilon >0$. The action of $\mathrm{Aut}(M)$ on $%
S(M,R) $ is defined by $\left( \alpha ,s\right) \mapsto s\circ \alpha ^{-1}$%
. By completeness of $R$, the uniform space $S(M,R)$ is complete as well.
Furthermore $S(M,R)$ it is compact whenever $R$ is compact. Let $\mathrm{Aut}%
\left( M,\Omega _{M}^{R}\right) $ be the \emph{stabilizer }of $\alpha $,
i.e.\ the group of automorphisms $\alpha $ of $M$ such that $\Omega
_{M}^{R}\circ \alpha =\Omega _{M}^{R}$. We also regard $\mathrm{Aut}(M)/%
\mathrm{Aut}\left( M,\Omega _{M}^{R}\right) $ as a uniform $\mathrm{\mathrm{%
Aut}}(M)$-space endowed with the quotient of the right uniformity on $%
\mathrm{Aut}(M)$ and the canonical action by translation. The sets of the
form%
\begin{equation*}
\left\{ \left( \alpha _{0},\alpha _{1}\right) \in \mathrm{Aut}(M)\times 
\mathrm{Aut}(M):d\left( \alpha _{0}^{-1}(x),\alpha _{1}^{-1}(x)\right)
<\varepsilon \right\}
\end{equation*}%
form a basis of entourages for the right uniformity on $\mathrm{Aut}(M)$.

One can define the map $\pi :\mathrm{Aut}(M)/\mathrm{Aut}\left( M,\Omega
_{M}^{R}\right) \rightarrow S(M,R)$ mapping the left coset of $\mathrm{Aut}%
\left( M,\Omega _{M}^{R}\right) $ with respect to $\alpha $ to $s\circ
\alpha $. Clearly $\pi $ is an injective $\mathrm{Aut}(M)$-equivariant
uniformly continuous map. Furthermore by genericity of $\Omega _{M}^{R}$, $%
\pi $ has dense image. We claim that $\pi ^{-1}$ is uniformly continuous as
well. Indeed suppose that $\bar{a}$ is a finite tuple in $M$ and $%
\varepsilon >0$. If $\alpha ,\beta $ are automorphisms such that $\Omega
_{M}^{R}\circ \alpha ^{-1}\thickapprox _{\bar{a},\varepsilon }\Omega
_{M}^{R}\circ \beta ^{-1}$, then by Condition (5) in the characterization of 
$\Omega _{M}^{R}$ there exists $\gamma \in \mathrm{Aut}\left( M,\Omega
_{M}^{R}\right) $ such that $\left( \alpha \circ \gamma \right)
^{-1}\thickapprox _{\bar{a},\varpi (\delta )}\beta ^{-1}$. This concludes
the proof that $\pi ^{-1}$ is uniformly continuous. In particular this shows
that the completion of $\mathrm{Aut}(M)/\mathrm{Aut}\left( M,\Omega
_{M}^{R}\right) $ is $\mathrm{Aut}(M)$-equivariantly uniformly isomorphic to 
$S(M,R)$.

Recall that, if $G\ $is a topological group, then a uniform $G$-space is 
\emph{minimal }if every orbit of $G$ is dense. We can provide a
reformulation of the assertion that $S(M,R)$ is a minimal $\mathrm{Aut}(M)$%
-space in terms of the Fra\"{\i}ss\'{e} class $\mathcal{C}$.

\begin{proposition}
\label{Proposition:minimal}Consider the following assertions:

\begin{enumerate}
\item For every tuple $\bar{a}$ in $M$, $s\in S(M,R)$, and $\varepsilon >0$,
there exists $B\in \mathcal{I}$ such that for any $t\in S\left( B,R\right) $
there exists a morphism $\phi :\left\langle \bar{a}\right\rangle \rightarrow
B$ such that $I\left( \phi \right) <\varepsilon $ and $d\left( \left( t\circ
\phi \right) (\bar{a}),s(\bar{a})\right) <\varepsilon $;

\item for every tuple $\bar{a}$ in $M$ such that $\left\langle \bar{a}%
\right\rangle \in \mathcal{I}$, $s\in S(M,R)$, and $\varepsilon >0$, there
exists a finitely generated substructure $B$ of $M$ such that for any $t\in
S(M,R)$ there exists a morphism $\phi :\left\langle \bar{a}\right\rangle
\rightarrow B$ such that $I\left( \phi \right) <\varepsilon $ and $d\left(
\left( t\circ \phi \right) (\bar{a}),s(\bar{a})\right) <\varepsilon $;

\item the action $\mathrm{Aut}(M)\curvearrowright S(M,R)$ is minimal.
\end{enumerate}

Then (1)$\Rightarrow $(2)$\Rightarrow $(3). If furthermore $R$ is compact,
then (3)$\Rightarrow $(1).
\end{proposition}

\begin{proof}
The implication (1)$\Rightarrow $(2) is obvious.

For (2)$\Rightarrow $(3), suppose that $s,t\in S(M,R)$, $\bar{a}$ is a tuple
in $M$, and $\varepsilon >0$. We want to find $\alpha \in \mathrm{Aut}(M)$
such that $d\left( \left( s\circ \alpha \right) (\bar{a}),t(\bar{a})\right)
<\varepsilon $. Without loss of generality we can assume that $\left\langle 
\bar{a}\right\rangle \in \mathcal{I}$. The automorphism $\alpha $ can then
be obtained from the hypothesis using the stable homogeneity property of $M$.

We now assume that $R$ is compact, and prove (3)$\Rightarrow $(1). Suppose
that $\mathrm{Aut}(M)\curvearrowright S(M,R)$ is minimal, but (1) does not
hold. Thus for some tuple $\bar{a}$ in $M$, $\varepsilon _{0}>0$, and $%
s_{0}\in S(M,R)$, for every $B\in \mathcal{I}$ there exists $t_{B}\in S(M,R)$
such that for every morphism $\phi :\left\langle \bar{a}\right\rangle
\rightarrow B$ such that $I\left( \phi \right) <\varepsilon _{0}$ one has
that $d\left( \left( t_{B}\circ \phi \right) (\bar{a}),s(\bar{a})\right)
\geq \varepsilon _{0}$. Without loss of generality we can assume that $\bar{a%
}$ is a basic tuple. Let $\mathcal{B}$ be the set of pairs $\left( B,\delta
\right) $ such that $B\in \mathcal{I}$ and $\delta >0$. For every $\eta >0$
and tuple $\bar{b}$ in $M$, let $\mathcal{B}_{\bar{b},\eta }$ be the set of $%
\left( B,\delta \right) \in \mathcal{B}$ such that $\bar{b}\subset _{\eta }B$
and $\delta <\eta $. Observe that the collection of subsets $\mathcal{B}_{%
\bar{b},\eta }$ of $\mathcal{B}$ where $\bar{b}$ varies among the finite
tuples in $M$ and $\eta >0$ has the finite intersection property.\ Therefore
there exists an ultrafilter $\mathcal{U}$ on $\mathcal{B}$ that contains the
set $\mathcal{B}_{\bar{b},\eta }$ for every tuple $\bar{b}$ in $M$ and $\eta
>0$. Fix $x\in M$ and $\left( B,\eta \right) \in \mathcal{B}$. Define $%
t_{B,\eta }(x)=t_{B}(x)$ for any $M\supset B\in \mathcal{I}$ such that $x\in
_{\eta }B$. Finally let $t(x)$ be the limit according to $\mathcal{U}$ of
the function $\left( B,\eta \right) \mapsto t_{B,\eta }(x)$. This defines an
element $t$ of $S(M,R)$. By minimality of the action $\mathrm{Aut}%
(M)\curvearrowright S(M,R)$, for every $\delta >0$ there exists $\alpha \in 
\mathrm{Aut}(M)$ such that $d\left( \left( t\circ \alpha \right) (\bar{a}),%
\bar{a}\right) <\delta $. Using the hypotheses on basic sequences from
Subsection \ref{Subsection:basic}, this easily leads to a contradiction with
our assumption.
\end{proof}

\begin{corollary}
\label{Corollary:minimal}If for every $A\in \mathcal{I}$, $s\in S\left(
A,R\right) $, and $\varepsilon >0$, there exists $B\in \mathcal{I}$ such
that for any $t\in S\left( B,R\right) $ there exists a morphism $\phi
:A\rightarrow B$ such that $I(\phi )<\varepsilon $ and $d\left( t\circ \phi
,s\right) <\varepsilon $, then $\mathrm{Aut}(M)\curvearrowright S(M,R)$ is
minimal.
\end{corollary}

Suppose that $G$ is a topological group, and $X$ is a compact space. A
continuous action of $G\curvearrowright X$ is called \emph{proximal }if for
every entourage $U$ of the unique compatible uniformity of $X$ and $x,y\in X$
there exists $g\in G$ such that $\left( gx,gy\right) \in U$ \cite[\S I.1]%
{glasner_proximal_1976}. More generally we call a uniform $G$-space $X$
proximal if it satisfies the same property where $U$ is an entourage of the
given uniformity of $X$. The following characterization of classes for which
the action $\mathrm{Aut}(M)\curvearrowright S(M,R)$ is proximal is an
immediate consequence of stable homogeneity of the limit $M$ and our
assumptions on basic sequences.

\begin{proposition}
The following assertions are equivalent:

\begin{enumerate}
\item For every tuple $\bar{a}$ in $M$, $s,t\in S(M,R)$, and $\varepsilon >0$%
, there exists $B\in \mathcal{I}$ and a morphism $\phi :\left\langle \bar{a}%
\right\rangle \rightarrow B$ such that $I\left( \phi \right) <\varepsilon $
and $d\left( \left( t\circ \phi \right) (\bar{a}),\left( s\circ \phi \right)
(\bar{a})\right) <\varepsilon $;

\item for every tuple $\bar{a}$ in $M$ such that $\left\langle \bar{a}%
\right\rangle \in \mathcal{I}$, $s,t\in S(M,R)$, and $\varepsilon >0$, there
exists a finitely generated structure $B$ in $\mathcal{A}$ and a morphism $%
\phi :\left\langle \bar{a}\right\rangle \rightarrow B$ such that $I\left(
\phi \right) <\varepsilon $ and $d\left( \left( t\circ \phi \right) (\bar{a}%
),s(\bar{a})\right) <\varepsilon $;

\item the action $\mathrm{Aut}(M)\curvearrowright S(M,R)$ is proximal.
\end{enumerate}
\end{proposition}

\section{Examples\label{Section:examples}}

In this section we explain how many classes of structures fit into the
framework of Sections \ref{Section:injective}, \ref{Section:retracts}, \ref%
{Section:operators}, \ref{Section:states}.

\subsection{Real Banach spaces\label{Subsection:Banach}}

In this subsection we assume all the Banach spaces to be over the real
numbers. Suppose that $\mathcal{L}$ is the language containing binary
function symbols $f_{\lambda ,\mu }$ for $\lambda ,\mu \in \mathbb{Q}$ such
that $\left\vert \lambda \right\vert +\left\vert \mu \right\vert \leq 1$. We
can identify a Banach space $X$ with its unit ball $\mathrm{Ball}(X)$, which
is naturally an $\mathcal{L}$-structure where the interpretation of $%
f_{\lambda ,\mu }$ is the function $\left( x,y\right) \mapsto \lambda x+\mu
y $. Under this identification, the morphisms according to Definition \ref%
{Definition:morphism} are precisely the restriction to the unit ball of
bounded linear maps of norm at most $1$. Indeed suppose that $T:\mathrm{Ball}%
(X)\rightarrow \mathrm{Ball}(Y)$ is a morphism. One can extend $T$ to a
linear map from $X$ to $Y$ of norm at most $1$ by setting $T(x)=\left\Vert
x\right\Vert T\left( x/\left\Vert x\right\Vert \right) $ for any nonzero $%
x\in X$. Conversely it is clear that if $T:X\rightarrow Y$ is a bounded
linear map with $\left\Vert T\right\Vert \leq 1$ then the restriction of $T$
to $\mathrm{Ball}(X)$ is a morphism. We can therefore identify morphisms
with bounded linear maps with norm at most one. If $T:X\rightarrow Y$ is a
bounded linear map of norm at most $1$ and $0\leq \delta \leq 1$, then $%
I(T)\leq \delta $ as in Definition \ref{Definition:I} if and only if $%
\left\Vert Tx\right\Vert \geq \left\Vert x\right\Vert -\delta $ whenever $%
\left\Vert x\right\Vert \leq 2$, which in turn happens if and only if $T$ is
injective and $\left\Vert T^{-1}\right\Vert \leq 1+\delta $.

It follows from the geometric version of the Hahn-Banach theorem that if $%
\bar{a}$ is a tuple in $\mathrm{Ball}(X)$ then the substructure generated by 
$\bar{a}$ according to Definition \ref{Definition:substructure} is the unit
ball of the linear span of $\bar{a}$ inside $X$. We declare a tuple $\bar{a}$
to be a basic tuple if and only if it is linearly independent. A simple
calculation shows that such a notion of basic tuple satisfies the
requirements of Subsection \ref{Subsection:basic}.

Let $\mathcal{I}$ be the collection of Banach spaces $\ell _{n}^{\infty }$
for $n\in \mathbb{N}$, which are precisely the injective finite-dimensional
Banach spaces. It is easy to verify that Conditions (1) and (2) of
Subsection \ref{Subsection:injective} hold in this context. This shows that
the class of finite-dimensional Banach spaces is a Fra\"{\i}ss\'{e} class.
The corresponding limit is the Gurarij space first constructed by Gurarij 
\cite{gurarij_spaces_1966} and proved to be unique by Lusky \cite%
{lusky_gurarij_1976}; see also \cite{lusky_separable_1977}.

The following well known fact is a consequence of classical results of
Lindenstrauss \cite{lindenstrauss_extension_1964}, Lazar--Lindenstrauss \cite%
{lazar_banach_1966,lazar_banach_1971}, and Michael--Pe\l czy\'{n}ski \cite%
{michael_separable_1966}:

\begin{fact}
\label{Fact:Lindenstrauss}For a separable Banach space the following
conditions are equivalent:

\begin{enumerate}
\item $X$ is approximately injective according to Definition \ref%
{Definition:approx-inj};

\item $X$ is an $\mathcal{I}$-structure according to Definition \ref%
{Definition:I-structure};

\item $X$ is a rigid $\mathcal{I}$-structure according to Definition \ref%
{Definition:I-structure}

\item $X$ is an isometric predual of an $L^{1}$ space;

\item $X$ is linearly isometric to the limit of an inductive sequence of
finite-dimensional injective Banach spaces.
\end{enumerate}
\end{fact}

When $X$ satisfies the equivalent conditions of Fact \ref{Fact:Lindenstrauss}%
, it is called a\emph{\ Lindenstrauss space}. It follows from Theorem \ref%
{Theorem:retracts} that a separable Banach space is a Lindenstrauss space if
and only if it is isometric to a $1$-complemented subspace of $\mathbb{G}$.
This recovers a classical result of Wojtaszczyk \cite%
{wojtaszczyk_remarks_1972}.

A Banach space $X$ is existentially closed (resp.\ positively existentially
close) if for any isometric inclusion $X\subset Y$ and quantifier-free
formula (resp.\ atomic formula) $\varphi \left( x,b\right) $ for $b\in 
\mathrm{Ball}(X)$ one has that $\inf_{a}\varphi \left( a,b\right) $ has the
same value when $a$ ranges in the unit ball of $X$ or the unit ball of $Y$;
see \cite[Subsection 4.4]{goldbring_model-theoretic_2015}. It is clear that
Condition (6) of Proposition \ref{Proposition:characterize-limit} can be
expressed by a first order formula in the language of Banach spaces.
Therefore Proposition \ref{Proposition:characterize-limit} shows that the
Gurarij Banach space is the unique separable model of its first order theory
as well as the only separable existentially closed Banach space, a fact
already proved in \cite{ben_yaacov_generic_2016}. Applying stable
homogeneity of $\mathbb{G}$ and \cite[Proposition 13.6]%
{ben_yaacov_model_2008} one can recover the following result from \cite%
{ben_yaacov_generic_2016}:\ the theory of $\mathbb{G}$ admits elimination of
quantifiers, and it is the model completion of the theory of Banach spaces.
Finally the characterization of Lindenstrauss spaces mentioned above shows
that a separable Banach space $X$ is Lindenstrauss if and only if it is
positively existentially closed.

\begin{definition}
A\emph{\ compact absolutely convex set} is a compact subset $K$ of a real
locally convex topological vector space with the property that $\lambda
x+\mu y\in K$ whenever $x,y\in K$ and $\lambda ,\mu \in \mathbb{R}$ are such
that $\left\vert \lambda \right\vert +\left\vert \mu \right\vert \leq 1$. If 
$K$ is a compact absolutely convex set and $F\subset K$, then the \emph{%
absolutely convex hull} of $F$ is the smallest absolutely convex subset of $%
K $ containing $F$.
\end{definition}

Let $\sigma :K\rightarrow K$ be the involution $p\mapsto -p$. A function $%
f:K\rightarrow \mathbb{R}$ is \emph{symmetric }if $f\circ \sigma =-f$. More
generally a function between compact absolutely convex sets is symmetric if
it commutes with the involution. Similarly, a signed\ Borel measure $\mu $
on $K$ is \emph{symmetric }if the pushforward $\sigma \mu $ of $\mu $ under $%
\sigma $ is equal to $-\mu $.

If $X$ is a Banach space, then the unit ball $\mathrm{Ball}(X^{\ast })$ of
the dual space of $X$ is a compact absolutely convex set. Suppose that $K$
is a compact absolutely convex set. We denote by $A_{\sigma }(K)$ the space
of continuous symmetric affine functions from $K$ to $\mathbb{R}$. The map
from $K$ to $\mathrm{Ball}\left( A_{\sigma }(K)^{\ast }\right) $ mapping $p$
to the evaluation functional at $p$ is an affine symmetric homeomorphism 
\cite[Lemma 1]{lazar_unit_1972}. Furthermore the assignment $K\mapsto
A_{\sigma }(K)$ is a contravariant equivalence of categories from the
category of Banach spaces and linear contractive maps to the category of
compact absolutely convex sets and continuous symmetric affine functions. In
the following we will assume all the Banach spaces to be separable, and all
the compact absolutely convex sets to be metrizable.

\begin{definition}
A \emph{Lazar simplex }is a compact absolutely convex set that is
symmetrically affinely homeomorphic to $\mathrm{Ball}(X^{\ast })$ for some
Lindenstrauss space $X$.
\end{definition}

Lazar provided in \cite{lazar_unit_1972}---see also \cite[Theorem 3.2]%
{effros_class_1971}---the following characterization of Lindenstrauss
simplices in terms of representing measures, similar in spirit to the
characterization of Choquet simplices in terms of representing probability
measures: a compact absolutely convex set $K$ is a Lazar simplex if and only
if given any two boundary Borel probability measures $\mu _{1},\mu _{2}$ on $%
K$ with the same barycenter one has that $\mu _{1}-\sigma \mu _{1}=\mu
_{2}-\sigma \mu _{2}$ or, equivalently, $\int fd\mu _{1}=\int fd\mu _{2}$
for any $f\in A_{\sigma }(K)$. We call the Lazar simplex $\mathrm{Ball}(%
\mathbb{G})$ associated with the Gurarij space the \emph{Lusky simplex}, and
denote it by $\mathbb{L}$.

Suppose that $X$ is a Banach space. Two elements $p,q\in X^{\ast }$ are
called \emph{codirectional }\cite{alfsen_structure_1972-1} (or without
cancellation \cite{effros_class_1971}) if $\left\Vert p+q\right\Vert
=\left\Vert p\right\Vert +\left\Vert q\right\Vert $. Several equivalent
characterization of codirectional functionals are provided in \cite[Lemma 2.3%
]{alfsen_structure_1972-1} and \cite[Lemma 4.1]{effros_class_1971}. An $L$-%
\emph{projection} is an idempotent map $P:X^{\ast }\rightarrow X^{\ast }$
such $\left\Vert x\right\Vert =\left\Vert P(x)\right\Vert +\left\Vert
x-P(x)\right\Vert $ for every $x\in X^{\ast }$. A subspace $J$ of $X^{\ast }$
is called an $L$-\emph{ideal} if it is the range of an $L$-projection. When
such an $L$-projection exists, it is necessarily unique \cite[Proposition 2.1%
]{harmand_banach_1984}. A subspace of a Banach space $X$ is an $M$-\emph{%
ideal} if its annihilator is an $L$-ideal of $X^{\ast }$. A complete survey
on the theory of $M$-ideals and $L$-ideals can be found in \cite%
{harmand_M-ideals_1993}.

\begin{definition}
\label{Definition:biface}Suppose that $K$ is a compact absolutely convex
set. A subset $H$ of $K$ is a \emph{biface }if it is convex and symmetric, $%
\left\Vert p\right\Vert ^{-1}p\in H$ whenever $p\in H$ is nonzero, and if $%
q_{0},q_{1}\in K$ are codirectional and $q_{0}+q_{1}\in H$ one has that $%
q_{0},q_{1}\in H$.
\end{definition}

A biface of $K$ is \emph{trivial} if $H=\left\{ 0\right\} $ and \emph{proper}
if $H\neq K$. When $X$ is a Lindenstrauss space, $K=\mathrm{Ball}(X^{\ast })$%
, and $H\subset K$ is a w*-closed absolutely convex subset, then $H$ is a
biface if and only if it is the absolutely convex hull of a face of $K$ \cite%
{gleit_note_1972}, if and only if the linear span of $H$ in $X^{\ast }$ is
an $L$-ideal \cite[\S 6]{alfsen_structure_1972-2}. Furthermore in this case
one has that $J\cap K=H$ \cite[Lemma 2.1]{lazar_banach_1971}.

Recall that a Banach space has the \emph{metric approximation property} if
its identity map is the pointwise limit of finite rank linear contractions.
Clearly any Lindenstrauss space has the metric approximation property. The
following proposition collects several equivalent characterizations of $M$%
-ideals and bifaces in Lindenstrauss spaces.

\begin{proposition}
\label{Proposition:characterize-biface}Assume that $Z,X$ are separable
Lindenstrauss spaces, and $P:Z\rightarrow X$ is a quotient mapping. Let $%
P^{\dagger }$ the corresponding dual map from $X^{\ast }$ to $Z^{\ast }$.
Let $K\ $be the Lazar simplex $\mathrm{Ball}\left( Z^{\ast }\right) $, and $%
H $ be the image of $\mathrm{Ball}(X^{\ast })$ under $P^{\dag }$. Let also $%
N $ be the kernel of $P$, and $N^{\bot }\subset Z^{\ast }$ be the
annihilator of $N$. Observe that $N^{\bot }$ coincides with the image of $%
X^{\ast }$ under $P^{\dag }$, as well as with the linear span of $H$ inside $%
Z^{\ast }$. The following statements are equivalent:

\begin{enumerate}
\item $N$ is an $M$-ideal of $X$;

\item whenever $\varepsilon >0$, $E\subset F$ are finite-dimensional Banach
spaces, $g:F\rightarrow X$ is a linear contraction and $f:E\rightarrow Z$ is
a linear isometry such that $P\circ f=g|_{E}$, then there exists a linear
contraction $\widehat{g}:F\rightarrow Z$ such that $P\circ \widehat{g}=g$
and $\left\Vert \widehat{g}|_{E}-f\right\Vert \leq \varepsilon $;

\item whenever $\varepsilon >0$, $A$ is a separable Banach space with the
metric approximation property, $E\subset A$ is a finite-dimensional
subspace, and $f:E\rightarrow Z$ and $g:A\rightarrow X$ are linear
contractions such that $\left\Vert P\circ f-g|_{E}\right\Vert <\varepsilon $%
, then there exists a linear contraction $\widehat{g}:A\rightarrow Z$ such
that $P\circ \widehat{g}=g$ and $\left\Vert \widehat{g}|_{E}-f\right\Vert
<6\varepsilon $;

\item for any subspace $E$ of $Z$, $\varepsilon \geq 0$, one has that $%
\left\Vert P(x)\right\Vert \geq \left( 1-\varepsilon \right) \left\Vert
x\right\Vert $ for any $x\in E$ if and only if there exists a linear
contraction $\eta :X\rightarrow Z$ such that $P\circ \eta $ is the identity
map of $X$ and $\left\Vert \eta \circ P|_{E}-id_{E}\right\Vert \leq
\varepsilon $

\item for any $\varepsilon >0$, $y\in Z$ and $u\in N$ such that $\left\Vert
y\right\Vert =\left\Vert u\right\Vert =1$, there exists $v\in Z$ such that $%
\left\Vert P\left( v\right) \right\Vert \leq \varepsilon $ and $\left\Vert
v-y\pm u\right\Vert \leq 1+\varepsilon $;

\item $H$ is a biface of $K$.
\end{enumerate}
\end{proposition}

\begin{proof}
In the proof we identify $Z$ with $A_{\sigma }\left( K\right) $ and $X$ with 
$A_{\sigma }(H)$. Under these identifications $P$ is just the restriction
mapping $A_{\sigma }\left( K\right) \rightarrow A_{\sigma }(H)$, $f\mapsto
f|_{H}$. The equivalence of (6) and (1) is proved in \cite[\S 6]%
{alfsen_structure_1972-2}. The equivalence of (6) and (5) is essentially 
\cite[Proposition 3]{lusky_construction_1979}. The implication (1)$%
\Rightarrow $(3) can be proved similarly as \cite[Theorem 2.6]%
{choi_lifting_1977} using \cite[Lemma 2.5]{choi_lifting_1977}. We prove the
other nontrivial implications below.

(2)$\Rightarrow $(1) Suppose that $y_{1},y_{2},y_{3}\in \mathrm{Ball}\left(
N\right) $ and $x\in \mathrm{Ball}(Z)$ and $\varepsilon >0$. In view of the
equivalence (i)$\Leftrightarrow $(iv) in \cite[Theorem 2.2]%
{harmand_M-ideals_1993}, it is enough to prove that there exists $y\in 
\mathrm{Ball}\left( N\right) $ such that $\left\Vert x+y^{(\ell
)}-y\right\Vert \leq 1+\varepsilon $ for $\ell \in \left\{ 1,2,3\right\} $.
Let $E=span\left\{ y_{1},y_{2},y_{3},x\right\} \subset Z$. Consider the
Banach space $F$ obtained from $E\oplus \mathbb{R}$ and the collection of
maps $\left( z,\lambda \right) \mapsto \varphi (z)+\lambda s$ where $\varphi
:E\rightarrow \mathbb{R}$ is a linear contraction and $s\in \left[ -1,1%
\right] $ is such that $\left\vert \varphi (x+y^{(\ell )})-s\right\vert \leq
1$ for $\ell \in \left\{ 1,2,3\right\} $. Define also the map $%
g:F\rightarrow X$ by $\left( z,\lambda \right) \mapsto P(z)$. Observe that
the canonical inclusion $E\subset F$ is isometric and the map $g$ is a
contraction such that $g|_{E}=P$. Hence by hypothesis there exists a linear
contraction $\widehat{g}:F\rightarrow Z$ such that $P\circ \widehat{g}=g$
and $\left\Vert \widehat{g}|_{E}-\iota _{E}\right\Vert \leq \varepsilon $,
where $\iota _{E}:E\rightarrow Z$ is the inclusion map. The element $y:=%
\widehat{g}\left( 0,1\right) $ is as desired.

(6)$\Rightarrow $(4): Suppose that $E\subset Z$ is a linear subspace such
that $\left\Vert P(x)\right\Vert \geq \left( 1-\varepsilon \right)
\left\Vert x\right\Vert $ for every $x\in E$. Let $k$ be an element of $K$.
We observe that there exists $h\in H$ such that $\left\Vert \left(
k-h\right) |_{E}\right\Vert \leq \varepsilon $. The assumption implies that $%
E\cap N=\left\{ 0\right\} $. Define $h\in \left( E+N\right) ^{\ast }$ by
setting $h\left( e+n\right) =\left( 1-\varepsilon \right) k\left( e\right) $%
. We have that $\left\vert k\left( e\right) \right\vert \leq \left\Vert
e\right\Vert \leq \left( 1-\varepsilon \right) ^{-1}\left\Vert P\left(
e\right) \right\Vert \leq \left( 1-\varepsilon \right) ^{-1}\left\Vert
e+n\right\Vert $. Thus $\left\Vert h\right\Vert \leq 1$ and hence it extends
to a linear functional on $X$ of norm at most $1$ that belongs to $H=\mathrm{%
Ball}(X^{\ast })\cap N^{\bot }$. It is clear from the definition that $%
\left\Vert \left( k-h\right) |_{E}\right\Vert \leq \varepsilon $. Define the
function defined by%
\begin{equation*}
\varphi :k\mapsto \left\{ h\in H:\left\Vert \left( k-h\right)
|_{E}\right\Vert \leq \varepsilon \right\}
\end{equation*}%
for $k\in K$. Observe that $\varphi $ satisfies the assumptions of \cite[%
Theorem 2.2]{lazar_banach_1971}. Hence there exists a continuous affine
symmetric function $Q:K\rightarrow H$ such that $Q|_{H}$ is the identity map
of $H$ and $\left\Vert \left( Q\left( k\right) -k\right) |_{E}\right\Vert
\leq \varepsilon $ for every $k\in K$. One can thus define $\eta :A_{\sigma
}(H)\rightarrow A_{\sigma }(K)$ by $\eta \mapsto \eta \circ Q$.

(5)$\Rightarrow $(6) Suppose that $q_{0},q_{1}\in K$ and $p\in H$ are such
that $\left\Vert q_{0}\right\Vert +\left\Vert q_{1}\right\Vert =\left\Vert
p\right\Vert $, $t\in \left( 0,1\right) $, and $tq_{0}+\left( 1-t\right)
q_{1}=p$. We want to prove that $q_{0}\in H$. Fix $u\in N$ of norm $1$ and $%
y\in Z$ of norm $1$ such that $p(y)=1$. Observe that $%
q_{0}(y)=q_{1}(y)=p(y)=1$. It is enough to prove that $q_{0}\left( u\right)
\leq 3\varepsilon $. By assumption there exists $v\in Z$ such that $%
\left\Vert P\left( v\right) \right\Vert \leq \varepsilon $ and $\left\Vert
v-y\pm u\right\Vert \leq 1+\varepsilon $. Then we have $\left\Vert
v-y\right\Vert \leq 1+\varepsilon $. Thus%
\begin{equation*}
q_{0}\left( v\right) \leq \varepsilon -q_{1}\left( v\right) \leq \varepsilon
+\left( 1+\varepsilon \right) -q_{0}(y)\leq 2\varepsilon
\end{equation*}%
and%
\begin{equation*}
q_{0}\left( u\right) =q_{0}\left( y+u-v\right) -q_{0}(y)-q_{0}\left(
v\right) \leq 3\varepsilon \text{.}
\end{equation*}%
This concludes the proof.
\end{proof}

\begin{remark}
\label{Remark:characterize-biface}The equivalence of (1)--(3) in\
Proposition \ref{Proposition:characterize-biface} holds even without the
assumption that $Z,X$ are Lindenstrauss spaces. Furthermore if $H$ is a
closed biface of a metrizable Lazar simplex $K$ then the restriction mapping 
$A_{\sigma }\left( K\right) \rightarrow A_{\sigma }(H)$, $f\mapsto f|_{H}$
is automatically a complete quotient mapping by \cite[Corollary 1]%
{lazar_banach_1971}.
\end{remark}

The equivalence of the conditions in Proposition \ref%
{Proposition:characterize-biface} justifies the following definition.

\begin{definition}
\label{Definition:facial}If $X,Z$ are Lindenstrauss spaces, and $%
P:Z\rightarrow X$ is a quotient linear mapping, then we say that $P$ is a 
\emph{facial quotient} if it satisfies any of the equivalent conditions of
Proposition \ref{Proposition:characterize-biface}. A facial quotient is 
\emph{trivial }if it is an isometric isomorphism.
\end{definition}

Let us now fix a separable Lindenstrauss space $X$ and consider the generic
operator $\Omega _{\mathbb{G}}^{X}:\mathbb{G}\rightarrow X$ constructed as
in Section \ref{Section:states}. It follows from the characterization of the
generic state from Subsection \ref{Subsection:universal-state} together with
Proposition \ref{Proposition:characterize-biface} that $\Omega _{\mathbb{G}%
}^{X}$ is a nontrivial facial quotient with kernel isometrically isomorphic
to $\mathbb{G}$. Therefore any Lazar simplex is symmetrically affinely
homeomorphic to a closed proper biface of $\mathbb{L}=\mathrm{Ball}(\mathbb{G%
})$ \cite[Corollary 4]{lusky_construction_1979}. In the rest of the section
we will prove that, conversely, any nontrivial facial quotient $P:\mathbb{G}%
\rightarrow X$ belongs to the $\mathrm{Aut}(\mathbb{G})$-orbit of $\Omega _{%
\mathbb{G}}^{X}$.

Let us consider initially the case $X=\mathbb{R}$. In this case we have that 
$\Omega _{\mathbb{G}}^{\mathbb{R}}$ is an extreme point of $\mathbb{L}:=%
\mathrm{Ball}(\mathbb{G}^{\ast })$. Hence the extreme boundary of $\mathbb{L}
$ is dense in $\mathbb{L}$. We now want to observe that, conversely, any
Lazar simplex with dense extreme boundary is symmetrically affinely
homeomorphic to $\mathbb{L}$.

\begin{proposition}[{\protect\cite[Theorem 6.1]{lindenstrauss_poulsen_1978}}]

\label{Proposition:dense-Lusky}Suppose that $L$ is a nontrivial metrizable
Lazar simplex with dense extreme boundary. Then $L$ is symmetrically
affinely homeomorphic to $\mathbb{L}$.
\end{proposition}

\begin{proof}
Set $G=A\left( L\right) $. We want to prove that $G$ is isometrically
isomorphic to $\mathbb{G}$. Suppose that $\varepsilon >0$, and $n\in \mathbb{%
N}$. Let $\phi :\ell _{n}^{\infty }\rightarrow \ell _{n+1}^{\infty }$ and $%
f:\ell _{n}^{\infty }\rightarrow G$ be linear isometries. We want to prove
that there exists an isometric linear map $\widehat{f}:\ell _{n+1}^{\infty
}\rightarrow G$ such that $\left\Vert \widehat{f}\circ \phi -f\right\Vert
<\varepsilon $. This will suffice in view of the characterization of the
limit provided by Proposition \ref{Proposition:characterize-limit}. In view
of Proposition \ref{Proposition:characterize-biface}, it is enough to find a
facial quotient map $Q:G\rightarrow \ell _{n+1}^{\infty }$ such that $%
\left\Vert Q\circ f-\phi \right\Vert <\varepsilon $. Fix $\eta >0$. Choose
standard bases $e_{1}^{n},\ldots ,e_{n}^{n}$ of $\ell _{n}^{\infty }$ and $%
e_{1}^{n+1},\ldots ,e_{n+1}^{n+1}$ of $\ell _{n+1}^{\infty }$ and $%
a_{1},\ldots ,a_{n}\in \mathbb{R}$ such that such $\left\vert
a_{1}\right\vert +\cdots +\left\vert a_{n}\right\vert \leq 1$ and $\phi
\left( e_{i}^{n}\right) =e_{i}^{n+1}+a_{i}^{n+1}e_{n+1}^{n+1}$. For every $%
i=1,2,\ldots ,n$ pick $s_{i}\in \partial _{e}L$ such that $s_{i}\left(
f\left( e_{i}^{n}\right) \right) =1$. Since $L$ is a nontrivial metrizable
Lazar simplex with dense extreme boundary, one can find then $s_{n+1}\in
\partial _{e}L$ such that $s_{n+1}$ does not belong to the absolutely convex
hull of $\left\{ s_{1},\ldots ,s_{n}\right\} $, and%
\begin{equation*}
\left\vert s_{n+1}\left( f\left( e_{i}^{n}\right) \right)
-\sum_{j=1}^{n}a_{j}s_{j}\left( f\left( e_{i}^{n}\right) \right) \right\vert
<\eta 
\end{equation*}%
for $i=1,\ldots ,n$. Let $Q:G\rightarrow \ell _{n+1}^{\infty }$ be the map $%
x\mapsto \left( s_{1}(x),\ldots ,s_{n+1}(x)\right) $. By \cite[Proposition
2.3]{alfsen_structure_1972-2}, $Q$ is a quotient mapping. Observe that $%
\left\Vert Q\circ f-\phi \right\Vert <\varepsilon $ for $\eta $ small
enough. For $k=1,2,\ldots ,n+1$ define $H_{k}$ to be $\left\{ \lambda
s_{k}:\lambda \in \left[ -1,1\right] \right\} $, and observe that $H_{k}$ is
a closed biface since $s_{k}$ is an extreme point of $L$. Set now $\widehat{H%
}$ to be the convex hull of $H_{1},\ldots ,H_{n+1}$. By \cite[Proposition 4.6%
]{effros_class_1971}, $\widehat{H}$ is a closed biface of $L$. Since $%
Q^{\dagger }\mathrm{Ball}\left( \ell _{n}^{1}\right) =\widehat{H}$, we have
that $Q$ is a facial quotient mapping. This concludes the proof.
\end{proof}

\begin{proposition}
\label{Proposition:characterize-operator-G}Suppose that $X$ is a separable
Lindenstrauss space, and $P:\mathbb{G}\rightarrow X$ is a contractive linear
map. Then $P$ belongs to the $\mathrm{\mathrm{Aut}}(\mathbb{G})$-orbit of $%
\Omega _{\mathbb{G}}^{X}$ if and only if $P$ is a nontrivial facial quotient.
\end{proposition}

\begin{proof}
We have already observe that $\Omega _{\mathbb{G}}^{X}$ is a nontrivial
facial quotient. We prove the converse implication. Let $H$ be the image of $%
\mathrm{Ball}(X^{\ast })$ under the dual map $P^{\dagger }$. Suppose that $%
\varepsilon >0$, and $n\in \mathbb{N}$. Let $\phi :\ell _{n}^{\infty
}\rightarrow \ell _{n+1}^{\infty }$ and $f:\ell _{n}^{\infty }\rightarrow 
\mathbb{G}$ be linear isometries, and $s:\ell _{n+1}^{\infty }\rightarrow X$
be a linear map such that $P\circ f=s\circ \phi $. In view of the
characterization of the generic state from Subsection \ref%
{Subsection:universal-state}, it is enough to prove that there exists a
linear isometry $\widehat{f}:\ell _{n+1}^{\infty }\rightarrow X$ such that $%
P\circ \widehat{f}=s$ and $\left\Vert \widehat{f}\circ \phi -f\right\Vert
<\varepsilon $. By Proposition \ref{Proposition:characterize-biface}, it is
enough to prove that there exists a linear map $Q:\mathbb{G}\rightarrow \ell
_{n+1}^{\infty }$ such that $Q\oplus P:\mathbb{G}\rightarrow \ell
_{n+1}^{\infty }\oplus ^{\infty }X$ is a facial quotient, and $\left\Vert
Q\circ f-\phi \right\Vert <\varepsilon $. For this purpose one can proceed
as in the proof of Proposition \ref{Proposition:dense-Lusky} and define $%
s_{1},\ldots ,s_{n+1}$. Let then for $k=1,2,\ldots ,n$, $t_{k}\in \partial
_{e}\mathbb{L}\backslash H$ such that $\left\vert t_{k}\left( e_{i}\right)
-s_{k}\left( e_{i}\right) \right\vert \leq \eta $ for every $i=1,2,\ldots ,n$%
. Define now $H_{k}$ to be $\left\{ \lambda t_{k}:\lambda \in \left[ -1,1%
\right] \right\} $ for $k=1,2,\ldots ,n+1$, and $\widehat{H}$ to be convex
hull of $H_{1},\ldots ,H_{n+1}$ and $H$. As in the proof of Proposition \ref%
{Proposition:dense-Lusky}, $\widehat{H}$ is a closed biface. Let $Q:\mathbb{G%
}\rightarrow \ell _{n+1}^{\infty }$ be the map $x\mapsto \left(
t_{1}(x),\ldots ,t_{n+1}(x)\right) $, and observe that it is a quotient
mapping. The image of $\mathrm{Ball}\left( \ell _{n}^{\infty }\oplus
^{\infty }X\right) $ under the dual map of $P\oplus Q$ is $\widehat{H}$.
This shows that $Q$ is a facial quotient. For $\eta >0$ small enough, one
has that $\left\Vert Q\circ f-\phi \right\Vert <\varepsilon $, concluding
the proof.
\end{proof}

The following corollary is an immediate consequence of Remark \ref%
{Remark:lift} and Proposition \ref{Proposition:characterize-operator-G}.

\begin{corollary}
\label{Corollary:homogeneity-L}Any symmetric affine homeomorphism between
closed proper bifaces of $\mathbb{L}$ extends to a symmetric affine
homeomorphism of $\mathbb{L}$.
\end{corollary}

If $L$ is a Lazar simplex, and $Z\subset \partial _{e}L$ is a compact
subset, then the absolutely convex hull $H$ of $Z$ is a closed biface of $L$
such that $\partial _{e}H=Z$ \cite[Theorem 5.8]{effros_class_1971}. By \cite[%
Lemma 3.1]{effros_class_1971} one can identify $A_{\sigma }(H)$ with the
space $C_{\sigma }(Z)$ of continuous real-valued symmetric functions on $Z$.

\begin{corollary}
A symmetric homeomorphism between proper compact subsets of $\partial _{e}%
\mathbb{L}$ extends to a symmetric affine homeomorphism of $\mathbb{L}$.
\end{corollary}

\begin{proof}
Suppose that, for $i=0,1$, $Z_{i}\subset \partial _{e}L$ is a proper compact
subset, $H_{i}$ is the absolutely convex hull of $Z_{i}$, and $\varphi
:Z_{0}\rightarrow Z_{1}$ is a symmetric homeomorphism. Then $\varphi $
induces an isometric isomorphism $\alpha $ from $C_{\sigma }\left(
Z_{1}\right) $ to $C_{\sigma }\left( Z_{0}\right) $. Since as remarked above
one can identify $C_{\sigma }\left( Z_{i}\right) $ with $A_{\sigma }\left(
H_{i}\right) $, $\alpha $ in turn induces a symmetric affine homeomorphism $%
\widehat{\varphi }$ from $H_{0}$ to $H_{1}$ that extends $\varphi $.
Applying Corollary \ref{Corollary:homogeneity-L} one can deduce that $%
\widehat{\varphi }$ can be extended to a symmetric affine homeomorphism of $%
\mathbb{L}$.
\end{proof}

Suppose that $X$ is a Lindenstrauss space, $K=\mathrm{Ball}(X^{\ast })$ is
the associated Lazar simplex, and $H$ is a proper closed biface of $K$. Let $%
N$ be the linear span of $H$ inside $X^{\ast }$ and $e:X^{\ast }\rightarrow
X^{\ast }$ be the corresponding $L$-projection. Then the range of $I-e$ is
the complementary (convex) cone $N^{\prime }$ of $N$; see \cite[Proposition
3.1]{alfsen_structure_1972-1}. The quotient mapping $X^{\ast }\rightarrow
X^{\ast }/N$ induces a linear isometry from $N^{\prime }$ onto $X^{\ast }/N$ 
\cite[Proposition 1.14]{alfsen_structure_1972-2}. The complementary biface
of $H$ is the intersection of $N^{\prime }$ with $\mathrm{Ball}(X^{\ast })$.

\begin{corollary}
\label{Corollary:dual-biface}Suppose that $H$ is a proper closed biface of $%
\mathbb{L}$. Endow the complementary biface $H^{\prime }$ with the
w*-topology induced by $a\in \mathbb{G}$ such that $a|_{H}\equiv 0$. Then $%
H^{\prime }$ is affinely homeomorphic to $\mathbb{G}$.
\end{corollary}

\begin{proof}
Suppose that $H$ is a closed proper biface of $\mathbb{L}$. Consider $%
J=\left\{ f\in A_{\sigma }\left( \mathbb{L}\right) :f|_{H}=0\right\} $ and
set $N:=J^{\bot }$. Observe that $N$ coincides with the linear span of $H$
inside $\mathbb{G}^{\ast }$. Let $N^{\prime }$ be the complementary cone of $%
N$, and $H^{\prime }=H\cap K$ be the complementary biface of $H$. By
Proposition \ref{Proposition:characterize-operator-G}, $J$ is isometrically
isomorphic to $\mathbb{G}$, and $\mathrm{Ball}\left( J^{\ast }\right) $ is
symmetrically affinely homeomorphic to $\mathbb{L}$. The inclusion $J\subset 
\mathbb{G}$ induces by duality w*-continuous linear map $\varphi :\mathbb{G}%
^{\ast }\rightarrow J^{\ast }$ We claim that the restriction of $\varphi $
to $N^{\prime }$ is 1:1 and, in fact, isometric. Indeed, as observed in the
proof of \cite[Lemma 3.4(b)]{alfsen_structure_1972-2}, one can identify $%
\varphi $ with the quotient mapping $\mathbb{G}^{\ast }\rightarrow \mathbb{G}%
^{\ast }/N$. Therefore we have that $H^{\prime }$ with the topology
described in the statement is symmetrically affinely homeomorphic to $%
\mathrm{Ball}(\mathbb{G})=\mathrm{Ball}\left( J^{\ast }\right) =\mathbb{L}$.
\end{proof}

Theorem \ref{Theorem:G} and Theorem \ref{Theorem:universal-projection-G} now
follow from the general results of Sections \ref{Section:injective}, \ref%
{Section:retracts}, \ref{Section:operators}, \ref{Section:states} together
with remarks above.

\subsection{Complex Banach spaces\label{Subsection:Banach-complex}}

One can regard complex Banach spaces as structures in a suitable language $%
\mathcal{L}$ similarly as real Banach spaces. In this section we will assume
all the Banach spaces to be complex and separable. The finite-dimensional
injective complex Banach spaces are precisely those of the form $\ell
_{n}^{\infty }$ (finite $\infty $-sum of $n$ copies of $\mathbb{C}$) for
some $n\in \mathbb{N}$. The Fra\"{\i}ss\'{e} limit of the class of
finite-dimensional complex Banach spaces is the complex Gurarij space $%
\mathbb{G}$. The analogue of Fact \ref{Fact:Lindenstrauss} for complex
Banach spaces holds due to results of Hustad \cite{hustad_intersection_1974}%
, Olsen \cite{olsen_edwards_1976}, and Nielsen-Olsen \cite%
{nielsen_complex_1977}. We call a complex Banach space satisfying the
(complex analogs of) any of the equivalent properties of Fact \ref%
{Fact:Lindenstrauss} a complex Lindenstrauss space.

\begin{definition}
A\emph{\ compact} \emph{circled convex set }is a compact subset $K$ of a
complex locally convex topological vector space such that $\lambda x+\mu
y\in K$ whenever $x,y\in K$ and $\lambda ,\mu \in \mathbb{C}$ are such that $%
\left\vert \lambda \right\vert +\left\vert \mu \right\vert \leq 1$.
\end{definition}

Let $K$ be a compact circled convex set and $\xi \in \mathbb{T}$. We denote
by $\sigma _{\xi }:K\rightarrow K$ the map $p\mapsto \xi p$. A
complex-valued function $f$ on $K$ is called $\mathbb{T}$-\emph{invariant}
if $f\circ \sigma _{\xi }=f$ for every $\xi \in \mathbb{T}$, and $\mathbb{T}$%
-\emph{homogeneous} if $f\circ \sigma _{\xi }=\xi f$ for every $\xi \in 
\mathbb{T}$. Similar definitions apply to complex Borel measures on $K$. The
map $f\mapsto \mathrm{inv}_{\mathbb{T}}f=\int \left( f\circ \sigma _{\xi
}\right) d\xi $ is a norm $1$ projection from $C(K)$ onto the space of
continuous $\mathbb{T}$-invariant functions, while the map $f\mapsto \mathrm{%
hom}_{\mathbb{T}}f=\int \xi \left( f\circ \sigma _{\xi }\right) d\xi $ is a
norm $1$ projection onto the space of continuous $\mathbb{T}$-homogeneous
functions. The adjoints of these projections give w*-continuous projections $%
\mu \mapsto \mathrm{inv}_{\mathbb{T}}\mu $ and $\mu \mapsto \mathrm{hom}_{%
\mathbb{T}}\mu $ of the space $M(K)$ of complex Borel probability measures
onto the spaces of $\mathbb{T}$-invariant and $\mathbb{T}$-homogeneous
measures, respectively. A continuous map $\phi :K_{0}\rightarrow K_{1}$ is $%
\mathbb{T}$-homogeneous if $\phi \circ \sigma _{\xi }=\sigma _{\xi }\circ
\phi $ for every $\xi \in \mathbb{T}$. Let $A_{\mathbb{T}}(K)\subset C(K)$
be the space of continuous $\mathbb{T}$-homogeneous complex-valued functions
on $K$. It follows from the geometric Hahn-Banach theorem for complex Banach
spaces that the map sending $p\in K$ to the corresponding evaluation
functional in $A_{\mathbb{T}}(K)^{\ast }$ is a $\mathbb{T}$-homogeneous
affine homeomorphism of $K$ onto the unit ball of $A_{\mathbb{T}}(K)^{\ast }$%
. Conversely, one can identify a complex Banach space $X$ with $A_{\mathbb{T}%
}(K)$ where $K$ is the unit ball of the dual space of $X$ endowed with the
w*-topology. The function $K\mapsto A_{\mathbb{T}}(K)$ is a contravariant
equivalence of categories from the category of compact circled convex sets
and $\mathbb{T}$-homogeneous affine continuous functions to the category of
complex Banach spaces and complex-linear maps of norm at most $1$. In the
following we assume all compact convex circled sets to be metrizable, and
all the Banach spaces to be separable.

\begin{definition}
An \emph{Effros simplex }is the unit ball of the dual space of a complex
Lindenstrauss space.
\end{definition}

Effros characterized in \cite{effros_class_1974} what we call Effros
simplices: a compact circled convex set $K$ is an Effros simplex if and only
if given boundary probability measures $\mu _{1},\mu _{2}$ of $K$ with the
same barycenter one has that $\mathrm{hom}_{\mathbb{T}}\mu _{1}=\mathrm{hom}%
_{\mathbb{T}}\mu _{2}$.

All the definitions about simplices carry over with no change from the real
to the complex setting, as well as the notions of collinear elements, $L$%
-ideals, and $M$-ideals. The notion of cone of a complex vector space is
defined as in the real case: a subset $C$ of a complex vector space is a
cone if $\lambda x\in C$ for any $x\in C$ and $\lambda \geq 0$. A cone $C$
in a dual Banach space $X^{\ast }$ is \emph{hereditary }if whenever $p,q\in
X^{\ast }$ are collinear and $p+q\in C$ then $p,q\in C$.

\begin{lemma}
\label{Lemma:complex-hereditary}Suppose that $X$ is a complex Lindenstrauss
space and that $J$ is a subspace of $X^{\ast }$. The following assertions
are equivalent.

\begin{enumerate}
\item $W$ is an $L$-ideal;

\item $\left\Vert x+y\right\Vert =\left\Vert x\right\Vert +\left\Vert
y\right\Vert $ for any $x\in J$ and $y\in J^{\prime }$ (the complementary
cone of $J$);

\item $W$ is hereditary.
\end{enumerate}
\end{lemma}

\begin{proof}
The implications (1)$\Rightarrow $(2)$\Rightarrow $(3) are obvious. The
implication (3)$\Rightarrow $(2) is \cite[Lemma 1]{lima_application_1976}%
---after observing that a Banach space is an $E\left( 3\right) $ space if
and only if it is a Lindenstrauss space---while the implication (2)$%
\Rightarrow $(1) is \cite[Theorem 5.5]{lima_intersection_1977}.
\end{proof}

\begin{definition}
\label{Definition:circled-face}A subset $H$ of a compact convex circled set $%
K$ is a $\emph{circled}$ \emph{face} if it is circled convex, it contains $%
\left\Vert x\right\Vert ^{-1}x$ whenever $x\in H$ is nonzero, and if $x,y\in
K$ are codirectional and $x+y\in H$ then $x,y\in H$.
\end{definition}

It follows from Lemma \ref{Lemma:complex-hereditary} and \cite[Proposition
2.1]{ellis_facial_1981} that if $X$ is a Lindenstrauss space and $K$ is the
Effros simplex $\mathrm{Ball}(X^{\ast })$, then the closed circled faces of $%
K$ are precisely the sets of the form $J\cap K$ where $J$ is a w*-closed $L$%
-ideal of $X^{\ast }$. Indeed if $H$ is a closed circled face of $K$, then
the linear span $J$ of $K\ $is a w*-closed $L$-ideal of $X^{\ast }$ such
that $H=J\cap K$. Furthermore if $A$ is a compact subset of $\partial _{e}K$
then the closure of the circled convex hull of $A$ is a circled face of $%
X^{\ast }$. The following is the natural complex analog of Proposition \ref%
{Proposition:characterize-circled-face}, and can be proved with similar
methods by replacing \cite[Theorem 2.2]{lazar_banach_1971} with \cite[%
Theorem 4.2]{olsen_edwards_1976}.

\begin{proposition}
\label{Proposition:characterize-circled-face}Suppose that $P:Z\rightarrow X$
is a quotient mapping between complex Lindenstrauss spaces. If $H$ is the
image of $\mathrm{Ball}(X^{\ast })$ under $P^{\dag }$ and $N$ is the kernel
of $P$, then the following statements are equivalent:

\begin{enumerate}
\item $N$ is an $M$-ideal of $X$;

\item whenever $\varepsilon >0$, $E\subset F$ are finite-dimensional Banach
spaces, $g:F\rightarrow X$ is a linear contraction, and $f:E\rightarrow Z$
is a linear isometry such that $P\circ f=g|_{E}$, then there exists a linear
contraction $\widehat{g}:F\rightarrow Z$ such that $P\circ \widehat{g}=g$
and $\left\Vert \widehat{g}|_{E}-f\right\Vert \leq \varepsilon $;

\item whenever $\varepsilon >0$, $A$ is a separable Banach space with the
metric approximation property, $E\subset A$ is a finite-dimensional
subspace, and $f:E\rightarrow Z$ and $g:A\rightarrow X$ are linear
contractions such that $\left\Vert P\circ f-g|_{E}\right\Vert <\varepsilon $%
, then there exists a linear contraction $\widehat{g}:A\rightarrow Z$ such
that $P\circ \widehat{g}=g$ and $\left\Vert \widehat{g}|_{E}-f\right\Vert
<6\varepsilon $;

\item for any subspace $E$ of $Z$ and $\varepsilon \geq 0$ one has that $%
\left\Vert P(x)\right\Vert \geq \left( 1-\varepsilon \right) \left\Vert
x\right\Vert $ for any $x\in E$ if and only if there exists a linear
contraction $\eta :X\rightarrow Z$ such that $P\circ \eta $ is the identity
map of $X$ and $\left\Vert \eta \circ P|_{E}-id_{E}\right\Vert \leq
\varepsilon $.

\item for any $\varepsilon >0$, $y,u\in Z$ such that $\left\Vert
y\right\Vert =\left\Vert u\right\Vert =1$, and $u\in N$, there exists $v\in
Z $ such that $\left\Vert P\left( v\right) \right\Vert \leq \varepsilon $
and $\left\Vert v-y+\xi u\right\Vert \leq 1+\varepsilon $ for every $\xi \in 
\mathbb{T}$;

\item $H$ is a circled face of $K$.
\end{enumerate}
\end{proposition}

\begin{remark}
Similarly as for real Banach spaces, the equivalence of (1)--(3) in
Proposition \ref{Proposition:characterize-circled-face} holds without the
assumption that $Z,X$ are Lindenstrauss spaces; see Remark \ref%
{Remark:characterize-biface}. Furthermore if $H$ is a closed circled face of
a metrizable Effros simplex $K$, then the restriction mapping $A_{\mathbb{T}%
}\left( K\right) \rightarrow A_{\mathbb{T}}(H)$, $f\mapsto f|_{H}$ is
automatically a complete quotient mapping by \cite[Theorem 4.2]%
{olsen_edwards_1976}.
\end{remark}

As in the real case, we define a quotient mapping $P:Z\rightarrow X$ to be 
\emph{facial quotient }if it satisfies any of the equivalent conditions of
Proposition \ref{Proposition:characterize-circled-face}. As in the real
case, one can deduce the complex analog of Theorem \ref{Theorem:G} and
Theorem \ref{Theorem:universal-projection-G} from\ Proposition \ref%
{Proposition:characterize-circled-face} and the general results from
Sections \ref{Section:injective}, \ref{Section:retracts}, \ref%
{Section:operators}, \ref{Section:states}.

\subsection{Function systems\label{Subsection:function-system}}

A \emph{function system} is an ordered real vector space $V$ endowed with a
distinguished element $e$ that is an \emph{Archimedean order unit }\cite[%
Chapter 2]{alfsen_compact_1971}. This means that, for every $v\in V$,

\begin{itemize}
\item there exists $n\in \mathbb{N}$ such that $-ne\leq v\leq ne$, and

\item if, for every $k\in \mathbb{N}$, $kv\leq e$, then $v\leq 0$.
\end{itemize}

An function system $V$ is naturally endowed with a norm defined by%
\begin{equation*}
\left\Vert v\right\Vert =\inf \left\{ r\in \mathbb{R}_{+}:-re\leq v\leq
re\right\} \text{.}
\end{equation*}%
\emph{We will always assume such a norm to be complete}. A \emph{state }on $%
V $ is a linear function $s$ that is \emph{positive} and \emph{unital}. This
means that $s$ maps positive elements of $V$ to positive real numbers, and
maps the order unit of $V$ to $1$. The space $S(V)$ of states of $V$ is a
w*-compact convex subset of the dual $V^{\ast }$ of $V$. A unital linear
functional on $V$ is positive if and only if it is contractive, and $v\in V$
is positive if and only if $s\left( v\right) $ is positive for every state $%
s $ of $V$. Hence in a function system one can reconstruct the order from
the norm and the order unit. Two function system $V$ and $W$ are order
isomorphic if there exists a surjective unital linear isometry from $V$ to $%
W $.

If $K$ is a compact convex set, then the space $A(K)$ of real-valued
continuous affine functions on $K$ is a function system with its usual order
structure, maximum norm, and the function constantly equal to $1$ as order
unit. Kadison's representation theorem asserts that the map from $V$ to $%
A\left( S(V)\right) $ mapping $v$ to the evaluation function at $v$ is a
surjective unital linear isometry \cite[Theorem II.1.8]{alfsen_compact_1971}%
; see also \cite{kadison_representation_1951,kadison_transformations_1965}.
Furthermore the map $K\mapsto A(K)$ is a contravariant isomorphism from the
category of compact convex sets and continuous affine maps to the category
of function systems and unital contractive linear maps. Using this
observation one can reformulate statements about compact convex sets into
statements about function systems, and vice versa. Considering \emph{complex 
}function systems rather than real function systems does not yield any
substantial difference. Indeed, any complex function system is the
complexification of a real function system, and any complex-linear unital
map between complex function systems is the complexification of a
real-linear unital linear map of the same norm.

We regard function systems as structures in the language of Banach spaces
with an additional constant symbol for the order unit. Basic tuples in this
context are linearly independent tuples whose first element is the order
unit. We consider the collection $\mathcal{I}$ of injective objects
consisting of the spaces $\ell _{n}^{\infty }$ for $n\in \mathbb{N}$ with
the $n$-tuple constantly equal to $1$ as order unit. The following lemma can
be proved as \cite[Theorem 3.5]{effros_lifting_1985} using \cite[Proposition
II.1.14]{alfsen_compact_1971}; see also Lemma \ref{Lemma:perturb-ucp}.

\begin{lemma}
\label{Lemma:perturb-function-system}Suppose that $V,W$ are function
systems, and $f:V\rightarrow W$ is a unital linear map such that $\left\Vert
f\right\Vert \leq 1+\delta $. If $W$ is injective, then there exists a
unital positive linear map $g:V\rightarrow W$ such that $\left\Vert
f-g\right\Vert \leq 2\delta $. If $W$ is arbitrary and $V$ is $n$%
-dimensional, then there exists a unital positive linear map $g:V\rightarrow
W$ such that $\left\Vert f-g\right\Vert \leq 2n\delta $.
\end{lemma}

It follows from Lemma \ref{Lemma:perturb-function-system} and the discussion
above that the all the conditions of Section \ref{Section:injective} are met
with $\varpi (\delta )=2\delta $. The following statement collects together
classical results from the theory of compact convex sets; see \cite%
{alfsen_compact_1971,jellett_homomorphisms_1968}.

\begin{fact}
\label{Fact:Lindenstrauss-system}Suppose that $V$ is a separable function
system. The following conditions are equivalent:

\begin{enumerate}
\item $V$ is a Lindenstrauss space;

\item $V$ is approximately injective as in Definition \ref%
{Definition:approx-inj},

\item $V$ is an $\mathcal{I}$-structure as in Definition \ref%
{Definition:I-structure},

\item $V$ is a rigid $\mathcal{I}$-structure as in Definition \ref%
{Definition:I-structure},

\item $V$ is the direct limit of copies of $\ell _{n}^{\infty }$ for $n\in 
\mathbb{N}$ with unital linear isometries as connective maps,

\item the state space of $V$ is a Choquet simplex.
\end{enumerate}
\end{fact}

We call a function system satisfying the equivalent conditions of Fact \ref%
{Fact:Lindenstrauss-system} above a \emph{simplex space}. All the function
systems below are assumed to be separable, and all the compact convex sets
are assumed to be metrizable. One can conclude from the general results of
Section \ref{Section:general} that finite-dimensional function systems form
a Fra\"{\i}ss\'{e} class. Let us denote by $A(\mathbb{P})$ the corresponding
limit, and by $\mathbb{P}$ the state space of $A(\mathbb{P})$. We will show
below that $\mathbb{P}$ is the unique metrizable Choquet simplex with dense
extreme boundary.

Suppose that $A(K)$ and $A(F)$ are function systems with state spaces $K$
and $F$ respectively. We let $S\left( A(K),A(F)\right) $ be the space of
unital positive linear maps $\phi :A(K)\rightarrow A(F)$ endowed with the
topology of pointwise convergence. Let also $A\left( F,K\right) $ be the
space of continuous affine functions $f:F\rightarrow K$ endowed with the
compact open topology and its natural convex structure. The assignment $\phi
\mapsto \phi ^{\dagger }$ where $\phi ^{\dagger }s=s\circ \phi $ is a
homeomorphism from $S\left( A(K),A(F)\right) $ onto $A\left( F,K\right) $.

We say that a function system $A$ has the metric approximation property if
it has such a property as a Banach space. In view of Lemma \ref%
{Lemma:perturb-function-system} this is equivalent to the assertion that the
identity map of $A$ is the pointwise limit of finite rank unital positive
linear maps.

\begin{proposition}
\label{Proposition:characterize-face}Suppose that $V,W$ are separable
simplex spaces, and $P:V\rightarrow W$ is unital quotient mapping. Let $K$
be the state space of $V$, $H$ be the image under $P^{\dagger }$ of the
state space of $W$, and $N$ be the kernel of $P$. The following statements
are equivalent:

\begin{enumerate}
\item $N$ is an $M$-ideal of $X$

\item whenever $\varepsilon >0$, $E_{0}\subset E_{1}$ are finite-dimensional
function systems, $g:E_{1}\rightarrow W$ is a unital positive linear map,
and $f:E_{0}\rightarrow V$ is a unital linear isometry such that $P\circ
f=g|_{E_{0}}$, then there exists a unital positive linear map $\widehat{g}%
:E_{1}\rightarrow V$ such that $P\circ \widehat{g}=g$ and $\left\Vert 
\widehat{g}|_{E_{0}}-f\right\Vert \leq \varepsilon $;

\item whenever $\varepsilon >0$, $A$ is a separable function system with the
metric approximation property, $E$ is a finite-dimensional subspace of $A$, $%
f:E\rightarrow V$ and $g:A\rightarrow W$ are unital positive linear maps
with $\left\Vert P\circ f-g|_{E}\right\Vert <\varepsilon $, then there
exists a unital positive linear map $\widehat{g}:A\rightarrow V$ such that $%
P\circ \widehat{g}=g$ and $\left\Vert \widehat{g}|_{E}-f\right\Vert
<3\varepsilon $;

\item if $\varepsilon \geq 0$ and $E\subset V$ is a subsystem such that $%
\left\Vert x\right\Vert \leq \left( 1+\varepsilon \right) \left\Vert
P(x)\right\Vert $ for every $x\in E$, then there exists a linear contraction 
$\eta :W\rightarrow V$ such that $P\circ \eta $ is the identity map of $W$
and $\left\Vert \eta \circ P|_{E}-id_{E}\right\Vert \leq 2\varepsilon $.

\item for any $\varepsilon >0$, $u\in N$ such that $\left\Vert u\right\Vert
=1$, there exists an $v$ of $V$ such that $0\leq v\leq 1$, $\left\Vert
P\left( v\right) \right\Vert \leq \varepsilon $, and $v\geq u-\varepsilon $;

\item $H$ is closed face of $K$.
\end{enumerate}
\end{proposition}

\begin{proof}
In the proof we identify $V$ with $A(K)$ and $W$ with $A(H)$. Under these
identifications $P$ is just the restriction mapping $A\left( K\right)
\rightarrow A(H)$, $f\mapsto f|_{H}$. Observe that in a Choquet simplex
every closed face is a split face \cite[Theorem II.6.22]{alfsen_compact_1971}%
. The equivalence of (6) and (1) thus follows from \cite[Corollary 5.9,
Proposition 5.10]{alfsen_structure_1972-2}. The implication (6)$\Rightarrow $%
(3) can be proved as \cite[Theorem 2.6]{choi_lifting_1977} using \cite[%
Proposition 2.2]{choi_completely_1976} instead of \cite[Lemma 2.1]%
{choi_lifting_1977}. The implication (2)$\Rightarrow $(5) can be proved as
(2)$\Rightarrow $(1). We prove the other nontrivial implications below.

(5)$\Rightarrow $(6) Suppose that $q_{0},q_{1}\in K$ and $p\in H$ are such
that $\left( q_{0}+q_{1}\right) /2=p$. We want to prove that $q_{0}\in H$.
Suppose that $u$ is an element of $N$ of norm $1$ and $\varepsilon >0$. By
assumption there exists an element $v$ of $V$ such that $0\leq v\leq 1$, $%
\left\Vert P\left( v\right) \right\Vert \leq \varepsilon /2$, and $v\geq
u-\varepsilon /2$. Then we have%
\begin{equation*}
q_{0}\left( v\right) \leq \varepsilon -q_{1}\left( v\right) \leq \varepsilon
/2
\end{equation*}%
and%
\begin{equation*}
q_{0}\left( u\right) =q_{0}\left( v\right) +q_{0}\left( u-v\right) \leq
\varepsilon \text{.}
\end{equation*}

(6)$\Rightarrow $(5) Consider the function%
\begin{equation*}
\varphi :k\mapsto \left\{ t\in \mathbb{R}:\max \left\{ k\left( u\right)
,0\right\} \leq t\leq 1\right\}
\end{equation*}%
and observe that it satisfies the hypothesis of \cite[Corollary 3.4]%
{lazar_spaces_1968}, and $0\in \varphi \left( k\right) $ for every $k\in H$.
It follows that there exists $v\in N$ such that $k\left( v\right) \in
\varphi \left( k\right) $ for every $k\in K$.

(6)$\Rightarrow $(4): Suppose that $E\subset V$ is a subsystem such that $%
\left\Vert P(x)\right\Vert \geq \left( 1-\varepsilon \right) \left\Vert
x\right\Vert $ for every $x\in E$. Let $k$ be an element of $K$. We observe
that there exists $h\in H$ such that $\left\Vert \left( k-h\right)
|_{E}\right\Vert \leq 2\varepsilon $. Define the map $h_{0}:P\left[ E\right]
\rightarrow \mathbb{R}$ by $h_{0}\left( P\left( e\right) \right) =k\left(
e\right) $. Observe that by assumption $h_{0}$ is a well-defined unital
linear functional such that $\left\Vert h_{0}\right\Vert \leq 1+\varepsilon $%
. Therefore by Lemma \ref{Lemma:perturb-function-system} there exists a
state $h_{1}$ of $W$ such that $\left\Vert h_{0}-h_{1}\right\Vert \leq
2\varepsilon $. One can then define $h:=h_{1}\circ P\in H$ and observe that $%
\left\Vert \left( h-k\right) |_{E}\right\Vert \leq 2\varepsilon $. Consider
the function defined by%
\begin{equation*}
\varphi :k\mapsto \left\{ h\in H:\left\Vert \left( k-h\right)
|_{E}\right\Vert \leq 2\varepsilon \right\} \text{.}
\end{equation*}%
Observe that $\varphi $ satisfies the assumptions of \cite[Corollary 3.4]%
{lazar_spaces_1968}. Hence there exists a continuous affine function $%
Q:K\rightarrow H$ such that $Q|_{H}$ is the identity map of $H$ and $%
\left\Vert \left( Q\left( k\right) -k\right) |_{E}\right\Vert \leq
\varepsilon $ for every $k\in K$. One can thus define $\eta :A(H)\rightarrow
A(K)$ by $\eta \mapsto \eta \circ Q$.

(2)$\Rightarrow $(1) Fix $y^{(1)},y^{(2)},y^{(3)}\in \mathrm{Ball}\left(
N\right) $, $x\in \mathrm{Ball}(X)$, and $\varepsilon >0$. By the
equivalence (i)$\Leftrightarrow $(iv) in \cite[Theorem 2.2]%
{harmand_M-ideals_1993} it is enough to prove that there exists $y\in 
\mathrm{Ball}\left( N\right) $ such that $\left\Vert x+y^{(\ell
)}-y\right\Vert \leq 1+\varepsilon $ for every $\ell \in \left\{
1,2,3\right\} $. Let $E$ be the linear span of $\left\{
y^{(1)},y^{(2)},y^{(3)},x,1\right\} $ inside $Z$. Consider the function
system obtained from $F\oplus \mathbb{R}$ and the collection of linear
functions $\left( z,\alpha \right) \mapsto s(z)+t\alpha $ where $t\in \left[
-1,1\right] $ and $s$ is a state of $F$ such that $\left\vert s\left(
x+y^{(\ell )}\right) -t\right\vert \leq 1$ for every $\ell \in \left\{
1,2,3\right\} $. Define also the linear map $g:F\rightarrow X$ by $\left(
z,\alpha \right) \mapsto P(z)$. Observe that the inclusion $E\subset F$ is a
unital linear isometry while $g$ is a unital positive linear map such that $%
g|_{E}=P$. By assumption there exists a unital positive linear map $\widehat{%
g}:F\rightarrow Z$ such that $P\circ \widehat{g}=g$ and $\left\Vert \widehat{%
g}|_{E}-\iota _{E}\right\Vert \leq \varepsilon $, where $\iota
_{E}:E\rightarrow Z$ is the inclusion map. Setting $y:=\widehat{g}\left(
0,1\right) $ concludes the proof.
\end{proof}

\begin{remark}
As in Proposition \ref{Proposition:characterize-face} and Proposition \ref%
{Proposition:characterize-biface}, the equivalence of (1)--(3) in
Proposition \ref{Proposition:characterize-face} holds for arbitrary
separable function systems $V,W$. If $F$ is a closed face of a metrizable
Choquet simplex $K$, then the function $A\left( K\right) \rightarrow A\left(
F\right) $, $f\mapsto f|_{F}$ is automatically a quotient mapping by \cite%
{lazar_spaces_1968}.
\end{remark}

We call a unital quotient mapping $P:A(K)\rightarrow A(F)$ between simplex
spaces satisfying the equivalent conditions of Proposition \ref%
{Proposition:characterize-face} $\emph{unital}$ \emph{facial quotient}. A
unital facial quotient $P:A(K)\rightarrow A(F)$ is \emph{nontrivial }if it
not an order isomorphism or, equivalently, $P^{\dag }F$ is a proper face of $%
K$.

Suppose that $F$ is a Choquet simplex. It follows from Proposition \ref%
{Proposition:characterize-face} that the universal positive linear map $%
\Omega _{A(\mathbb{P})}^{A\left( F\right) }:A(\mathbb{P})\rightarrow A\left(
F\right) $ from Subsection \ref{Subsection:universal-state} is a unital
facial quotient mapping. In particular when $K$ is the trivial simplex one
obtains an extreme point $\Omega _{A(\mathbb{P})}^{\mathbb{R}}:A(\mathbb{P}%
)\rightarrow \mathbb{R}$ of the state space $\mathbb{P}$ of $A(\mathbb{P})$.
Since $\Omega _{A(\mathbb{P})}^{\mathbb{R}}$ has a dense $G_{\delta }$ orbit
in $\mathbb{P}$, we conclude that $\mathbb{P}$ has dense extreme boundary.
Conversely assuming that $S$ is a metrizable Choquet simplex with dense
extreme boundary, one can prove that $A\left( S\right) $ is unitally
isometrically isomorphic to $A(\mathbb{P})$ arguing as in Proposition \ref%
{Proposition:dense-Lusky}. Thus $\mathbb{P}$ is the unique metrizable
Choquet simplex with dense extreme boundary.

Suppose now that $F$ is a Choquet simplex, and $\Omega _{A(\mathbb{P}%
)}^{A(F)}:A(\mathbb{P})\rightarrow A(F)$ is the generic unital positive
linear map obtained from the general results of Subsection \ref%
{Subsection:universal-state}. The proof of Proposition \ref%
{Proposition:characterize-operator-G} can be adapted in a straightforward
way to show that a unital quotient mapping $P:A(\mathbb{P})\rightarrow A(F)$
is a unital facial quotient if and only if it belongs to the $\mathrm{Aut}%
\left( A(\mathbb{P})\right) $-orbit of $\Omega _{A(\mathbb{P})}^{A(F)}$, if
and only if the image of $F$ under $P^{\dagger }$ is a closed proper face of 
$\mathbb{P}$. It follows that any metrizable Choquet simplex is affinely
homeomorphic to a closed proper face of $\mathbb{P}$ \cite[Theorem 2.5]%
{lindenstrauss_poulsen_1978}, and any affine homeomorphism between proper
closed faces of $\mathbb{P}$ extends to an affine homeomorphism of $\mathbb{P%
}$. Furthermore if $F$ is a closed proper face of $\mathbb{P}$, then $%
\left\{ f\in A(\mathbb{P}):f\text{ is constant on }F\right\} $ is a function
system order isomorphic to $A(\mathbb{P})$.

Suppose that $K,K_{0}$ are Choquet simplices, $\varphi :K\rightarrow K_{0}$
is a surjective continuous affine map, $F$ is a proper closed face of $K$,
and $F^{\prime }$ is the complementary face of $F$; see \cite[Section 6]%
{alfsen_compact_1971}. It follows from Edwards' separation theorem \cite[%
Theorem II.3.10]{alfsen_compact_1971} and \cite[Proposition II.6.5]%
{alfsen_compact_1971} that $F^{\prime }$ is the set of points $s\in K$ such
that for any $\varepsilon >0$ there exists $h\in A(K)$ such that $0\leq
h\leq 1$, $h|_{F}$ is constantly equal to $1$, and $h\left( s\right)
<\varepsilon $. It follows that the image of $F$ under $\varphi $ is a
closed face $F_{0}$ of $K_{0}$, and the image of $F^{\prime }$ under $%
\varphi $ is the complementary face of $F_{0}$.

\begin{proposition}[{\protect\cite[Corollary 2.4]{lindenstrauss_poulsen_1978}%
}]
Suppose that $F_{0},F_{1}$ are proper closed faces of $\mathbb{P}$. Endow
the complementary face $F_{0}^{\prime }$ of $F_{0}$ with the w*-topology
induced by the elements $a$ of $A(\mathbb{P})$ such that $a|_{F_{0}}$ is
constant, and similarly for $F_{1}^{\prime }$. Then $F_{0}^{\prime }$ and $%
F_{1}^{\prime }$ are affinely homeomorphic.
\end{proposition}

\begin{proof}
Suppose that $F$ is a proper closed face of $\mathbb{P}$ and $F^{\prime }$
is the complementary face of $F$. Consider $W=\left\{ f\in A(\mathbb{P}):f%
\text{ is constant on }F\right\} $ and observe that $W$ is a function
system. Let $K$ be the state space of $W$. As observed before, $K$ is
affinely homeomorphic to $\mathbb{P}$. Denote by $\varphi :\mathbb{P}%
\rightarrow K$ the surjective continuous affine map obtained from the
inclusion $A(K)\subset A(\mathbb{P})$ by duality. The image of $F$ is a
single point $s_{F}$ of $K$. Since $F$ is a face, $s_{F}$ is an extreme
point of $K$. The map $\varphi $ is 1:1 on the complementary face $F^{\prime
}$ of $F$ by \cite[Corollary II.6.17]{alfsen_compact_1971}. It follows from
the remarks above that the image of $F^{\prime }$ under $\varphi $ is the
complementary face of $\left\{ s_{F}\right\} $ in $K$. The w*-topology on $%
F^{\prime }$ induced by the elements of $W$ makes the restriction of $%
\varphi $ to $F^{\prime }$ a homeomorphism. The conclusion now follows from
the fact that $K$ is affinely homeomorphic to $\mathbb{P}$ and that $\mathrm{%
Aut}(\mathbb{P})$ acts transitively on the extreme points of $\mathbb{P}$.
\end{proof}

Applying the criterion from Corollary \ref{Corollary:minimal} one can see
that the canonical continuous action of $\mathrm{Aut}(\mathbb{P})$ on $%
\mathbb{P}$ is minimal, recovering a result of Glasner from \cite[Theorem 5.2%
]{glasner_distal_1987}.

\begin{proposition}
The canonical action $\mathrm{Aut}(\mathbb{P})\curvearrowright \mathbb{P}$
is minimal.
\end{proposition}

\begin{proof}
In view of Proposition \ref{Proposition:minimal} it is enough to prove that
for any $\varepsilon >0$ and $d\in \mathbb{N}$ there exists $m\in \mathbb{N}$
such that for any $s\in S\left( \ell _{d}^{\infty }\right) $ and $t\in
S\left( \ell _{m}^{\infty }\right) $ there exists a unital linear isometry $%
\phi :\ell _{d}^{\infty }\rightarrow \ell _{m}^{\infty }$ such that $%
\left\Vert t\circ \phi -s\right\Vert <\varepsilon $. Set $\eta =\frac{%
\varepsilon }{2d}$ and $m\in \mathbb{N}$ be such that $m\geq 1/\eta +d$.
Suppose that $s\in S\left( \ell _{d}^{\infty }\right) $ and $t\in S\left(
\ell _{m}^{\infty }\right) $. Then $s=\left( s_{1},\ldots ,s_{d}\right) $
can be seen as a stochastic vector of length $d$, and $t=\left( t_{1},\ldots
,t_{m}\right) $ can be seen as a stochastic vector of length $m$. Let $%
A\subset \left\{ 1,2,\ldots ,m\right\} $ be the set of $k$ such that $%
t_{k}\geq \eta $. Observe that $\left\vert A\right\vert \leq 1/\eta $. We
can assume that $A=\left\{ 1,2,\ldots ,\ell \right\} $ for some $\ell \leq
1/\eta $. Define the map $\phi :\ell _{d}^{\infty }\rightarrow \ell
_{m}^{\infty }$ by $x=\left( x_{1},\ldots ,x_{d}\right) \mapsto \left(
s(x),\ldots ,s(x),x_{1},\ldots ,x_{d}\right) $. Observe that $\phi $ is
indeed a unital linear isometry. Furthermore we have that, for $x\in \ell
_{d}^{\infty }$ such that $\left\Vert x\right\Vert \leq 1$, 
\begin{eqnarray*}
\left\vert \left( t\circ \phi \right) (x)-s(x)\right\vert &=&\left\vert
s(x)\left( t_{1}+\cdots +t_{m-d}\right) +x_{1}t_{m-d+1}+\cdots
+x_{d}t_{m}-s(x)\right\vert \\
&=&\left\vert x_{1}t_{m-d+1}+\cdots +x_{d}t_{m}-s(x)\left( t_{m-d+1}+\cdots
+t_{m}\right) \right\vert \leq 2d\eta \leq \varepsilon \text{.}
\end{eqnarray*}%
This concludes the proof.
\end{proof}

As for Banach spaces, one can conclude from uniqueness of the Fra\"{\i}ss%
\'{e} limit, the characterization of the Fra\"{\i}ss\'{e} limit, and \cite[%
Proposition 13.6]{ben_yaacov_model_2008} that the following facts, already
proved implicitly in \cite{goldbring_model-theoretic_2015}, hold: the first
order theory of $A(\mathbb{P})$ has a unique separable model and it admits
elimination of quantifiers; the group $\mathrm{Aut}(\mathbb{P})$ of affine
homeomorphisms of $\mathbb{P}$ is Roelcke precompact; $A(\mathbb{P})$ is the
unique existentially closed separable function system; the theory of $A(%
\mathbb{P})$ is the model completion of the theory of function systems; a
compact convex set $K$ is a Choquet simplex if and only if $A(K)$ is a
positively existentially closed function system.

\subsection{$p$-multinormed spaces\label{Subsection:p-multinormed}}

Fix $p\in \left[ 1,+\infty \right] $. Consider the space $B\left( \ell
^{p}\right) $ of bounded linear operators on $\ell ^{p}$, and let $\mathcal{K%
}^{p}\subset B\left( \ell ^{p}\right) $ be the space of compact operators.
Observe that if $X$ is a complex vector space, then the algebraic tensor
product $\ell ^{p}\otimes X$ has a natural left $\mathcal{K}^{p}$-module
structure. A $p$-\emph{multinormed space} is a complex vector space such
that $\ell ^{p}\otimes X$ is endowed with a norm such that $\left\Vert
\alpha x\right\Vert \leq \left\Vert \alpha \right\Vert \left\Vert
x\right\Vert $ for $\alpha \in \mathcal{K}^{p}$ and $x\in \ell ^{p}\otimes X$%
, where $\left\Vert \alpha \right\Vert $ denotes the norm of $\alpha $
regarded as an element of $B\left( \ell ^{p}\right) $. A linear map $\phi
:X\rightarrow Y$ between $p$-multinormed spaces is multicontractive if $id_{%
\mathcal{K}^{p}}\otimes \phi $ is contractive, and multi-isometric if $id_{%
\mathcal{K}^{p}}\otimes \phi $ is isometric. If $X,Y$ are $p$-multinormed
spaces, then the $\infty $-sum $X\oplus ^{\infty }Y$ is defined by
identifying isometrically $\ell ^{p}\otimes \left( X\oplus Y\right) $ with
the $\infty $-sum of $\ell ^{p}\otimes X$ and $\ell ^{p}\otimes Y$. One can
similarly define the $\infty $-sum of an arbitrary collection of $p$%
-multinormed spaces.

Multinormed spaces have been introduced and studied in \cite%
{dales_multi-norms_2012,dales_multi-normed_2012,dales_equivalence_2014}. The
generalization to $p$-multinormed spaces for arbitrary $p\in \lbrack
1,+\infty ]$ has been studied in the recent work of Dales, Laustsen,
Oikhberg, and Troitsky \cite{dales_multinormed_2015,dales_multi-norms_2016}.
Multinormed spaces correspond to the case $p=+\infty $. If $E$ is a Banach
space, we denote by $\max^{p}(E)$ the largest compatible $p$-multinorm
structure on $E$. In the following we will always assume $p\in \left(
1,+\infty \right) $. It has been recently shown by Oikhberg \cite%
{oikhberg_personal} that, if $q$ is the conjugate exponent of $p$ and $%
\left( X,\mu \right) $ is a measure space, then $\max^{p}(L^{q}\left( X,\mu
\right) )$ is an injective $p$-multinormed space. Furthermore any $p$%
-multinormed space embeds multi-isometrically into the $\infty $-sum of $p$%
-multinormed spaces of the form $\max^{p}\left( \ell _{n}^{q}\right) $.

Let $\mathcal{L}$ be the language containing binary function symbols $%
f_{\alpha ,\beta }$ for any $\alpha ,\beta \in \mathcal{K}_{0}^{p}(\mathbb{Q}%
(i))$ such that $\left\Vert \alpha \right\Vert +\left\Vert \beta \right\Vert
\leq 1$.\ Here $\mathbb{Q}(i)$ is the field of Gauss rationals, while $%
\mathcal{K}_{0}^{p}(\mathbb{Q}(i))$ denotes the space of operators whose
representative matrices with respect to the canonical basis of $\ell ^{p}$
have coefficients that belong to $\mathbb{Q}(i)$, and are all zero but
finitely many. A $p$-multinormed space $X$ can be regarded as an $\mathcal{L}
$-structure supported by the unit ball of $\ell ^{p}\otimes X$, where $%
f_{\alpha ,\beta }$ is interpreted as the function $\left( x,y\right)
\mapsto \alpha x+\beta y$. A morphism in this context is a linear
multicontraction, and an embedding is a linear multi-isometry.

We can then consider the category $\mathcal{A}$ of $p$-multinormed spaces
and multicontractive maps, and the collection $\mathcal{I}\subset \mathcal{A}
$ of $p$-multinormed spaces that are a finite $\infty $-sum of copies of $%
\max^{p}\left( \ell _{n}^{q}\right) $. If $f:X\rightarrow Y$ is a linear
multicontraction, then $I(f)\leq \delta <1$ as in Definition \ref%
{Definition:I} if and only if $\left\Vert (id_{\mathcal{K}^{p}}\otimes
f)(x)\right\Vert \geq \left\Vert x\right\Vert -\delta $ for any $x\in 
\mathcal{K}^{p}\otimes X$ of norm at most $2$, which happens if and only if $%
T$ is injective and $\left\Vert id_{\mathcal{K}^{p}}\otimes
T^{-1}\right\Vert \leq 1+\delta $. We stipulate that a finite tuple $\bar{a}$
in a $p$-multinormed space is a basic tuple if it is linearly independent.
An argument similar to the small perturbations lemma \cite[Lemma 2.13.2]%
{pisier_introduction_2003} shows that the conditions from Subsection \ref%
{Subsection:basic} are satisfied.

We can conclude that finite-dimensional $p$-multinormed spaces form a Fra%
\"{\i}ss\'{e} class. We call the corresponding limit $\mathbb{GM}^{p}$ the 
\emph{Gurarij }$p$\emph{-multinormed space}. It has been proved by Oikhberg 
\cite{oikhberg_personal} that for every $\varepsilon >0$ and $n\in \mathbb{N}
$ there exists $k\in \mathbb{N}$ (depending only on $n$ and $\varepsilon $)
with the following property: for any $n$-dimensional $p$-multinormed spaces $%
E,F$ and linear map $\phi :E\rightarrow F$, one has that $||id_{\mathcal{K}%
^{p}}\otimes \phi ||{}\leq \left( 1+\varepsilon \right)
||id_{M_{k}^{p}}\otimes \phi ||$. Here we regard $M_{k}^{p}\subset \mathcal{K%
}^{p}$ as the subspace of $x\in \mathcal{K}^{p}$ such that $p_{k}x=xp_{k}=x$%
, where $p_{k}$ is the projection onto the span of the first $k$ vectors of
the canonical basis of $\ell ^{p}$. Therefore the characterizing property of 
$\mathbb{\mathbb{\mathbb{\mathbb{G}}}M}^{p}$ given by Proposition \ref%
{Proposition:characterize-limit} is elementary, that is, it can be expressed
by formulas in the logic for metric structures. Hence $\mathbb{\mathbb{%
\mathbb{\mathbb{G}}}M}^{p}$ is the unique separable model of its first order
theory. As for the Gurarij space, one can also observe that by \cite[%
Proposition 13.6]{ben_yaacov_model_2008} the theory of $\mathbb{\mathbb{%
\mathbb{\mathbb{G}}}M}^{p}$ admits elimination of quantifiers. It follows
from this and \cite[Theorem 2.4]{ben_yaacov_weakly_2013} that the Polish
group of surjective linear multi-isometries of $\mathbb{GM}^{p}$ endowed
with the topology of pointwise convergence is Roelcke precompact; see \cite[%
Definition 2.2]{ben_yaacov_weakly_2013}. Since every $p$-multinormed space
embeds into a model of the theory of $\mathbb{GM}^{p}$, one can conclude
that the theory of $\mathbb{GM}^{p}$ is the model completion of the theory
of $p$-multinormed spaces. Finally one can observe that for a separable $p$%
-multinormed space $X$, the following assertions are equivalent:

\begin{itemize}
\item $X$ is approximately injective as in Definition \ref%
{Definition:approx-inj};

\item $X$ is multi-isometric to the range of $p$-multicontractive linear
projections on $\mathbb{GM}^{p}$;

\item the identity map of $X$ is the pointwise limit of multicontractive
maps that factor through finite $\infty $-sums of copes of $\max^{p}(\ell
_{n}^{p})$ for $n\in \mathbb{N}$;

\item $X$ is positively existentially closed in the class of $p$-multinormed
spaces.
\end{itemize}

\subsection{Operator sequence spaces\label{Subsection:sequence}}

An \emph{operator sequence space }is a $2$-multinormed space which is
moreover $2$-convex, in the sense that it satisfies 
\begin{equation*}
\left\Vert \sum_{i=1}^{n}e_{i}\otimes x_{i}+\sum_{i=n+1}^{m}e_{i}\otimes
x_{i}\right\Vert ^{2}\leq \left\Vert \sum_{i=1}^{n}e_{i}\otimes
x_{i}\right\Vert ^{2}+\left\Vert \sum_{i=n+1}^{m}e_{i}\otimes
x_{i}\right\Vert ^{2}
\end{equation*}%
for $n,m\in \mathbb{N}$, where $\left( e_{i}\right) $ is the canonical
orthonormal basis of $\ell ^{2}$. It is easy to see that such a definition
is equivalent to \cite[Definition 2.1]{lambert_operator_2004}. Let us denote
by $\mathcal{K}$ the algebra of compact operators on $B(\ell ^{2})$. A
linear map $\phi :X\rightarrow Y$ between operator sequence spaces is \emph{%
sequentially contractive }if $id_{\mathcal{K}}\otimes \phi $ is a
contraction, and a\emph{\ sequential isometry} if $id_{\mathcal{K}}\otimes
\phi $ is an isometry. Operator sequence spaces have been introduced and
studied in \cite{lambert_operatorfolgenraume_2002,lambert_operator_2004}.
They have been used in \cite{lambert_operator_2004} to shed light on the
properties of Fig\'{a}-Talamanca--Hertz algebras. A systematic study of
operator sequence spaces is presented in \cite%
{lambert_operatorfolgenraume_2002}. Every operator sequence space is
canonically endowed with a \emph{minimal operator space structure }\cite[%
Definition 3.1]{lambert_operator_2004}. We can regard operator sequence
spaces as operator spaces endowed with their minimal operator space
structure. It is proved in \cite{lambert_operatorfolgenraume_2002} that the
operator spaces arising in this way are precisely the subspaces of $\infty $%
-sums of column operator Hilbert spaces \cite[Subsection 3.4]%
{effros_operator_2000}.

We can therefore regard the category $\mathcal{A}$ of operator sequence
spaces and sequential contractions as a full subcategory of the category of
operator spaces and completely contractive maps. The collection $\mathcal{I}$
of finite $\infty $-sum of finite-dimensional column operator Hilbert spaces
is a collection of injective objects of $\mathcal{A}$ that satisfies
Conditions (1) and (2) of Subsection \ref{Subsection:injective}. Again the
notion of basic tuple is provided by independent tuples. We can therefore
conclude that finite-dimensional operator sequence spaces form a Fra\"{\i}ss%
\'{e} class. We call the corresponding limit $\mathbb{CG}$ the \emph{column
Gurarij space}. The same argument as for $p$-multinormed spaces shows that
the first order theory of $\mathbb{CG}$ has a unique separable model and it
admits elimination of quantifiers. As a consequence the Polish group of
surjective complete isometries of $\mathbb{CG}$ endowed with the topology of
pointwise convergence is Roelcke precompact. As before, $\mathbb{CG}$ is the
unique existentially closed operator sequence space, and the theory of $%
\mathbb{CG}$ is the model completion of the theory of operator sequence
spaces. For a separable operator sequence space $X$, the following
assertions are equivalent:

\begin{itemize}
\item $X$ is approximately injective in the sense of Definition \ref%
{Definition:approx-inj};

\item $X$ is sequentially isometric to the range of a sequentially
contractive projection on $\mathbb{CG}$;

\item the identity map of $X$ is the pointwise limit of multicontractive
maps that factor through finite $\infty $-sums of finite-dimensional column
operator Hilbert spaces;

\item $X$ is a nuclear operator space with its canonical minimal operator
space structure;

\item $X$ is positively existentially closed in the class of operator
sequence spaces.
\end{itemize}

\subsection{$M_{q}$-spaces\label{Subsection:Mq}}

Fix $q\in \mathbb{N}$ and let $M_{q}(\mathbb{C)}$ be the space of complex $%
q\times q$ matrices. If $X$ is a complex vector space, then the algebraic
tensor product $M_{q}(\mathbb{C)}\otimes X$ can be canonically identified
with the space $M_{q}(X)$ of $q\times q$ matrices with entries in $X$. There
is a natural bimodule action of $M_{q}(\mathbb{C)}$ on $M_{q}(\mathbb{C)}%
\otimes X$. An $M_{q}$-spaces is a complex vector space such that $M_{q}(%
\mathbb{C)}\otimes X$ is endowed with a norm satisfying%
\begin{equation*}
\left\Vert \sum_{i=1}^{n}\alpha _{i}^{\ast }x_{i}\beta _{i}\right\Vert \leq
\left\Vert \sum_{i=1}^{n}\alpha _{i}^{\ast }\alpha _{i}\right\Vert
\max_{1\leq i\leq n}\left\Vert x_{i}\right\Vert \left\Vert
\sum_{i=1}^{n}\beta _{i}^{\ast }\beta _{i}\right\Vert
\end{equation*}%
for $\alpha _{i}\in M_{q}(\mathbb{C)}$ and $x_{i}\in M_{q}(\mathbb{C)}%
\otimes X$, where the norms of complex $q\times q$ matrices are the operator
norms. Such spaces have been introduced and studied in \cite%
{lehner_mn-espaces_1997}, and subsequently used in \cite%
{oikhberg_operator_2004,oikhberg_non-commutative_2006,lupini_uniqueness_2016}%
. For $q=1$ one obtains the class of complex Banach spaces. The language for 
$M_{q}$-spaces contains function symbols for the functions $\left(
x_{1},\ldots ,x_{n}\right) \mapsto \alpha _{1}x_{1}\beta _{1}+\cdots +\alpha
_{n}x_{n}\beta _{n}$ for every $n\in \mathbb{N}$ and $\alpha _{1},\ldots
,\alpha _{n},\beta _{1},\ldots ,\beta _{n}\in M_{q}(\mathbb{Q}(i))$ such
that $\left\Vert \sum_{i=1}^{n}\alpha _{i}^{\ast }\alpha _{i}\right\Vert
\leq 1$ and $\left\Vert \sum_{i=1}^{n}\beta _{i}^{\ast }\beta
_{i}\right\Vert \leq 1$.

A linear map $\phi :X\rightarrow Y$ between $M_{q}$-spaces is $q$%
-contractive if $id_{M_{q}}\otimes \phi $ is contractive, and $q$-isometric
if $id_{M_{q}}\otimes \phi $ is isometric. The category of $M_{q}$-spaces
also has products, given by $\infty $-sums \cite[Subsection I.2.2]%
{lehner_mn-espaces_1997}. Any $M_{q}$-space is $q$-isometric to a subspace
of $C\left( K,M_{q}(\mathbb{C)}\right) =M_{q}\left( C\left( K\right) \right) 
$ for some compact Hausdorff space $K$ \cite[Th\'{e}or\`{e}me I.1.9]%
{lehner_mn-espaces_1997}.\ Proposition I.1.16 of \cite%
{lehner_mn-espaces_1997} shows that $M_{q}$ is an injective object in the
category of $M_{q}$-spaces and $q$-contractions. Using these facts, one can
easily show that the category $\mathcal{A}$ of $M_{q}$-spaces and $q$%
-contractive maps, together with the collection $\mathcal{I}\subset \mathcal{%
A}$ of finite $\infty $-sums of copies of $M_{q}(\mathbb{C)}$, satisfy the
assumptions of Section \ref{Section:injective} with $\varpi (\delta )=\delta 
$. Hence all the results from Sections \ref{Section:injective}, \ref%
{Section:retracts}, \ref{Section:operators}, and \ref{Section:states} apply
in this setting. The corresponding limit $\mathbb{G}_{q}$ has analogous
property as $\mathbb{G}$; see \cite[\S 3]{goldbring_model-theoretic_2015}.

\subsection{$M_{q}$-system\label{Subsection:Mqu}}

An $M_{q}$-\emph{system} is an $M_{q}$-space $X$ with a distinguished
element $1$ (the \emph{unit}) such that there exists a compact Hausdorff
space $K$ and a completely isometric linear map $\phi :X\rightarrow C\left(
K,M_{q}\right) $ which is \emph{unital }in the sense that maps $1$ to the
function constantly equal to the identity matrix of $M_{q}$. Any $M_{q}$%
-system is endowed with an involution $x\mapsto x^{\ast }$ coming from the
inclusion into $C\left( K,M_{q}\right) $. Any unital linear contraction is
automatically \emph{self-adjoint}, that is commutes with the involution.
These spaces have been introduced and studied in \cite{xhabli_super_2012}
under the name of $q$-minimal operator systems. For $q=1$ one obtains the
notion of \emph{complex function system\ }\cite{paulsen_vector_2009}, which
is the complex analog of the notion of (real) function systems as in
Subsection \ref{Subsection:function-system}.

As in the case of function systems, $M_{q}$-systems are in functorial 1:1
correspondence with natural geometric objects that we call $M_{q}$-convex
sets. Suppose that $V$ is a locally convex topological vector space $V$, and 
$K_{n}\subset M_{n}\left( V\right) $ are compact convex sets for $%
n=1,2,\ldots ,q$. An $M_{q}$-convex combination is an expression of the form 
$\alpha _{1}^{\ast }v_{1}\alpha _{1}+\cdots +\alpha _{\ell }^{\ast }v_{\ell
}\alpha _{\ell }$ where $\alpha _{i}\in M_{n_{i},q}$ and $v_{i}\in K_{n_{i}}$
for $1\leq n_{i}\leq q$ and $1\leq i\leq \ell $. We say that $\left(
K_{1},\ldots ,K_{q}\right) $ is an $M_{q}$-convex set if it is closed under $%
M_{q}$-convex combinations. The notion of $M_{q}$-affine function between $%
M_{q}$-convex sets is defined in the obvious way by considering $M_{q}$%
-convex combinations rather than usual matrix convex combinations. To any $%
M_{q}$-convex set one can associate the $M_{q}$-system $A\left( K_{1},\ldots
,K_{q}\right) $ of real-valued $M_{q}$-convex affine functions, endowed with
its canonical $M_{q}$-system structure. Conversely, any $M_{q}$-system $X$
arises in this way from the $M_{q}$-convex set $\left( S_{1}(X),\ldots
,S_{q}(X)\right) $, where $S_{n}(X)$ is the space of $q$-contractive linear
maps from $X$ to $M_{n}$. Indeed such a correspondence is a particular case
of the correspondence between operator systems and matrix convex sets
established in \cite{webster_krein-milman_1999}.

We regard $M_{q}$-systems as structures in the language of $M_{q}$-spaces
with the addition of a constant symbol for the unit and a unary function
symbol for the involution. Again one can show that the category $\mathcal{A}$
of $M_{q}$-systems and unital $q$-contractive maps, and the collection $%
\mathcal{I}\subset \mathcal{A}$ of finite $\infty $-sums of copies of $M_{q}$
satisfy the assumptions of Section \ref{Section:injective}. To see this one
can use the small perturbation lemma \cite[Lemma 2.13.2]%
{pisier_introduction_2003} together with the fact that approximately unital
approximately $q$-isometric maps are close to unital $q$-contractive maps.%
\textrm{\ }This follows from the more general Lemma \ref{Lemma:perturb-ucp}.
In this context a basic tuple is a tuple $\bar{a}$ of linearly independent
elements such that the first element of the tuple $\bar{a}$ is the unit. One
then can infer that the conclusions of Section \ref{Section:injective}, \ref%
{Section:retracts}, \ref{Section:operators}, and \ref{Section:states} hold
in the setting of $M_{q}$-systems.

The limit of the class of finite-dimensional $M_{q}$-systems is an $M_{q}$%
-system $A(\mathbb{P}_{1}^{(q)},\ldots ,\mathbb{P}_{q}^{(q)})$. It follows
from the general results of this paper that---as shown in \cite[Section 3]%
{goldbring_model-theoretic_2015}---the first order theory of $A(\mathbb{P}%
_{1}^{(q)},\ldots ,\mathbb{P}_{q}^{(q)})$ has a unique separable model and
admits quantifier elimination. Applying Corollary \ref{Corollary:minimal}
one can conclude that the action of $\mathrm{Aut}(A(\mathbb{P}%
_{1}^{(q)},\ldots ,\mathbb{P}_{q}^{(q)}))$ on $\mathbb{P}_{1}^{(q)}$ is
minimal using the following lemma, which can be proved similarly as Lemma %
\ref{Lemma:minimal}.

\begin{lemma}
Suppose that $q,d\in \mathbb{N}$ and $\varepsilon >0$.\ There exists $n\in 
\mathbb{N}$ such that for any states $s$ on $\ell _{d}^{\infty }\left(
M_{q}\right) $ and $t$ on $\ell _{n}^{\infty }\left( M_{q}\right) $ there
exists an embedding $\phi :\ell _{d}^{\infty }\left( M_{q}\right)
\rightarrow \ell _{n}^{\infty }\left( M_{q}\right) $ such that $\left\Vert
t\circ \phi -s\right\Vert <\varepsilon $.
\end{lemma}

When $q=1$ one recovers the Poulsen simplex $\mathbb{P}=\mathbb{P}_{1}^{(1)}$%
. The sequence of spaces $(\mathbb{P}_{1}^{(q)},\ldots ,\mathbb{P}%
_{q}^{(q)}) $ for $q\in \mathbb{N}$ can be seen as a sequence interpolating
between the Poulsen simplex $\mathbb{P}$ and the noncommutative Poulsen
simplex $\mathbb{\mathbb{NP}}$; see \ref{Subsection:osystems}.

\section{More general Fra\"{\i}ss\'{e} classes\label{Section:general}}

\subsection{Stratified Fra\"{\i}ss\'{e} classes\label{Subsection:more
general}}

In this section we discuss how the framework of Section \ref%
{Section:injective} can be generalized to apply to other classes of
structures from functional analysis, such as the classes of exact operator
spaces and exact operator systems. We still assume that $\mathcal{L}$ is a
countable language in the logic for metric structures, and keep the same
notation and terminology as in Subsection \ref{Subsection:morphism}.

Suppose that $\mathcal{A}$ is a category of $\mathcal{L}$-structures with
morphisms defined as in Definition \ref{Definition:morphism}. Let also $%
\mathcal{A}_{q}\subset \mathcal{A}$ for $q\in \mathbb{N}\cup \left\{ \infty
\right\} $ be a full subcategory such that $\mathcal{A}_{q}\subset \mathcal{A%
}_{q+1}$, and $\mathcal{I}_{q}\subset \mathcal{A}_{q}$ be a countable
collection of separable injective structures closed under finite products
such that $\mathcal{I}_{q}\subset \mathcal{I}_{q+1}$ and $\mathcal{I}%
_{\infty }=\bigcup_{q}\mathcal{I}_{q}$. Let $\mathcal{I}\subset \mathcal{I}%
_{\infty }$ be a cofinal collection, that is such that any structure in $%
\mathcal{I}_{\infty }$ admits an embedding into a structure of $\mathcal{I}$%
. We will assume that

\begin{itemize}
\item $\mathcal{A}$ satisfies Conditions (1)--(4) of Subsection \ref%
{Subsection:basic} where the assumption (3) that $\mathcal{A}$ has arbitrary
products is replaced by the hypothesis that $\mathcal{A}$ has finite
products;

\item $\mathcal{A}_{q}$ satisfies all the conditions of Subsection \ref%
{Subsection:basic};

\item $\mathcal{A}_{q}$ has enough injectives from $\mathcal{I}_{q}$ with
modulus $\varpi $ as in Definition \ref{Definition:approximate-inverses};

\item $\mathcal{A}$ has enough injectives from $\mathcal{I}_{\infty }$ with
modulus $\varpi $ as in Definition \ref{Definition:approximate-inverses}.
\end{itemize}

The arguments of Section \ref{Section:injective} apply in this more general
situation to show the following.

\begin{theorem}
\label{Theorem:more-general-limit}The class of finitely generated structures
in $\mathcal{A}$ is a Fra\"{\i}ss\'{e} class. The corresponding limit $M$
can be realized as limit of a direct sequence of structures in $\mathcal{I}$
with embeddings as connective maps. Furthermore the limit admits the same
characterization as in Proposition \ref{Proposition:characterize-limit}.
\end{theorem}

Also the arguments of Section \ref{Section:operators} and Section \ref%
{Section:states} go through in this more general setting. This yields a
generic morphism $\Omega _{M}^{R}:M\rightarrow R$ for any approximately
injective separable structure $R$ in $\mathcal{A}$, and a generic operator $%
\Omega _{M}:M\rightarrow M$, with the same characterization and properties
as in Section \ref{Section:operators}.

\subsection{Approximately injective objects and retracts of the limit\label%
{Subsection:general-retracts}}

The class of approximately injective structures in $\mathcal{A}$ can be
defined similarly as in Definition \ref{Definition:approx-inj}. Again, since
the elements of $\mathcal{I}$ are injective, and the Fra\"{\i}ss\'{e} limit $%
M$ of the class of finitely generated structures in $\mathcal{A}$ is the
limit of an inductive sequence of elements of $\mathcal{I}$ with embeddings
as connective maps, it follows that $M$ is approximately injective. As a
consequence the retracts of $M$ are approximately injective as well. The
following theorem shows that, conversely, any separable approximately
injective structure in $\mathcal{A}$ is isomorphic to a retract of $M$.

\begin{theorem}
\label{Theorem:general-retracts}Let $M$ denote the Fra\"{\i}ss\'{e} limit of
the class of finitely generated structures in $\mathcal{A}$. A separable
structure $X$ in $\mathcal{A}$ is approximately injective if and only if
there exist an embedding $\eta :X\rightarrow M$ and an idempotent morphism $%
\pi :M\rightarrow M$ such that the range of $\eta $ equals the range of $\pi 
$.
\end{theorem}

\begin{proof}
Suppose that $X$ is a separable approximately injective structure in $%
\mathcal{A}$. Our aim is to construct a separable structure $Z$ in $\mathcal{%
A}$, an embedding $\eta :X\rightarrow Z$, and a morphism $\pi :Z\rightarrow
X $ such that $\pi \circ \eta $ is the identity map of $X$ and $Z$ is the Fra%
\"{\i}ss\'{e} limit of the class of finitely generated elements of $\mathcal{%
A}$.

Recall that if $f,g:A\rightarrow B$ are morphisms between structures in $A$, 
$\bar{a}$ is a tuple in $A$, and $\varepsilon >0$, then we write $f\approx _{%
\bar{a},\varepsilon }g$ to indicate that $d(f(\bar{a}),g(\bar{b}%
))<\varepsilon $. Fix an enumeration $\left\{ A_{q,m}:m\in \mathbb{N}%
\right\} $ of the elements of $\mathcal{I}_{q}$. Fix for every $m\in \mathbb{%
N}$ a countable fundamental subset $D_{q,m}$ of $A_{q,m}$ as in Definition %
\ref{Definition:fundamental} and an enumeration $\bar{a}_{q,m,i}$ of the
basic tuples in $D_{q,m}$. Let $E_{q,m,i}=\left\langle \bar{a}%
_{q,m,i}\right\rangle $. One can build by recursion on $n$

\begin{itemize}
\item an increasing sequence $\left( q_{n}\right) $ in $\mathbb{N}$ such
that $q_{n}\geq n$,

\item an increasing sequence $\left( X_{n}\right) $ of substructures of $X$
with dense union,

\item basic generating tuples $\bar{b}_{n}$ of $X_{n}$,

\item a sequence $\left( \varepsilon _{n}\right) $ of strictly positive real
numbers such that $\varpi \left( \varepsilon _{n+1}\right) \leq \varepsilon
_{n}\leq 2^{-2n}$,

\item a sequence $(\widehat{X}_{n})$ of finitely generated structures in $%
\mathcal{A}$ such that $\widehat{X}_{n}\in \mathcal{A}_{q_{n}}$,

\item morphisms $\varphi _{n}:X_{n}\rightarrow \widehat{X}_{n}$, $\psi _{n}:%
\widehat{X}_{n}\rightarrow X_{n}$, and $\widehat{i}_{n}:\widehat{X}%
_{n}\rightarrow \widehat{X}_{n+1}$ such that, if $i_{n}:X_{n}\rightarrow
X_{n+1}$ denotes the inclusion map, $I\left( \varphi _{n}\right)
<\varepsilon _{n}$, $I\left( \psi _{n}\right) <\varepsilon _{n}$, $d\left(
\psi _{n}\circ \varphi _{n},id_{X_{n}}\right) <\varepsilon _{n}$, $\widehat{i%
}_{n+1}=\varphi _{n+1}\circ i_{n}\circ \psi _{n}$, and hence $d(\widehat{i}%
_{n+1}\circ \varphi _{n},\varphi _{n+1}\circ i_{n})<\varepsilon _{n}$,

\item a direct sequence $(Z_{n})$ of finitely generated structures in $%
\mathcal{A}$ with embeddings $j_{n}:Z_{n}\rightarrow Z_{n+1}$ as connective
maps,

\item embeddings $\eta _{n}:\widehat{X}_{n}\rightarrow Z_{n}$ such that $%
\eta _{n+1}\circ \widehat{i}_{n}\circ \varphi _{n}\approx _{\bar{b}%
_{n},\varepsilon _{n}}j_{n}\circ \eta _{n}\circ \varphi _{n}$,

\item morphisms $\pi _{n}:Z_{n}\rightarrow \widehat{X}_{n}$ such that $\pi
_{n+1}\circ j_{n}=\widehat{i}_{n}\circ \pi _{n}$ and $\pi _{n+1}\circ \eta
_{n+1}=id_{\widehat{X}_{n+1}}$

\item a finite $\varepsilon _{n}$-dense set $\mathcal{G}_{n}$ of morphisms $%
g:E_{q,m,i}\rightarrow Z_{n}$ for $m,i\leq n$ and $q\leq q_{n}$ such that $%
j_{n}\circ g\in \mathcal{G}_{n}$ for every $g\in \mathcal{G}_{n-1}$,

\item a finite $\varepsilon _{n}$-dense set $\mathcal{F}_{n}$ of morphisms $%
f:E_{q,m,i}\rightarrow X_{n}$ for $m,i\leq n$ and $q\leq q_{n}$ such that $%
i_{n}\circ f\in \mathcal{F}_{n}$ for every $f\in \mathcal{F}_{n-1}$ and $%
\psi _{n}\circ \pi _{n}\circ g\in \mathcal{F}_{n}$ for every $g\in \mathcal{G%
}_{n}$,
\end{itemize}

such that

\begin{enumerate}
\item for any $i,m\leq n$, $q\leq q_{n}$, and morphism $f:E_{q,m,i}%
\rightarrow X_{n}$ in $\mathcal{F}_{n}$ there exists a morphisms $\widehat{f}%
:A_{q,m}\rightarrow X_{n+1}$ such that $\widehat{f}\approx _{\bar{a}%
_{q,m,i},\varepsilon _{n}}f$;

\item for any $m,i\leq n$, $q\leq q_{n}$, and morphism $g:E_{q,m,i}%
\rightarrow Z_{n}$ in $\mathcal{G}_{n}$, there exists an embedding $\widehat{%
g}:A_{q,m}\rightarrow Z_{n+1}$ such that $d(\widehat{g}\left( \bar{a}%
_{q,m,i}\right) ,g\left( \bar{a}_{q,m,i}\right) )\leq \varpi \left(
I(g)+\varepsilon _{n}\right) $.
\end{enumerate}



The construction proceeds as follows. Fix an enumeration $\left\{ w_{n}:n\in 
\mathbb{N}\right\} $ of a dense subset of $X$. Set $X_{1}=\left\langle
w_{1}\right\rangle $. Using Condition (1) of Subsection \ref%
{Subsection:injective} for the classes $\mathcal{A}$ and $\mathcal{I}%
_{\infty }$ one can find $q_{1}\in \mathbb{N}$, a finitely generated
structure $\widehat{X}_{1}$ in $\mathcal{A}_{q_{1}}$, and morphisms $\varphi
_{1}:X_{1}\rightarrow \widehat{X}_{1}$ and $\psi _{1}:\widehat{X}%
_{1}\rightarrow X_{1}$ such that $I\left( \varphi _{1}\right) <\varepsilon
_{1}$, $I\left( \psi _{1}\right) <\varepsilon _{1}$, $I\left( \varphi
_{1}\right) <\varepsilon _{1}$, $d\left( \psi _{1}\circ \varphi
_{1},id_{X_{1}}\right) <\varepsilon _{1}$, and $d(\varphi _{1}\circ \psi
_{1},id_{\widehat{X}_{1}})<\varepsilon _{1}$.

Suppose now that, by induction hypothesis, $X_{k},\bar{b}_{k},q_{k},\widehat{%
X}_{k},\varphi _{k},Z_{k},\eta _{k},\pi _{k},\varepsilon _{k},\mathcal{F}%
_{k},\mathcal{G}_{k},j_{k-1}$ have been defined for $k\leq n$. Since
Condition (1) only involves finitely many morphisms $f$, one can find $\bar{b%
}_{n+1}\supset \bar{b}_{n}\cup \left\{ w_{n+1}\right\} $ such that $%
X_{n+1}=\left\langle \bar{b}_{n+1}\right\rangle $ satisfies (1) by
repeatedly applying the assumption that $X$ is approximately injective. One
can then find $\widehat{X}_{n+1},\varphi _{n},\psi _{n},q_{n+1}$ reasoning
as for $\widehat{X}_{1},\varphi _{1},\psi _{1}$. Let now $Z_{n+1}$ be the
approximate pushout within the class $\mathcal{A}_{q_{n+1}}$ constructed as
in Lemma \ref{Lemma:pushout2} of the maps $\eta _{n}\circ \varphi
_{n}:X_{n}\rightarrow Z_{n}$ and $\varphi _{n+1}\circ i_{n}:X_{n}\rightarrow 
\widehat{X}_{n+1}$ over $\bar{b}_{n}$ with tolerance $\varepsilon _{n}$, and
of the maps $f:E_{q,m,i}\rightarrow Z_{n}$ and $E_{q,m,i}\hookrightarrow
A_{q,m}$ (inclusion map) over $\bar{a}_{q,m,i}$ with tolerance $\varpi
\left( I(f)+\varepsilon _{n}\right) $, where $m,i\leq n$, $q\leq q_{n}$ are
such that $E_{q,m,i}\subset A_{q,m}$ and $f:E_{q,m,i}\rightarrow Z_{n}$ is a
morphism in $\mathcal{F}_{n}$. Let also $j_{n}:Z_{n}\rightarrow Z_{n+1}$ and 
$\eta _{n+1}:\widehat{X}_{n+1}\rightarrow Z_{n+1}$ be the canonical
morphisms of the approximate pushout, and observe that by definition $\eta
_{n+1}\circ \varphi _{n+1}\circ i_{n}\approx _{\bar{b}_{n},\varepsilon
_{n}}j_{n}\circ \eta _{n}\circ \varphi _{n}$. We want to define a morphism $%
\pi _{n+1}:Z_{n+1}\rightarrow \widehat{X}_{n+1}$. By inductive hypothesis we
have that $\widehat{i}_{n}\circ \pi _{n}:Z_{n}\rightarrow \widehat{X}_{n+1}$
and $\varphi _{n+1}:X_{n+1}\rightarrow \widehat{X}_{n+1}$ are morphisms such
that%
\begin{equation*}
\widehat{i}_{n}\circ \pi _{n}\circ \eta _{n}\circ \varphi _{n}=\widehat{i}%
_{n}\circ \varphi _{n}\approx _{\bar{b}_{n},\varepsilon _{n}}\varphi
_{n+1}\circ i_{n}\text{.}
\end{equation*}%
Furthermore if $m,i\leq n$ and $q\leq q_{n+1}$ are such that $%
E_{q,m,i}\subset A_{q,m}$ and $g:E_{q,m,i}\rightarrow Z_{n}$ is a morphism
in $\mathcal{G}_{n}$, then $f:=\psi _{n}\circ \pi _{n}\circ
g:E_{q,m,i}\rightarrow X_{n}$ is a morphism in $\mathcal{F}_{n}$. Therefore
by inductive hypothesis there exists a morphism $\widehat{f}%
:A_{q,m}\rightarrow X_{n+1}$ such that $\widehat{f}\approx _{\bar{a}%
_{q,m,i},\varepsilon _{n}}f$. Hence $\varphi _{n+1}\circ \widehat{f}%
:A_{q,m}\rightarrow \widehat{X}_{n+1}$ is a morphism such that%
\begin{equation*}
\varphi _{n+1}\circ \widehat{f}\approx _{\bar{a}_{q,m,i},\varepsilon
_{n}}\varphi _{n+1}\circ f=\varphi _{n+1}\circ \psi _{n}\circ \pi _{n}\circ
g=\widehat{i}_{n}\circ \pi _{n}\circ g\text{.}
\end{equation*}%
Therefore by the universal property of the approximate pushout there exists
a morphism $\pi _{n+1}:Z_{n+1}\rightarrow \widehat{X}_{n+1}$ such that $\pi
_{n+1}\circ j_{n}=\widehat{i}_{n}\circ \pi _{n}$ and $\pi _{n+1}\circ \eta
_{n+1}=id_{\widehat{X}_{n+1}}$. This concludes the recursive construction.
Granted the construction one can then define $Z$ to be the limit in $%
\mathcal{A}$ of the inductive sequence $\left( Z_{n}\right) $ with
connecting maps $j_{n}$. Let $\eta $ be the embedding of $X$ into $Z$
obtained as the limit of the sequence $\eta _{n}\circ \varphi
_{n}:X_{n}\rightarrow Z_{n}$. Finally let $\pi :Z\rightarrow X$ be the
morphism obtained as the limit of the sequence $\psi _{n}\circ \pi
_{n}:Z_{n}\rightarrow X_{n}$. It follows from the properties of the maps $%
\eta _{n},\varphi _{n},\psi _{n},\pi _{n}$ listed above that $\eta $ and $%
\pi $ are well defined and satisfy $\pi \circ \eta =id_{X}$. Furthermore the
assumption (2) in the construction guarantees that $Z$ is the Fra\"{\i}ss%
\'{e} limit of the class of finite dimensional structures in $\mathcal{A}$.
\end{proof}

Using Theorem \ref{Theorem:general-retracts} one can also prove that the
approximately injective structures in $\mathcal{A}$ are precisely the $%
\mathcal{I}$-nuclear structures as in Definition \ref{Definition:nuclear};
see the proof of Proposition \ref{Proposition:nuclear}.

One can alternatively prove Theorem \ref{Theorem:general-retracts} using the
construction from Subsection \ref{Subsection:universal-state} generalized to
the setting of stratified Fra\"{\i}ss\'{e} classes generated by injective
objects. Indeed if $X$ is a separable approximately injective structure in $%
\mathcal{A}$, there exist a morphism $\Omega _{M}^{X}:M\rightarrow X$ and an
embedding $\eta _{M}^{X}:X\rightarrow M$ such that $\Omega _{M}^{X}\circ
\eta _{M}^{X}$ is the identity of $X$. Thus $\eta _{M}^{X}\circ \Omega
_{M}^{X}:X\rightarrow X$ is a retraction of $X$ onto a substructure
isomorphic to $M$.

\section{More examples\label{Section:more-examples}}

\subsection{Exact operator spaces\label{Subsection:ospaces}}

Let $\mathcal{K}$ be the space of compact linear operators on $\ell ^{2}$.
If $X$ is a complex vector space, then the space $\mathcal{K}\otimes X$ is
naturally endowed with a $\mathcal{K}$-bimodule structure. An \emph{operator
space }is a complex vector space $X$ such that $\mathcal{K}\otimes X$ is
endowed with a norm satisfying%
\begin{equation*}
\left\Vert \sum_{i=1}^{n}\alpha _{i}^{\ast }x_{i}\beta _{i}\right\Vert \leq
\left\Vert \sum_{i=1}^{n}\alpha _{i}^{\ast }\alpha _{i}\right\Vert
\max_{1\leq i\leq n}\left\Vert x_{i}\right\Vert \left\Vert
\sum_{i=1}^{n}\beta _{i}^{\ast }\beta _{i}\right\Vert
\end{equation*}%
where $n\in \mathbb{N}$, $\alpha _{i},\beta _{i}\in \mathcal{K}$, and $%
x_{i}\in \mathcal{K}\otimes X$. A linear map $\phi :X\rightarrow Y$ between
operator spaces is \emph{completely contractive }if $id_{\mathcal{K}}\otimes
\phi $ is contractive, and completely isometric if $id_{\mathcal{K}}\otimes
\phi $ is isometric.

As before, we let $\mathcal{K}_{0}(\mathbb{Q}(i))$ be the space of finite
rank operators whose coefficients with respect to the canonical basis of $%
\ell ^{2}$ belong to the field of Gauss rationals $\mathbb{Q}(i)$. Let $%
\mathcal{L}$ be the language containing an $n$-ary function symbol $f_{%
\overline{\alpha },\overline{\beta }}$ for every $n\in \mathbb{N}$ and $n$%
-tuples $\overline{\alpha }$ and $\overline{\beta }$ in $\mathcal{K}_{0}(%
\mathbb{Q}(i\mathbb{))}$ such that $\left\Vert \sum_{i=1}^{n}\alpha
_{i}^{\ast }\alpha _{i}\right\Vert \leq 1$ and $\left\Vert
\sum_{i=1}^{n}\beta _{i}^{\ast }\beta _{i}\right\Vert \leq 1$. If $X$ is an
operator space, then one can regard $X$ as an $\mathcal{L}$-structure with
support the unit ball of $\mathcal{K}\otimes X$, where the interpretation of 
$f_{\overline{\alpha },\overline{\beta }}$ is the function%
\begin{equation*}
\left( x_{1},\ldots ,x_{n}\right) \mapsto \alpha _{1}x_{1}\beta _{1}+\cdots
+\alpha _{n}x_{n}\beta _{n}\text{.}
\end{equation*}%
It is clear that under this identification a morphism in the sense of
Subsection \ref{Subsection:morphism} is (the restriction to the unit ball
of) a completely contractive linear map, and an embedding is (the
restriction to the unit ball of) a completely isometric linear map. It is
not hard to verify that if $f:X\rightarrow Y$ is a completely contractive
linear map between operator spaces and $0\leq \delta \leq 1$, then $I(f)\leq
\delta $ if and only if $\left\Vert id_{\mathcal{K}}\otimes
f^{-1}\right\Vert \leq 1+\delta $.

Suppose that $H$ is a Hilbert space. Denote by $B(H)$ the algebra of bounded
linear operators on $H$ endowed with the operator norm. If $X$ is a linear
subspace of $B(H)$, then $X$ has a natural operator space structure obtained
by identifying $M_{n}(X)$ with a subspace of the algebra $B(H^{\oplus n})$
of bounded linear operators on the $n$-fold Hilbertian direct sum of $H$ by
itself. Conversely an operator space is linearly completely isometric to a
space of this form \cite{ruan_subspaces_1988}. We denote by $M_{d,k}(\mathbb{%
C})$ the operator space of $d\times k$ matrices, identified with the space $%
B\left( H,K\right) $ of bounded linear operators from a $k$-dimensional
Hilbert space $H$ to a $d$-dimensional Hilbert space $K$. By the
Arveson-Wittstock-Paulsen extension theorem \cite[Theorem 8.2]%
{paulsen_completely_2002} and the main result of \cite{smith_finite_2000},
the finite-dimensional \emph{injective }operator spaces are precisely the
finite $\infty $-sums of copies of $M_{d,k}(\mathbb{C})$ for $d,k\in \mathbb{%
N}$. These are also precisely the finite-dimensional \emph{ternary rings of
operators}; see \cite{kaur_local_2002}. When $k=d$ we simply write $M_{d}(%
\mathbb{C})$.

An operator space $X$ is called \emph{exact }if for any $\delta >0$ and for
any finite-dimensional subspace $E$ of $X$ there exists $n\in \mathbb{N}$
and a completely contractive linear map $f:X\rightarrow M_{n}(\mathbb{C})$
such that $\left\Vert id_{\mathcal{K}}\otimes f^{-1}\right\Vert \leq
1+\delta $. If $X$ is an $M_{q}$-space as in Subsection \ref{Subsection:Mq}
then one can canonically endow $X$ with an (exact) operator space structure $%
\mathrm{MIN}_{q}(X)$ defined by setting%
\begin{equation*}
\left\Vert x\right\Vert =\sup_{\phi }\left\Vert \left( id_{\mathcal{K}%
}\otimes \phi \right) (x)\right\Vert
\end{equation*}%
for $x\in \mathcal{K}\otimes X$, where $\phi $ ranges among all the $q$%
-contractions from $X$ to $M_{q}(\mathbb{C})$.

Let now $\mathcal{A}$ be the class of operator spaces, $\mathcal{A}_{q}$ be
the class of operator spaces of the form $\mathrm{MIN}_{q}(X)$ for some $%
M_{q}$-space $X$, $\mathcal{I}_{q}$ be the class of finite $\infty $-sums of
copies of $M_{d,k}(\mathbb{C})$ for $d,k\leq q$ (these are precisely the
finite-dimensional $q$-minimal injective operator spaces), $\mathcal{I}%
_{\infty }$ be the union of $\mathcal{I}_{q}$ for $q\in \mathbb{N}$, and $%
\mathcal{I}\subset \mathcal{I}_{\infty }$ be the class of operator spaces of
the form $M_{n}(\mathbb{C})$ for $n\in \mathbb{N}$. The small perturbation
lemma\emph{\ } shows that, by declaring a tuple in an operator space basic
if it is linearly independent, one obtains a notion of basic tuples that
satisfies the assumptions of Subsection \ref{Subsection:basic}. The
definition of exact operator spaces implies that the classes $\mathcal{A}$
and $\mathcal{I}$ satisfy Condition\ (1) of Subsection \ref%
{Subsection:injective}. Condition (2) of Subsection \ref%
{Subsection:injective} with $\varpi (\delta )=\delta $ is easily verified by
considering the composition of $f$ with the inverse map of $\phi $ (when $%
\phi $ is injective) and then normalizing. The operator spaces that are
approximately injective according to Definition \ref{Definition:approx-inj}
are precisely the nuclear operator spaces; see \cite[\S 14.6]%
{effros_operator_2000}. Similarly the operator spaces that are rigid $%
\mathcal{I}_{\infty }$-structures as in Definition \ref%
{Definition:I-structure} are precisely the rigid rectangular $\mathcal{OL}%
_{\infty ,1+}$ spaces \cite[\S 2]{junge_OL_2003}. It follows from \cite[%
Proposition 5.15]{lupini_uniqueness_2016} that the separable rigid
rectangular $\mathcal{OL}_{\infty ,1+}$ spaces are precisely the operator
spaces that can be written as limits of inductive sequences of
finite-dimensional injective operator spaces with completely isometric
connective maps. Not every nuclear operator space is rigid rectangular $%
\mathcal{OL}_{\infty ,1+}$ space. An example is the Cuntz C*-algebra $%
\mathcal{O}_{2}$ \cite[\S V.4]{davidson_c*-algebras_1996}.

One can then apply the conclusions of Section \ref{Section:general} to prove
that the class of finite-dimensional exact operator spaces form a Fra\"{\i}ss%
\'{e} class, recovering a result from \cite{lupini_uniqueness_2016}. The
corresponding limit is the Gurarij operator space $\mathbb{NG}$ introduced
in \cite{oikhberg_non-commutative_2006} and proved to be unique in \cite%
{lupini_uniqueness_2016}. Theorem \ref{Theorem:retracts} implies that a
separable exact operator space is nuclear if and only if it is completely
isometric to the range of a completely contractive projection of $\mathbb{NG}
$. This recover a result from \cite{lupini_operator_2015}. The existence of
the universal operator on $\mathbb{NG}$ described in Theorem \ref%
{Theorem:universal-cc} follows from considering the class of completely
contractive linear maps between finite-dimensional exact operator spaces, as
discussed in Subsection \ref{Subsection:universal} and Subsection \ref%
{Subsection:general-retracts}. The model-theoretic properties of $\mathbb{NG}
$ have been considered in \cite[Section 5.8]{lupini_uniqueness_2016},
building on \cite{goldbring_model-theoretic_2015}, where it is shown among
other things that $\mathbb{NG}$ is the unique separable exact existentially
closed operator space and the prime model of its first order theory. An
operator space is nuclear if and only if it is positively existentially
closed.

The noncommutative analogs of $M$-ideals in Banach spaces are the complete $%
M $-ideals in operator spaces introduced in \cite{effros_mapping_1994}. It
is proved in \cite[Proposition 4.4]{effros_mapping_1994} that a subspace $N$
of an operator space $Z$ is a complete $M$-ideal if and only if $M_{n}\left(
N\right) $ is an $M$-ideal of $M_{n}(Z)$ for every $n\in \mathbb{N}$. The
following is the natural noncommutative analog of the notion of facial
quotient from Definition \ref{Definition:facial}.

\begin{definition}
\label{Definition:complete-facial}A \emph{complete facial quotient }mapping $%
P:Z\rightarrow X$ between operator spaces is a complete quotient mapping
whose kernel if a complete $M$-ideal.
\end{definition}

It is clear that when $Z,X$ are Banach spaces endowed with the canonical
minimal operator space structure, then $P:Z\rightarrow X$ is a complete
facial quotient if and only if it is a facial quotient.

If $\boldsymbol{K}$ is a compact rectangular matrix convex set as in the
introduction---see also \cite[Section 3]{fuller_boundary_2016}---then one
can define the notion of closed rectangular matrix face of $\boldsymbol{K}$
in terms of complete facial quotients. By definition, a closed rectangular
convex subset $\boldsymbol{F}$ of $\boldsymbol{K}$ is a closed rectangular
matrix face whenever the associated restriction mapping $A_{\sigma }(%
\boldsymbol{K})\rightarrow A_{\sigma }(\boldsymbol{F})$ is a complete facial
quotient.

Recall that an operator space $X$ satisfies the operator metric
approximation property if the identity map of $X$ is the pointwise limit of
finite rank completely contractive linear maps \cite%
{effros_approximation_1990}. The following characterization of complete
facial quotients is the natural noncommutative analog of Proposition \ref%
{Proposition:characterize-face}.

\begin{proposition}
\label{Proposition:characterize-operator-face}Suppose that $X,Y$ are
operator spaces, and $P:Z\rightarrow X$ is a complete quotient map. The
following statements are equivalent:

\begin{enumerate}
\item $P$ is a complete facial quotient;

\item whenever $\varepsilon >0$, $E\subset F$ are finite-dimensional
operator spaces, $g:F\rightarrow X$ is a linear complete contraction, and $%
f:E\rightarrow Z$ is a linear complete isometry such that $P\circ f=g|_{E}$,
then there exists a linear complete contraction $\widehat{g}:F\rightarrow Z$
such that $P\circ \widehat{g}=g$ and $\left\Vert \widehat{g}%
|_{E}-f\right\Vert _{cb}\leq \varepsilon $;

\item whenever $\varepsilon >0$, $A$ is a separable operator space with the
operator metric approximation property, $E\subset A$ is a finite-dimensional
subspace, and $f:E\rightarrow Z$ and $g:A\rightarrow X$ are linear complete
contractions such that $\left\Vert P\circ f-g|_{E}\right\Vert
_{cb}<\varepsilon $, then there exists a linear complete contraction $%
\widehat{g}:A\rightarrow Z$ such that $P\circ \widehat{g}=g$ and $\left\Vert 
\widehat{g}|_{E}-f\right\Vert _{cb}<6\varepsilon $;
\end{enumerate}

If furthermore $Z$ is exact and $X$ is nuclear, then these are also
equivalent to:

\begin{enumerate}
\item[(4)] for any $\varepsilon >0$, $q\in \mathbb{N}$, finite-dimensional $%
q $-minimal operator spaces $E\subset F$, linear complete contractions $%
f:E\rightarrow Z$ and $g:F\rightarrow X$ such that $P\circ f=g|_{E}$, there
exists a linear complete contraction $\widehat{g}:F\rightarrow Z$ such that $%
P\circ \widehat{g}=g$ and $\left\Vert \widehat{g}|_{E}-f\right\Vert
_{cb}\leq \varepsilon $.
\end{enumerate}
\end{proposition}

\begin{proof}
The implication (1)$\Rightarrow $(3) can be proved as \cite[Theorem 5.2]%
{effros_mapping_1994}. The implications (3)$\Rightarrow $(2) is obvious.

(2)$\Rightarrow $(1) We denote by $N$ the kernel of $P$. Fix $n\in \mathbb{N}
$. It is enough to prove that $M_{n}\left( N\right) $ is an $M$-ideal of $%
M_{n}\left( Z\right) $. Fix $\varepsilon >0$, $%
y^{(1)}=[y_{ij}^{(1)}],y^{(2)}=[y_{ij}^{(2)}],y^{(3)}=[y_{ij}^{(3)}]\in
M_{n}\left( N\right) $ and $x=[x_{ij}]\in M_{n}\left( Z\right) $ such that $%
\max \left\{ \left\Vert y^{(1)}\right\Vert ,\left\Vert y^{(2)}\right\Vert
,\left\Vert y^{(3)}\right\Vert ,\left\Vert x\right\Vert \right\} \leq 1$. In
view of the implication (iv)$\Rightarrow $(i) in \cite[Theorem 2.2]%
{harmand_M-ideals_1993}, it is enough to prove that there exists $y\in
M_{n}\left( N\right) $ such that $\left\Vert y\right\Vert \leq 1$ and $%
\left\Vert x+y^{(\ell )}-y\right\Vert \leq 1+\varepsilon $ for $\ell \in
\left\{ 1,2,3\right\} $. Consider 
\begin{equation*}
E=span\left\{ y_{ij}^{(k)},x_{ij}:i,j\leq n\right\} \subset Z\text{.}
\end{equation*}

We denote by $e_{ij}$ the matrix units of $M_{n}(\mathbb{C})$ and by $e$ the
element $\left[ e_{ij}\right] $ of $M_{n}(M_{n}(\mathbb{C}))$. Let $F$ be
the operator space obtained from $E\oplus M_{n}(\mathbb{C})$ and the
collection of linear maps $E\oplus M_{n}(\mathbb{C})\rightarrow B(H)$ of the
form $\left( z,\alpha \right) \mapsto \varphi \left( z\right) +\psi \left(
\alpha \right) $ where $\varphi :E\rightarrow B(H)$ is completely
contractive, $\psi :M_{n}(\mathbb{C})\rightarrow B(H)$ is such that $%
\left\Vert \psi ^{(n)}\left( e\right) \right\Vert \leq 1$, and $\left\Vert
\varphi ^{(n)}\left( x+y^{(\ell )}\right) -\psi ^{(n)}\left( e\right)
\right\Vert \leq 1$. By definition we have that the norm of $\left(
x+y^{(\ell )},-e\right) $ evaluated in $M_{n}(F)$ is at most $1$. We observe
that the canonical inclusion $E\subset F$ is completely isometric. Indeed if 
$k\in \mathbb{N}$ and $z\in M_{k}\left( E\right) $ is such that $\left\Vert
z\right\Vert =1$ then there exists a completely contractive map $\varphi
:E\rightarrow B(H)$ such that $\left\Vert \varphi ^{(k)}\left( z\right)
\right\Vert =1$. Define $\psi :M_{n}(\mathbb{C})\rightarrow B(H)$, $%
e_{ij}\mapsto \varphi \left( x_{ij}\right) $. The maps $\varphi $ and $\psi $
witness that the image of $z$ inside $M_{k}\left( F\right) $ has norm $1$.
This concludes the proof that the inclusion $E\subset F$ is completely
isometric. Define now the map $g:F\rightarrow X$ by mapping $\left( z,\alpha
\right) $ to $P\left( z\right) $. Observe that $g$ is completely
contractive. Indeed if $k\in \mathbb{N}$, $z\in M_{k}(X)$, and $\alpha \in
M_{k}(\mathbb{C})$, pick a completely contractive map $\rho :X\rightarrow
B(H)$ such that $\left\Vert \left( \rho \circ P\right) \left( z\right)
\right\Vert =\left\Vert P\left( z\right) \right\Vert $. Then the maps $%
\varphi :=\left( \rho \circ P\right) |_{E}$ and $\psi =0$ witness that $%
\left\Vert P\left( z\right) \right\Vert $ is smaller than or equal to the
norm of $\left( z,\alpha \right) $ evaluated in $F$. This shows that the map 
$g$ is completely contractive. Applying our assumption to the map $g$ and
the inclusion map $f:E\rightarrow Z$ one obtains a completely contractive
map $\widehat{g}:F\rightarrow Z$ such that $P\circ \widehat{g}=g$ and $%
\left\Vert \widehat{g}|_{E}-f\right\Vert _{cb}\leq \varepsilon $. Set now $%
y_{ij}=\widehat{g}(e_{ij})$ for $i,j\leq n$ and $y=[y_{ij}]\in M_{n}\left(
Z\right) $. We have that for $\ell \in \left\{ 1,2,3\right\} $,%
\begin{eqnarray*}
\left\Vert x+y^{(\ell )}-y\right\Vert _{M_{n}(Z)} &=&\left\Vert x+y^{(\ell
)}-\widehat{g}^{(n)}(e)\right\Vert _{M_{n}(Z)} \\
&\leq &\left\Vert \widehat{g}^{(n)}(x+y^{(\ell )}-e)\right\Vert
_{M_{n}(Z)}+\varepsilon \leq \left\Vert (x+y^{(\ell )},-e)\right\Vert
_{M_{n}(F)}+\varepsilon \leq 1+\varepsilon \text{.}
\end{eqnarray*}%
This concludes the proof.

Suppose now that $X,Z$ are rigid rectangular $\mathcal{OL}_{\infty ,1+}$
spaces.

(4)$\Rightarrow $(1) As in the proof of (2)$\Rightarrow $(1), we fix $%
\varepsilon \in (0,1]$, $%
y^{(1)}=[y_{ij}^{(1)}],y^{(2)}=[y_{ij}^{(2)}],y^{(3)}=[y_{ij}^{(3)}]\in
M_{n}\left( N\right) $ and $x=[x_{ij}]\in M_{n}\left( Z\right) $ such that $%
\max \left\{ \left\Vert y^{(1)}\right\Vert ,\left\Vert y^{(2)}\right\Vert
,\left\Vert y^{(3)}\right\Vert ,\left\Vert x\right\Vert \right\} \leq 1$. We
want to prove that there exists $y\in M_{n}\left( N\right) $ such that $%
\left\Vert y\right\Vert \leq 1$ and $\left\Vert x+y^{(\ell )}-y\right\Vert
\leq 1+\varepsilon $ for $\ell \in \left\{ 1,2,3\right\} $. Define $E\subset
Z$ and $e\in M_{n}(E)$ as in the proof of (2)$\Rightarrow $(1). Fix $\delta
\in (0,\varepsilon /4]$. For $q\in \mathbb{N}$, we denote by $\mathrm{MIN}%
_{q}\left( E\right) $ the space $E$ endowed with its canonical $q$-minimal
operator space structure; see \cite[Section 2]{oikhberg_operator_2004}.
Define $B$ to be the image of $E$ under $P$, and $\iota _{B}:B\rightarrow X$
the inclusion map. Since $Z$ is exact and $X$ is nuclear, there exist $q\geq
n$ and completely contractive maps $\gamma :B\rightarrow M_{q}(\mathbb{C})$
and $\rho :M_{q}(\mathbb{C})\rightarrow Z$ such that $\left\Vert \rho \circ
\gamma -\iota _{B}\right\Vert _{cb}\leq \delta /2$ and the inclusion map $%
\iota _{E}:\mathrm{\mathrm{MIN}}_{q}\left( E\right) \rightarrow Z$ has
completely bounded norm at most $1+\delta $. Let $F$ be the $q$-minimal
operator space obtained from $E\oplus M_{n}(\mathbb{C})$ and the collection
of linear maps $E\oplus M_{n}(\mathbb{C})\rightarrow M_{q}(\mathbb{C})$, $%
\left( z,\alpha \right) \mapsto \varphi \left( z\right) +\psi \left( \alpha
\right) $ such that $\varphi :E\rightarrow M_{q}(\mathbb{C})$ is completely
contractive, $\left\Vert \psi ^{(n)}(e)\right\Vert \leq 1$, and $\left\Vert
\varphi ^{(n)}\left( x+y^{(\ell )}\right) -\psi ^{(n)}\left( e\right)
\right\Vert \leq 1+\varepsilon $. Define also $g:F\rightarrow X$ by $g\left(
z,\alpha \right) =\frac{1}{1+\delta }P\left( z\right) $ and $f:\mathrm{MIN}%
_{q}\left( E\right) \rightarrow Z$ by $f:=\frac{1}{1+\delta }\iota _{E}$.
Observe that the inclusion $\mathrm{MIN}_{q}\left( E\right) \subset F$ is
completely isometric, the maps $g:F\rightarrow X$ and $f:\mathrm{MIN}%
_{q}\left( E\right) \rightarrow Z$ are completely contractive, and $P\circ
f=g|_{\mathrm{MIN}_{q}\left( E\right) }$. Therefore by assumption there
exists a completely contractive map $\widehat{g}:F\rightarrow Z$ such that $%
P\circ \widehat{g}=g$ and $\left\Vert \widehat{g}|_{\mathrm{MIN}_{q}\left(
E\right) }-f\right\Vert _{cb}\leq \delta $. Set $y:=\widehat{g}^{(n)}\left(
e\right) $. Hence we have for $\ell \in \left\{ 1,2,3\right\} $,%
\begin{eqnarray*}
\left\Vert x+y^{(\ell )}-y\right\Vert _{M_{n}(Z)} &\leq &\left\Vert \frac{1}{%
1+\delta }(x+y^{(\ell )})-\widehat{g}^{(n)}\left( e\right) \right\Vert
_{M_{n}(Z)}+2\delta =\left\Vert f^{(n)}(x+y^{(\ell )})-\widehat{g}%
^{(n)}\left( e\right) \right\Vert _{M_{n}(Z)}+2\delta \\
&\leq &\left\Vert \widehat{g}(x+y^{(\ell )}-e)\right\Vert
_{M_{n}(Z)}+4\delta \leq \left\Vert (x+y^{(\ell )},-e)\right\Vert
_{M_{n}(F)}+4\delta \leq 1+\varepsilon \text{.}
\end{eqnarray*}%
This concludes the proof.
\end{proof}

Suppose that $H,K$ are Hilbert spaces, and $X\subset B\left( H,K\right) $
and $Z$ are operator spaces. A \emph{rectangular operator convex combination 
}as defined in \cite{fuller_boundary_2016} is an expression $\alpha
_{1}^{\ast }\phi _{1}\beta _{1}+\cdots +\alpha _{n}^{\ast }\phi _{n}\beta
_{n}$ where $\phi _{i}:Z\rightarrow B(H_{i},K_{i})$ are completely
contractive maps for some Hilbert spaces $H_{i},K_{i}$, and $\beta
_{i}:H\rightarrow H_{i}$ and $\alpha _{i}:K\rightarrow K_{i}$ are linear
maps of norm at most $1$. We say that $\alpha _{1}^{\ast }\phi _{1}\beta
_{1}+\cdots +\alpha _{n}^{\ast }\phi _{n}\beta _{n}$ is a \emph{proper }%
rectangular operator convex combination if $\alpha _{i},\beta _{i}$ are
surjective, $\alpha _{1}^{\ast }\alpha _{1}+\cdots +\alpha _{n}^{\ast
}\alpha _{n}=1$, and $\beta _{1}^{\ast }\beta _{1}+\cdots +\beta _{n}^{\ast
}\beta _{n}=1$. A proper rectangular operator convex combination $\phi
=\alpha _{1}^{\ast }\phi _{1}\beta _{1}+\cdots +\alpha _{n}^{\ast }\phi
_{n}\alpha _{n}$ is \emph{trivial }if $\alpha _{i}^{\ast }\alpha
_{i}=\lambda _{i}1$, $\beta _{i}^{\ast }\beta _{i}=\lambda _{i}$, and $%
\alpha _{i}^{\ast }\phi _{i}\beta _{i}=\lambda _{i}\phi $ for some $\lambda
_{i}\in \left[ 0,1\right] $. A completely contractive map $\phi
:Z\rightarrow X$ such that $\left\Vert \phi \right\Vert _{cb}=1$ is a \emph{%
rectangular operator extreme point} if any proper rectangular operator
convex combination $\phi =\alpha _{1}^{\ast }\phi _{1}\beta _{1}+\cdots
+\alpha _{n}^{\ast }\phi _{n}\beta _{n}$ is trivial. We observe that, if $V$
is a finite-dimensional injective operator space, then the identity map $%
V\rightarrow V$ is a rectangular operator extreme point. Indeed in this case 
$V$ is a ternary ring of operators \label{smith_injective_2000}. The
conclusion follows by passing to the linking algebra \cite{kaur_local_2002}
and then applying \cite[Corollary 1.4.3]{arveson_subalgebras_1969}.

\begin{proposition}
\label{Proposition:rectangular-extreme}Suppose that $Z$ and $X$ are rigid
rectangular $\mathcal{OL}_{\infty ,1+}$ spaces and $\phi :Z\rightarrow X$ is
a complete facial quotient. Then $\phi $ is a rectangular operator extreme
point.
\end{proposition}

\begin{proof}
Consider $X\subset B\left( H,K\right) $. Suppose that $\phi =\alpha
_{1}^{\ast }\phi _{1}\beta _{1}+\cdots +\alpha _{n}^{\ast }\phi _{n}\alpha
_{n}$ is a proper rectangular matrix convex combination as above. Fix $%
\varepsilon >0$ and a finite-dimensional injective operator space $V\subset
Z $. Since $X$ is a rigid $\mathcal{OL}_{\infty ,1+}$ space, we can find a
finite-dimensional injective operator space $W\subset X$ and a completely
contractive map $\psi :V\rightarrow W$ such that $\left\Vert \psi -\phi
\right\Vert _{cb}<\varepsilon $. Consider now the complete isometry $\eta
:V\rightarrow V\oplus ^{\infty }W$, $x\mapsto \left( x,\psi (x)\right) $ and
the completely contractive map $g:V\oplus ^{\infty }W\rightarrow X$, $\left(
z,y\right) \mapsto y$. Observe that $g\circ \eta =\psi $. Since $\phi $ is a
complete facial quotient, there exists a completely contractive map $%
\widehat{g}:V\oplus ^{\infty }W\rightarrow Z$ such that $\phi \circ \widehat{%
g}=g$ and $\left\Vert \widehat{g}\circ \eta -\iota \right\Vert
_{cb}<6\varepsilon $, where $\iota :V\rightarrow Z$ is the inclusion map. We
have that 
\begin{equation*}
g=\phi \circ \widehat{g}=\alpha _{1}^{\ast }\left( \phi _{1}\circ \widehat{g}%
\right) \beta _{1}+\cdots +\alpha _{n}^{\ast }\left( \phi _{n}\circ \widehat{%
g}\right) \alpha _{n}\text{.}
\end{equation*}%
Since $g$ is a rectangular operator extreme point, we can conclude that
there exist $\lambda _{1},\ldots ,\lambda _{n}\in \left[ 0,1\right] $ such
that $\alpha _{i}^{\ast }\alpha _{i}=\lambda _{i}1$, $\beta _{i}^{\ast
}\beta _{i}=\lambda _{i}$, and $\alpha _{i}^{\ast }\left( \phi _{i}\circ 
\widehat{g}\right) \beta _{i}=\lambda _{i}\left( \phi \circ \widehat{g}%
\right) $ for $i=1,2,\ldots ,n$. Since $\left\Vert \widehat{g}\circ \eta
-\iota \right\Vert _{cb}<\varepsilon $ we conclude that%
\begin{equation*}
\left\Vert \alpha _{i}^{\ast }\phi _{i}|_{V}\beta _{i}-\lambda _{i}\phi
|_{V}\right\Vert <12\varepsilon
\end{equation*}%
for $i=1,2,\ldots ,n$. Since this holds for any $\varepsilon >0$ and any
finite-dimensional injective operator space $V\subset Z$, it follows by
compactness and the fact that $Z$ is a rigid rectangular $\mathcal{OL}%
_{\infty ,1+}$ space that the proper rectangular operator convex combination 
$\phi =\alpha _{1}^{\ast }\phi _{1}\beta _{1}+\cdots +\alpha _{n}^{\ast
}\phi _{n}\alpha _{n}$ is trivial. This concludes the proof that $\phi $ is
a rectangular operator extreme point.
\end{proof}

Fix now a separable nuclear operator space $X$, and consider the generic
completely contractive map $\Omega _{\mathbb{\mathbb{NG}}}^{X}:\mathbb{%
\mathbb{N}G}\rightarrow X$ as in Subsection \ref{Subsection:universal}. Then
the characterization of such a map from Subsection \ref{Subsection:universal}
together with Proposition \ref{Proposition:characterize-operator-face} shows
that $\Omega _{\mathbb{NG}}^{X}$ is a complete facial quotient in the sense
of Definition \ref{Definition:complete-facial}. Furthermore if $X$ is a
rigid rectangular $\mathcal{OL}_{\infty ,1+}$ space (and particularly when $%
X=M_{n,m}(\mathbb{C})$ for some $n,m\in \mathbb{C}$) one has that $\Omega _{%
\mathbb{NG}}^{X}$ is a rectangular operator extreme point. These
observations together with the general results on universal morphisms from
Section \ref{Section:operators} and Section \ref{Section:states} conclude
the proof of Theorem \ref{Theorem:NG}, Theorem \ref{Theorem:universal-cc},
and Theorem \ref{Theorem:universal-projection-NG}.

\subsection{Exact operator systems\label{Subsection:osystems}}

An operator system can be defined as an operator space $X$ with a
distinguished element $1$ (its \emph{unit}) such that there exists a
completely isometric linear map from $X$ to the space $B(H)$ of bounded
linear operators on a Hilbert space that moreover maps the distinguished
element $1$ of $X$ to the identity operator of $H$. Any operator system is
endowed with an involution $x\mapsto x^{\ast }$ coming from the inclusion $%
X\subset B(H)$. A linear map between operator systems is \emph{unital} if it
maps the unit to the unit. A unital linear completely contractive maps
between operator systems is automatically \emph{self-adjoint}, that is it
commutes with taking adjoints.

An operator system can be regarded as a structure in the language of
operator spaces with the addition of a constant symbol for the unit and a
unary function symbol for the involution. The results from \cite%
{blecher_metric_2011} show that operator systems form an axiomatizable class
in this language. An earlier characterization of operator systems due to
Choi and Effros involves the unit and the matrix positive cones \cite%
{choi_injectivity_1977}. In this setting morphisms\emph{\ }will be unital
completely contractive linear maps. Similarly embeddings will be unital
completely isometric linear maps.

An operator system $X$ is called \emph{exact} if it is exact as an operator
space or, equivalently, for every $\delta >0$ and finite-dimensional
subspace $E$ of $X$, there exists $n\in \mathbb{N}$ and a \emph{unital}
completely contractive map $f:X\rightarrow M_{n}$ such that $\left\Vert id_{%
\mathcal{K}}\otimes f^{-1}\right\Vert \leq 1+\delta $; see \cite[Section 5]%
{kavruk_quotients_2013}. Any $M_{q}$-system $X$ has a canonical (exact)
operator system structure $\mathrm{OM\mathrm{IN}}_{q}(X)$ obtained by setting%
\begin{equation*}
\left\Vert x\right\Vert =\sup_{\phi }\left\Vert \left( id_{\mathcal{K}%
}\otimes \phi \right) (x)\right\Vert
\end{equation*}%
for $x\in \mathcal{K}\otimes X$, where $\phi $ ranges among all the unital $%
q $-contractive linear maps from $X$ to $M_{q}$; see \cite{xhabli_super_2012}%
. The operator systems of the form $\mathrm{OMIN}_{q}(X)$ are called $q$%
-minimal in \cite{xhabli_super_2012}. By the Arveson extension theorem \cite[%
Theorem 7.5]{paulsen_completely_2002} the finite-dimensional injective
operator systems are the finite $\infty $-sums of copies of $M_{n}(\mathbb{C}%
)$ for $n\in \mathbb{N}$. These are also precisely the finite-dimensional
C*-algebras.

Let now $\mathcal{A}$ be the class of exact operator systems and, for every $%
q\in \mathbb{N}$, $\mathcal{A}_{q}$ be the class of $q$-minimal operator
systems and $\mathcal{I}_{q}$ the class of finite $\infty $-sums of copies
of $M_{d}(\mathbb{C})$ for $d\leq q$ (these are precisely the
finite-dimensional $q$-minimal injective operator systems). The class $%
\mathcal{I}_{\infty }$ is the union of $\mathcal{I}_{q}$ for $q\in \mathbb{N}
$. Finally we let $\mathcal{I}$ be the class of operator systems of the form 
$M_{n}(\mathbb{C})$ for some $n\in \mathbb{N}$. One can verify as for
operator spaces that the assumptions of Section \ref{Section:general} apply.
The main difference lies in verifying Condition (2) of Subsection \ref%
{Subsection:injective} As for $M_{q}$-systems, here one needs to approximate
an approximately completely contractive self-adjoint unital linear map by a
completely contractive unital linear map. This can be done using the
following lemma. The\emph{\ completely bounded norm }$\left\Vert
f\right\Vert _{cb}$ of a linear map between operator spaces $f:X\rightarrow
Y $ is the norm of $id_{\mathcal{K}}\otimes f:\mathcal{K}\otimes
X\rightarrow \mathcal{K}\otimes Y$. Recall that a unital linear map between
operator systems is completely contractive if and only if it is completely
positive, i.e.\ for every $n\in \mathbb{N}$ and positive element $x$ of $%
\mathcal{K}\otimes X$ the image $\left( id_{\mathcal{K}}\otimes f\right) (x)$
is positive.

\begin{lemma}
\label{Lemma:perturb-ucp}Suppose that $V,W$ are operator systems, and $%
f:V\rightarrow W$ is a self-adjoint linear map such that $\left\Vert
f\right\Vert _{cb}\leq 1+\delta $. If $W$ is injective and $f$ is unital,
then there exists a unital completely positive linear map $g:V\rightarrow W$
such that $\left\Vert g-f\right\Vert _{cb}\leq 2\delta $. If $W$ is an
arbitrary operator system, $V$ has finite dimension $n$, and $f$ is either
unital or completely contractive, then there exists a unital completely
positive linear map $g:V\rightarrow W$ such that $\left\Vert g-f\right\Vert
_{cb}\leq 2n\delta $.
\end{lemma}

\begin{proof}
If $W$ is injective, then we can assume without loss of generality that $%
W=B(H)$. In this case the first assertion follows from \cite[Corollary B.9]%
{brown_c*-algebras_2008}. Suppose now that $W\subset B(H)$ is an arbitrary
operator system and $f$ is unital. The proof in the case when $f$ is
completely contractive is analogous. By Wittstock's decomposition theorem 
\cite[Theorem 8.5]{paulsen_completely_2002} there exist completely positive
maps $\phi _{1},\phi _{2}:V\rightarrow B(H)$ such that $f=\phi _{1}-\phi
_{2} $ and $\left\Vert \phi _{1}+\phi _{2}\right\Vert _{cb}\leq \left\Vert
f\right\Vert _{cb}\leq 1+\delta $. In particular by \cite[Proposition 3.2]%
{paulsen_completely_2002} we have that%
\begin{equation*}
\left\Vert \phi _{1}(1)\right\Vert \leq \left\Vert \phi _{1}(1)+\phi
_{2}(1)\right\Vert \leq \left\Vert \phi _{1}+\phi _{2}\right\Vert _{cb}\leq
1+\delta \text{.}
\end{equation*}%
Since $\phi _{1}(1)-\phi _{2}(1)$ is the identity operator on $H$, this
implies that $\left\Vert \phi _{2}\right\Vert _{cb}=\left\Vert \phi
_{2}(1)\right\Vert \leq \delta $.

By \cite[Lemma 2.4]{effros_lifting_1985} there exists a positive linear
functional $\theta $ on $V$, which can we regard as a function $\theta
:V\rightarrow W$, such that $\theta -\phi _{2}$ is completely positive and $%
\left\Vert \theta \right\Vert \leq n\delta $. Consider now the completely
positive map $g_{0}=f+\theta =\phi _{1}+\left( \theta -\phi _{2}\right) $
and observe that $\left\Vert g_{0}-f\right\Vert \leq n\delta $. Set $%
g(x):=g_{0}(x)+\tau (x)(g_{0}(1)-1)$, where $\tau $ is a state on $V_{0}$.
Then $g$ is a unital completely positive map such that $\left\Vert
g-f\right\Vert _{cb}\leq 2n\delta $.
\end{proof}

Lemma \ref{Lemma:perturb-ucp} shows that Condition (2) of Subsection \ref%
{Subsection:injective} holds for operator systems with $\varpi (\delta
)=2\delta $. As basic tuples one can consider in this context linearly
independent tuples whose first element is the unit. To verify that the
assumptions of Subsection \ref{Subsection:basic} are satisfied one can use
Lemma \ref{Lemma:perturb-ucp} together with the small perturbation argument 
\cite[Lemma 2.13.2]{pisier_introduction_2003}. An operator system is
approximately injective according to Definition \ref{Definition:approx-inj}
if and only if it is nuclear. A (rigid) $\mathcal{I}_{\infty }$-structure as
in Definition \ref{Definition:I-structure} is an operator system which is a
(rigid) $\mathcal{OL}_{\infty ,1+}$ space in the sense of \cite%
{junge_OL_2003}. This follows from Lemma \ref{Lemma:perturb-ucp} together
with the following lemma, which can be proved as \cite[Lemma 2.6]%
{eckhardt_perturbations_2010}.

\begin{lemma}
\label{Lemma:unitize}Suppose that $X,Y$ are operator systems and $\phi
:X\rightarrow Y$ a completely positive map such that $\left\Vert \phi
\right\Vert _{cb}\leq 1+\delta <2$. Consider a state $\tau $ of $X$. If $%
\psi :X\rightarrow Y$ is defined by $\psi (x)=\phi (x)+\tau (x)\left( 1-\phi
(1)\right) $, then $\psi $ is an injective unital completely positive map
such that $\left\Vert \psi ^{-1}\right\Vert \leq \left( 1+\delta \right)
\left( 1-\delta \right) ^{-1}$.
\end{lemma}

A separable operator system is a rigid $\mathcal{OL}_{\infty ,1+}$ space if
and only if it is unitally isometrically isomorphic to the limit of an
inductive sequence of finite-dimensional C*-algebras with unital completely
isometric connective maps. This is a consequence of the following lemma,
which can be proved similarly as \cite[Lemma 7.1]{kerr_gromov-hausdorff_2009}
using \cite[Proposition 4.2.8]{blackadar_generalized_1997}.

\begin{lemma}
\label{Lemma:perturbation}Suppose that $B$ is a finite-dimensional
C*-algebra and $\varepsilon >0$. Then there exists $\delta =\delta
_{p}\left( \varepsilon ,B\right) $ such that for any finite-dimensional
C*-algebra $A$ and injective linear map $\phi :B\rightarrow A$ such that $%
\left\Vert \phi \right\Vert \leq 1+\delta $, $\left\Vert \phi
^{-1}\right\Vert \leq 1+\delta $, and $\left\Vert \phi (1)-1\right\Vert \leq
\delta $, there exists a complete order embedding $\psi :B\rightarrow A$
such that $\left\Vert \psi -\phi \right\Vert _{cb}\leq \varepsilon $.
\end{lemma}

It follows from the discussion above that finite-dimensional operator
systems form a Fra\"{\i}ss\'{e} class. We will call the corresponding limit $%
A(\mathbb{NG})$ the \emph{noncommutative Poulsen system}. The matrix state
space $\mathbb{\mathbb{NG}}$ of the operator system $A(\mathbb{NG})$ will be
called the \emph{noncommutative Poulsen simplex}. Since $A(\mathbb{\mathbb{NP%
}})$ is a separable nuclear operator system---and, in fact, a rigid $%
\mathcal{OL}_{\infty ,1+}$-space---$\mathbb{\mathbb{\mathbb{NP}}}$ is a
metrizable noncommutative Choquet simplex in the sense of \cite%
{davidson_noncommutative_2016}. The noncommutative Poulsen simplex satisfies
the natural noncommutative analog of the defining property of the Poulsen
simplex: the set of matrix extreme points of $\mathbb{NP}$ is dense in $%
\mathbb{NP}$. It is furthermore proved in \cite{davidson_noncommutative_2016}
that $\mathbb{NP}$ is the unique metrizable noncommutative Choquet simplex
with such a property.

The operator system $A(\mathbb{\mathbb{NP}})$ associated with the
noncommutative Poulsen simplex is the first example of a separable
exact---in fact, nuclear---operator system that contains a unital completely
isometric copy of any other separable exact operator system. It is
furthermore proved in \cite{davidson_noncommutative_2016} that $A(\mathbb{%
\mathbb{NP}})$ is the unique separable nuclear operator system that is
universal in the sense of Kirchberg and Wassermann \cite%
{kirchberg_c*-algebras_1998}. The model-theoretic properties of the
noncommutative Poulsen system $A(\mathbb{\mathbb{NP}})$ have been considered
in \cite{goldbring_model-theoretic_2015}, where is it shown that $A(\mathbb{%
\mathbb{NP}})$ is the unique separable existentially closed operator system,
and the unique prime model of its first order theory. Furthermore, an
operator system is nuclear if and only if it is positively existentially
closed.

In analogy with the case of function systems, we consider the following
notion of face for compact matrix convex sets.

\begin{definition}
\label{Definition:complete-facial-system}A \emph{unital complete facial
quotient }mapping $P:Z\rightarrow X$ between operator systems is a unital
complete quotient mapping whose kernel if a complete $M$-ideal.
\end{definition}

Suppose that $\boldsymbol{K}$ is a compact matrix convex set. The notion of
closed matrix face of $\boldsymbol{K}$ can be defined in terms of unital
facial quotients. By definition, a compact matrix convex subset $\boldsymbol{%
F}$ of $\boldsymbol{K}$ is a closed matrix face if the induced map $A(%
\boldsymbol{K})\rightarrow A(\boldsymbol{F})$ is a unital complete facial
quotient in the sense of Definition \ref{Definition:complete-facial-system}.

We say that an operator system $A$ satisfies the operator metric
approximation property if it satisfies such a property as an operator space.
It follows from Lemma \ref{Lemma:perturbation} that this is equivalent to
the assertion that the identity map of $A$ is the pointwise limit of finite
rank unital completely positive maps. The similar proof as Proposition \ref%
{Proposition:characterize-operator-face} gives the following result.

\begin{proposition}
\label{Proposition:characterize-operator-face-system}Suppose that $X,Y$ are
operator systems, and $P:Z\rightarrow X$ is a unital complete quotient
mapping. The following statements are equivalent:

\begin{enumerate}
\item $P$ is a unital complete facial quotient;

\item whenever $\varepsilon >0$, $E\subset F$ are finite-dimensional
operator systems, $g:F\rightarrow X$ is a unital completely positive map,
and $f:E\rightarrow Z$ is a unital complete isometry such that $P\circ
f=g|_{E}$, then there exists a unital completely positive map $\widehat{g}%
:F\rightarrow Z$ such that $P\circ \widehat{g}=g$ and $\left\Vert \widehat{g}%
|_{E}-f\right\Vert _{cb}\leq \varepsilon $;

\item whenever $\varepsilon >0$, $A$ is a separable operator systems with
the operator metric approximation property, $E\subset A$ is a
finite-dimensional subsystem, and $f:E\rightarrow Z$ and $g:A\rightarrow X$
are unital completely positive maps such that $\left\Vert P\circ
f-g|_{E}\right\Vert _{cb}<\varepsilon $, then there exists a unital
completely positive map $\widehat{g}:A\rightarrow Z$ such that $P\circ 
\widehat{g}=g$ and $\left\Vert \widehat{g}|_{E}-f\right\Vert
_{cb}<3\varepsilon $.
\end{enumerate}

If furthermore $Z$ is exact and $X$ is nuclear, then these are also
equivalent to:

\begin{enumerate}
\item[(4)] for any $\varepsilon >0$, $q\in \mathbb{N}$, finite-dimensional $%
q $-minimal operator systems $E\subset F$, and unital completely positive
maps $f:E\rightarrow Z$ and $g:F\rightarrow X$ such that $P\circ f=g|_{E}$,
there exists a unital completely positive map $\widehat{g}:F\rightarrow Z$
such that $P\circ \widehat{g}=g$ and $\left\Vert \widehat{g}%
|_{E}-f\right\Vert _{cb}<\varepsilon $.
\end{enumerate}
\end{proposition}

The implication (1)$\Rightarrow $(3) of Proposition \ref%
{Proposition:characterize-operator-face-system} can be proved similarly as 
\cite[Theorem 5.2]{effros_mapping_1994}, where one starts from \cite[%
Proposition 2.2]{choi_completely_1976} instead of \cite[Lemma 5.1]%
{effros_mapping_1994}. For the implication (4)$\Rightarrow $(1) one can use
Lemma \ref{Lemma:perturb-ucp}.

Suppose that $X\subset B(H)$ and $Z$ are operator systems. An\emph{\
operator convex combination} as defined in \cite{fuller_boundary_2016} is an
expression $\alpha _{1}^{\ast }\phi _{1}\alpha _{1}+\cdots +\alpha
_{n}^{\ast }\phi _{n}\alpha _{n}$ where $\phi _{i}:Z\rightarrow B(H_{i})$
are unital completely positive maps, and $\alpha _{i}:K\rightarrow K_{i}$
are linear maps of norm at most $1$. We say that $\alpha _{1}^{\ast }\phi
_{1}\alpha _{1}+\cdots +\alpha _{n}^{\ast }\phi _{n}\alpha _{n}$ is a \emph{%
proper} operator convex combination if $\alpha _{i}$ are surjective and $%
\alpha _{1}^{\ast }\alpha _{1}+\cdots +\alpha _{n}^{\ast }\alpha _{n}=1$. It
is clear that when $H$ is finite-dimensional the notion of proper operator
convex combination coincides with the notion of matrix convex combination
considered in \cite{webster_krein-milman_1999,farenick_extremal_2000}. A
proper rectangular operator convex combination $\phi =\alpha _{1}^{\ast
}\phi _{1}\beta _{1}+\cdots +\alpha _{n}^{\ast }\phi _{n}\alpha _{n}$ is 
\emph{trivial }if $\alpha _{i}^{\ast }\alpha _{i}=\lambda _{i}1$, $\beta
_{i}^{\ast }\beta _{i}=\lambda _{i}1$, and $\alpha _{i}^{\ast }\phi
_{i}\beta _{i}=\lambda _{i}\phi $ for some $\lambda _{i}\in \left[ 0,1\right]
$. Then a unital completely positive map $\phi :Z\rightarrow X$ is an\emph{\
operator extreme point} if any proper operator convex combination $\phi
=\alpha _{1}^{\ast }\phi _{1}\beta _{1}+\cdots +\alpha _{n}^{\ast }\phi
_{n}\beta _{n}$ is trivial. Theorem A of \cite{farenick_extremal_2000} shows
that when $H$ is finite-dimensional, an operator extreme point is the same
as a matrix extreme point as defined in \cite%
{webster_krein-milman_1999,farenick_extremal_2000}. We observe that, if $A$
is a unital C*-algebra, then the identity map $A\rightarrow A$ is an
operator extreme point by \cite[Corollary 1.4.3]{arveson_subalgebras_1969}.

The same proof as Proposition \ref{Proposition:rectangular-extreme} gives
the following result.

\begin{proposition}
\label{Proposition:matrix-extreme}Suppose that $Z$ and $X$ are rigid $%
\mathcal{OL}_{\infty ,1+}$ systems, and $\phi :Z\rightarrow X$ is a unital
complete facial quotient. Then $\phi $ is an operator extreme point.
\end{proposition}

One can now deduce Theorem \ref{Theorem:NG} (1)--(3), Theorem \ref%
{Theorem:universal-ucp}, and Theorem \ref{Theorem:universal-projection-NP}
from Proposition \ref{Proposition:characterize-operator-face-system},
Proposition \ref{Proposition:matrix-extreme} and the general results from
Section \ref{Section:general}, Section \ref{Section:operators}, and Section %
\ref{Section:states}. Indeed, suppose that $\boldsymbol{F}$ is a metrizable
noncommutative Choquet simplex. Recall that this means that $\boldsymbol{F}$
is the matrix state space of a separable nuclear operator system $A(%
\boldsymbol{F})$. Consider the generic completely positive map $\Omega
_{A\left( \mathbb{NG}\right) }^{A(\boldsymbol{F})}:A\left( \mathbb{NG}%
\right) \rightarrow A(\boldsymbol{F})$ as constructed in Section \ref%
{Section:states}. The characterization of $\Omega _{A\left( \mathbb{NG}%
\right) }^{A(\boldsymbol{F})}$ from Section \ref{Section:states} together
with the equivalence of (1) and (4) in Proposition \ref%
{Proposition:characterize-operator-face-system} show that $\Omega _{A\left( 
\mathbb{NG}\right) }^{A(\boldsymbol{F})}$ is a unital complete facial
quotient mapping, and hence the dual map induces an inclusion of $%
\boldsymbol{F}$ inside the noncommutative Poulsen simplex as a
noncommutative face. The other assertions are proved analogously.

It remains to prove that the canonical action of the group $\mathrm{Aut}(%
\mathbb{\mathbb{NP}})$ of matrix affine homeomorphisms of $\mathbb{NP}$%
---which can be identified with the space of surjective unital complete
isometries of $A(\mathbb{\mathbb{NP}})$ endowed with the topology of
pointwise convergence---on the space $S_{1}(A(\mathbb{NP}))$ of states of
the noncommutative Poulsen system is minimal. In view of Corollary \ref%
{Corollary:minimal}, this is a consequence of the following lemma.

\begin{lemma}
\label{Lemma:minimal}Fix $d\in \mathbb{N}$ and $\varepsilon >0$. There
exists $m\in \mathbb{N}$ such that for any $s\in S_{1}(M_{d}(\mathbb{C}))$
and $t\in S_{1}(M_{m}(\mathbb{C}))$ there exists a complete order embedding $%
\phi :M_{d}(\mathbb{C})\rightarrow M_{m}(\mathbb{C})$ such that $\left\Vert
s\circ \phi -t\right\Vert \leq \varepsilon $.
\end{lemma}

\begin{proof}
Pick $\ell \in \mathbb{N}$ such that $1/\ell \leq \varepsilon /16$. Let $%
\mathcal{P}$ be a finite set of positive elements of $M_{d}(\mathbb{C})$ of
norm at most $1$ with the property that for any positive element $x$ of $%
M_{d}(\mathbb{C})$ of norm at most $1$ there exists $x_{0}\in \mathcal{P}$
such that $\left\Vert x-x_{0}\right\Vert <\eta $. Consider $k\in \mathbb{N}$
such that $k>\ell \left\vert \mathcal{P}\right\vert $ and set $m:=kd$.
Suppose that $s\in S_{1}\left( M_{kd}(\mathbb{C})\right) $ and $t\in
S_{1}(M_{d}(\mathbb{C}))$. Then there exists a positive matrix $a\in M_{kd}(%
\mathbb{C})$ such that $\mathrm{Tr}_{kd}\left( a\right) =1$ and $s(x)=%
\mathrm{Tr}_{kd}(ax)$ for every $x\in M_{kd}(\mathbb{C})$, where $\mathrm{Tr}%
_{kd}$ denotes the usual trace on $M_{kd}(\mathbb{C})$. We regard $M_{kd}(%
\mathbb{C})$ as the space of bounded linear operators on the space $\mathbb{C%
}^{kd}$ with canonical basis $\left( e_{1},\ldots ,e_{kd}\right) $. For $%
1\leq i\leq k$, let $p_{i}$ be the orthogonal projection on the span of $%
\left\{ e_{\left( i-1\right) d+1},\ldots ,e_{id}\right\} $, and $%
a_{i}=p_{i}ap_{i}$. Suppose that $b$ is an element of $\mathcal{P}$, and
define $b=p_{1}bp_{1}+\cdots +p_{k}bp_{k}$. Then%
\begin{equation*}
s(B)=\mathrm{Tr}_{m}(ab)=\sum_{i=1}^{\ell }\mathrm{\mathrm{Tr}}%
_{d}(a_{i}b)\leq 1\text{,}
\end{equation*}%
where $\mathrm{Tr}_{d}$ denotes the canonical trace on $M_{d}(\mathbb{C})$.
Therefore the set of $i\in \left\{ 1,2,\ldots ,k\right\} $ such that $%
\mathrm{Tr}_{d}(a_{i}b)\geq 1/\ell $ contains at most $\ell $ elements.
Therefore the set of $i\in \left\{ 1,2,\ldots ,k\right\} $ such that $%
\mathrm{Tr}_{d}(a_{i}b)\geq 1/\ell $ for every $b\in \mathcal{P}$ contains
at most $\ell \left\vert \mathcal{P}\right\vert $ elements. Therefore there
exists $i\in \left\{ 1,2,\ldots ,k\right\} $ such that $\mathrm{Tr}%
_{d}(a_{i}b)\leq 1/\ell $ for every $b\in \mathcal{P}$. Without loss of
generality we can assume that $i=1$. We also have $\mathrm{Tr}%
_{d}(a_{1}b)\leq 8/\ell $ for every $b\in M_{d}(\mathbb{C})$ of norm at most 
$1$, since $\mathcal{P}$ is $1/\ell $-dense in the set of positive elements
of $M_{d}(\mathbb{C})$ of norm at most $1$. Therefore $\left\Vert
a_{1}\right\Vert \leq 8/\ell $ and $\left\vert \sum_{i=2}^{\ell }\mathrm{Tr}%
_{d}(a_{i})-1\right\vert \leq 8/\ell $. Define now the complete order
embedding $\phi :M_{d}(\mathbb{C})\rightarrow M_{kd}(\mathbb{C})$ by%
\begin{equation*}
x\mapsto 
\begin{bmatrix}
x & 0 \\ 
0 & t(x)I_{\left( k-1\right) d}%
\end{bmatrix}%
\end{equation*}%
where $I_{\left( k-1\right) d}$ is the identity $\left( k-1\right) d\times
\left( k-1\right) d$ matrix. Observe that, for every $x\in M_{d}(\mathbb{C})$
of norm at most $1$,%
\begin{equation*}
\left\vert s\left( \phi (x)\right) -t(x)\right\vert =\left\vert \mathrm{Tr}%
\left( ax\right) +t(x)\sum_{i=2}^{\ell }\mathrm{Tr}_{d}\left( a_{i}\right)
-t(x)\right\vert \leq \left\vert \mathrm{Tr}\left( ax\right) \right\vert
+\left\vert \sum_{i=2}^{\ell }\mathrm{Tr}_{d}\left( a_{i}\right)
-1\right\vert \leq 16/\ell \text{.}
\end{equation*}%
This concludes the proof.
\end{proof}

\providecommand{\bysame}{\leavevmode\hbox to3em{\hrulefill}\thinspace} %
\providecommand{\MR}{\relax\ifhmode\unskip\space\fi MR } 
\providecommand{\MRhref}[2]{  \href{http://www.ams.org/mathscinet-getitem?mr=#1}{#2}
} \providecommand{\href}[2]{#2}


\end{document}